\documentclass[reqno,centertags,11pt]{amsart}

\usepackage{amssymb,amsmath,amsfonts,amssymb}
\usepackage{hyperref}
\usepackage{enumerate}

-\marginparwidth \oddsidemargin -9mm \evensidemargin \oddsidemargin
\usepackage{latexsym}
\advance\hoffset by 5mm

\def\@abssec#1{\vspace{.05in}\footnotesize \parindent .2in
{\bf #1. }\ignorespaces}

\newtheorem{theorem}{Theorem}[section]

\newtheorem{lemma}[theorem]{Lemma}
\newtheorem{proposition}[theorem]{Proposition}

\newtheorem{remark}[theorem]{Remark}

\newcounter{hypo}

\DeclareMathOperator{\divg}{div}

\newcommand{\R}{\ensuremath{\mathbb{R}}}

\newcommand{\N}{\ensuremath{\mathbb{N}}}

\newcommand{\Id}{\ensuremath{\mathrm{Id}}}

\newcommand{\mh}{\ensuremath{\mathrm{h}}}
\newcommand{\mtd}{\ensuremath{\mathrm{2d}}}

\newcommand{\beq}[1]{\begin{equation}\label{#1}}
\newcommand{\eeq}{\end{equation}}
\newcommand{\beqs}{\begin{equation*}}
\newcommand{\eeqs}{\end{equation*}}
\newcommand{\set}[1]{\left\{#1\right\}}

\newcommand{\no}[1]{\left\|#1\right\|}
\newcommand{\dd}{\mathrm{d}}
\renewcommand{\d}{\partial}

\allowdisplaybreaks \numberwithin{equation}{section}

\begin{document}

\title[Between HNS and INS system: the issue of stability]{Between homogeneous and inhomogeneous Navier-Stokes systems: the issue of stability}
\author{Piotr B. Mucha}
\address{Institute of Applied Mathematics and Mechanics, University of Warsaw, ul. Banacha 2 Warszawa, Poland}
\email{p.mucha@mimuw.edu.pl}
\author{Liutang Xue}
\address{Laboratory of Mathematics and Complex Systems (MOE), School of Mathematical Sciences, Beijing Normal University, Beijing 100875, P.R. China}
\email{xuelt@bnu.edu.cn}
\author{Xiaoxin Zheng}
\address{School of Mathematics and Systems Science, Beihang University, Beijing 100191, P.R. China}
\email{xiaoxinzheng@buaa.edu.cn}
\date{\today}
\maketitle

\begin{abstract}
We construct large velocity vector solutions to the three dimensional inhomogeneous Navier-Stokes system. The result is proved via the stability of two dimensional solutions with constant density,
under the assumption that initial density is point-wisely close to a constant. Key elements of our approach are estimates in the maximal regularity regime and the Lagrangian coordinates.
Considerations are done in the whole $\R^3$.
\end{abstract}

\section{Introduction}

The Navier-Stokes equations take an extraordinary position in Partial Differential Equations. Thanks to the still open Millennium Problem \cite{Mil}, asking
if weak solutions in the three space dimensional case are indeed classical, provided the given data are smooth, the system is a symbol of challenging problems in nowadays mathematics.
From the physical viewpoint the Navier-Stokes equations describe the motion of incompressible flows of viscous Newtonian fluid with constant density. Here we want to study the
connection of this classical system with its modification allowing variable density (\cite{PLL}). Namely, we consider the three dimensional (3D) inhomogeneous Navier-Stokes equations
\begin{equation}\label{INS}
\textrm{(INS)}\;\begin{cases}
  \partial_t \rho + v\cdot\nabla \rho =0, &\quad \textrm{in}\;\;\R^3 \times \R^+,\\
  \rho \partial_t v + \rho v\cdot\nabla v - \nu\Delta v + \nabla p =0, &\quad \textrm{in}\;\;\R^3 \times \R^+,\\
  \divg v=0, &\quad\textrm{in}\;\;\R^3 \times \R^+, \\
  \rho|_{t=0}(x)=\rho_0(x), \quad v|_{t=0}(x)= v_0(x), &\quad \textrm{on}\;\;\R^3,
\end{cases}
\end{equation}
where $x=(x_1,x_2,x_3)\in \R^3$, $\nu> 0$ is the kinematic viscosity, $\rho$ is the scalar density field, and $v$ is the incompressible velocity vector field. For simplicity, we assume $\nu\equiv 1$.

Mathematical properties of (\ref{INS}) are almost the same as those of the classical 3D Navier-Stokes equations (with the constant density $\rho$) \cite{K,LS}.
The main difference is found in the issue related to the density. Questions concerned with the low regularity of initial density or the possibility of vacuum states are the subjects of current studies of (IHS) system (\ref{INS}) (see e.g. \cite{DanM12,DanM13,DanM17,HPZ,Li,PZZ}).

Our goal here is slightly different: we want to find solutions to the (INS) system (\ref{INS}) with large velocity vector field like in \cite{ChPZ,PZ15}.
The plan is to consider the stability issue of the equations \eqref{INS} around the 2D homogeneous Navier-Stokes equations with constant-valued density $1$.
More precisely, let $v^\mtd=(v^\mtd_1,v^\mtd_2, v^\mtd_3)$ be a three-component two dimensional vector field which solves the 2D homogeneous Navier-Stokes equations
\begin{equation}\label{HNS}
\textrm{(HNS)}\;\begin{cases}
  \partial_t v^\mtd + v^\mtd_\mh \cdot\nabla_\mh v^\mtd -  \Delta_\mh v^\mtd + \nabla p^\mtd=0, \\
  \nabla_\mh \cdot v^\mtd_\mh =0, \\
  v^\mtd |_{t=0}(x_{\mathrm{h}})=v^\mtd_0(x_{\mathrm{h}}),
\end{cases}
\end{equation}
where $x_{\mathrm{h}}=(x_1,x_2)\in\R^2$, $ t\in \R^+$, $\nabla_\mh:= (\partial_{x_1},\partial_{x_2})$, $\Delta_\mh:= \partial_{x_1}^2 + \partial_{x_2}^2$, then for the solution $(v, \rho,p)$ of the (INS) system \eqref{INS} and the solution $(v^\mtd,p^\mtd)$ of 2D (HNS) equations \eqref{HNS}, denoting by
\begin{equation}\label{}
   w(t,x):= v(t,x)-v^\mtd(t,x_\mh) , \quad h(t,x) :=\rho(t,x) -1, \quad q(t,x) :=p(t,x)-p^\mtd(t,x_\mh),
\end{equation}
we mainly consider the following perturbed system
\begin{equation}\label{pertS1}
\begin{cases}
  h_t + v\cdot\nabla h=0, \\
  w_t + v\cdot\nabla w-  \Delta w + \nabla q = F, \\
  \divg w=0, \\
  w|_{t=0}=w_0,\quad h|_{t=0}=h_0,
\end{cases}
\end{equation}
where
\begin{equation}\label{pertS2}
  F:= -h (v^\mtd)_t - h \,w_t -h(v\cdot\nabla w) -\rho( w_\mh \cdot\nabla_\mh v^\mtd) -h (v^\mtd_\mh\cdot\nabla_\mh v^\mtd).
\end{equation}

Let us emphasize that such a stability analysis has been well developed for the 3D homogeneous Navier-Stokes equations (see e.g. \cite{CG06,ChGP,G,KRZ,Zaj}),
but has not been pursued for (INS) system \eqref{INS}.


\bigskip

Our first result concerns the flow in the whole space with regular initial density.

\begin{theorem}\label{thm:regu}
  Let  $p>3$ and $v^\mtd=(v^\mtd_1,v^\mtd_2, v^\mtd_3)$ be the unique strong solution to 2D (HNS) system \eqref{HNS} on $\R^2$ with $v^\mtd_0\in L_2\cap \dot B^{3-2/p}_{p,p}(\R^2)$.
Assume that $\rho_0-1\in L_2\cap L_\infty(\R^3)$, $\nabla \rho_0\in L_3(\R^3)$ and $v_0-v^\mtd_0\in L_2\cap \dot B^{2-2/p}_{p,p}(\R^3)$.
There exist two generic constants $c_0,C'>0$ such that if $(\rho_0,v_0)$ satisfies
\begin{equation}\label{t1}
  \|\rho_0-1\|_{ L_2\cap L_\infty(\R^3)} + \|v_0-v^\mtd_0\|_{L_2\cap\dot B^{2-2/p}_{p,p}(\R^3)}\leq c_0\,
  \exp\left\{-C'\left(\|v^\mtd_0\|_{L_2\cap \dot B^{2-2/p}_{p,p}}^{4p} +1\right)e^{C'(1+\|v^\mtd_0\|_{L_2}^4)}\right\},
\end{equation}
then we have a unique global-in-time solution $(\rho,v)$ to the system (INS).
Furthermore, the solution $(\rho,v)$ obeys the following estimates
\begin{equation}\label{t2}
  \|\rho-1\|_{L_\infty(0,\infty; L_2\cap L_\infty(\R^3))}\leq \|\rho_0-1\|_{L_2\cap L_\infty(\R^3)},
\end{equation}
and
\begin{equation}\label{t3}
  \sup_{t<\infty} \|v(t)-v^\mtd(t)\|_{L_2\cap \dot B^{2-2/p}_{p,p}(\mathbb{R}^3)} + \|(v-v^\mtd)_t,\nabla^2 (v-v^\mtd),\nabla(p-p^\mtd)\|_{L_p(\R^3\times(0,\infty))} \leq C \,c_0,
\end{equation}
and
\begin{equation}\label{t4}
  \sup_{t\in [0,T]} \|\nabla \rho(t)\|_{L_3(\mathbb{R}^3)} \leq  \|\nabla \rho_0\|_{L_3(\R^3)} e^{(1+ T)C(h_0,w_0,v^\mtd_0)}, \quad \textrm{for any $T>0$},
\end{equation}
with $C(h_0,w_0,v^\mtd_0)>0$ some constant depending only on the initial data.
\end{theorem}

The above Theorem is a version of result for the 3D homogeneous Navier-Stokes system (with the constant density) like in \cite{CG06}.
What is important, from the viewpoint of regularity of density, Theorem \ref{thm:regu} is not optimal.
It shall be underlined that the extra regularity of density $\nabla \rho_0\in L_3$ is needed to control the uniqueness only.

Our second result removes this extra regularity condition of density and also shows the global stability result.
\begin{theorem}\label{thm:rough}
  Let $p>3$ and $v^\mtd=(v^\mtd_1,v^\mtd_2,v^\mtd_3)$ be the unique strong solution to 2D (HNS) system \eqref{HNS} on $\R^2$ with $v^\mtd_0\in L_2\cap \dot B^{4-2/p}_{p,p}(\R^2)$.
Assume that $\rho_0-1\in L_2\cap L_\infty(\R^3)$ and $v_0-v^\mtd_0\in L_2\cap \dot B^{2-2/p}_{p,p}(\R^3)$.
There exist two generic constants $c_0,C'>0$ such that if the initial data $(\rho_0,v_0)$ satisfies \eqref{t1},
then we have a unique global-in-time solution $(\rho,v)$ to the system (INS) which obeys the uniform estimates \eqref{t2} and \eqref{t3}.
\end{theorem}

As a direct application of Theorem \ref{thm:rough}, we have the following result on the density patch problem of (INS) system \eqref{INS}.
\begin{proposition}
  Let $\mathcal{D}_0\subset\R^3$ a bounded simple and connected set,
and let $\rho_0 =1- \eta 1_{\mathcal{D}_0}$
with $\eta\in \R$ a small constant and $1_{\mathcal{D}_0}$ the standard indicator function on $\mathcal{D}_0$.
Assume that $v_0(x) = v^\mtd_0(x_\mh) + w_0(x)$ is defined as in Theorem \ref{thm:rough}.
There exist two generic constants $c_0,C'>0$ such that if $(\rho_0,v_0)$ satisfies
\begin{equation*}
  |\eta|(1+|\mathcal{D}_0|^{1/2}) + \|v_0-v^\mtd_0\|_{L_2\cap\dot B^{2-2/p}_{p,p}(\R^3)}\leq c_0\, \exp\left\{-C' \left(\|v^\mtd_0\|_{L_2\cap \dot B^{2-2/p}_{p,p}(\R^2)}^{4p} +1\right)e^{C'(1+\|v^\mtd_0\|_{L_2}^4)}\right\},
\end{equation*}
with $|\mathcal{D}_0|$ the volume of $\mathcal{D}_0$, then the (INS) system \eqref{INS} generates a unique global solution $(\rho,v)$ which satisfies \eqref{t2}-\eqref{t3}.
\end{proposition}

By virtue of \eqref{t3}, Lemmas \ref{lem:2DNS-L2} - \ref{lem:2DNS-2} below and the Sobolev embedding, we deduce that for $\gamma\in(0, 1-3/p] $, $p>3$ and for every $T>0$,
\begin{equation*}\label{eq:v-es}
  \|v\|_{L_p(0,T; C^{1,\gamma}(\R^3))} \leq \|v^\mtd\|_{L_p(0,\infty; C^{1,\gamma}(\R^2))} + C \|v-v^\mtd\|_{L_p(0,\infty; \dot W^2_p)}
  + C T^{1/p} \|v-v^\mtd\|_{L_\infty(0,\infty; L_2)}<\infty.
\end{equation*}
An important consequence of this estimate is that if initial boundary $\partial \mathcal{D}_0$ is $C^{1,\gamma}$-regular, then its evolution
$\partial \mathcal{D}_0(t) = X_v(t,\partial \mathcal{D}_0)$ remains $C^{1,\gamma}$-regular (e.g. see \cite[Pg. 346]{GGJ}),
where $X_v(t,\cdot)$ defined as \eqref{flow} is the flow generated by $v$.
For 2D or 3D inhomogeneous Navier-Stokes system with more general density patches $\rho_0 = \rho_1 1_{\mathcal{D}_0^c} + \rho_2 1_{\mathcal{D}_0}$, $\rho_1,\rho_2>0$ constants,
one can see the recent interesting works \cite{DanZ,GGJ,LiaZ16,LiaZ18} for various results on the persistence of initial regularity of the free boundary $\partial \mathcal{D}_0$.
\\

For the proof of Theorem \ref{thm:regu}, we mainly follow the ideas from \cite{Mu01,Mu08}, where an information coming from the energy is fit to the estimates in the maximal regularity regime
(in our case in the $L_p$-spaces).
We first show the regularity and decay estimates of smooth solution to 2D (HNS), and based on which we prove the $L_2$-energy estimate of the perturbed system as well as the maximal regularity estimate in $L_p$-type spaces.
In particular, we have the regularity preservation estimate of density from the regular assumption of initial density.
Then we build a suitable approximate system of the considered perturbed equations,
and we use the a priori estimates to show the uniform estimates of the approximate solutions and get the $L_2$-strong convergence and uniqueness.
Note that the proof of uniqueness is based on the Eulerian coordinates approach, and that is why an information about the gradient of the density is required.

For the proof of Theorem \ref{thm:rough}, the existence is followed from the uniform estimates of approximative solutions established in Theorem \ref{thm:regu},
and the uniqueness is the main part. We adopt the Lagrangian coordinates approach originated in \cite{DanM12,DanM13} to prove the uniqueness of solutions in the rough density case, more precisely,
we consider the difference equation \eqref{LpertS-del} in the Lagrangian coordinates, and with the aid of Lemma \ref{lem:divEq} on the linear twisted divergence equation,
we can adapt an energy type argument to show the uniqueness.\\

At the end of this section we return to the physics and give an interpretation of our results. Theorems \ref{thm:regu} and \ref{thm:rough} say that if the initial perturbation is small (\ref{t1}), then the solutions
exist globally in time and they are close to the ones of 2D (HNS) system \eqref{HNS}. The condition (\ref{t1}) says the initial density must be close, only point-wisely,
to a constant. Hence the physical interpretation is the following: all
solutions to 2D (HNS) are stable globally in time in the inhomogeneous Navier-Stokes system regime, provided that perturbation of density is close to a constant in the $L_2\cap L_\infty$-norm.
It means the higher norms of the density
have no influence of our issue of stability. In other words, the dynamics of (INS) is the same as 2D (HNS), provided (\ref{t1}) is fulfilled.
\vskip0.1cm

The paper is organized as follows. In Section \ref{sec:prel} we give the preliminaries of our studies, introducing some basic notations, definitions and results.
In the next section we prove Theorem \ref{thm:regu}. Then in Section \ref{sec:thm2} we show Theorem \ref{thm:rough}.
In the last Appendix section we give the proof of two auxiliary results used in the previous sections.

\section{Preliminaries}\label{sec:prel}

Throughout the paper we use the standard notations: for every $p\in [1,\infty]$ and $m\in \N$, by $L_p(\R^n)$ we denote the standard Lebesgue space, by $W^m_p(\R^n)$ its natural generalization on Sobolev spaces,
by $\dot W^m_p(\R^n)$ the corresponding homogeneous Sobolev spaces, by $B^{m-n/p}_{p,p}(\R^n)$ and $\dot B^{m-n/p}_{p,p}(\R^n)$ the usual inhomogeneous and homogeneous Besov spaces (see e.g. \cite{Lad,LS,Zaj}).

\subsection{Linear estimates}
We will extensively use the following a priori estimate of the Stokes system at the whole-space case (the proof is classical, e.g. one can see \cite[Theorem 5]{DanM13}).
\begin{lemma}\label{lem:StokesR}
  Let $1<p<\infty$, $u_0\in \dot B^{2-2/p}_{p,p}(\R^n)$, $f\in L_p(\R^n\times(0,T))$.
Then the Stokes system
\begin{equation}\label{StokesR}
\begin{cases}
  \partial_t u - \Delta u + \nabla Q =f, &\quad\textrm{in   }\,\mathbb{R}^n\times (0,T), \\
  \divg u=0, &\quad\textrm{in   }\,\mathbb{R}^n\times (0,T), \\
  u|_{t=0}=u_0, &\quad\textrm{on   }\, \mathbb{R}^n,
\end{cases}
\end{equation}
has a unique solution $(u,\nabla Q)$ to \eqref{StokesR} with
\begin{equation*}
  u\in C(0,T; \dot B^{2-2/p}_{p,p}(\R^n)),\quad \textrm{and}\quad u_t,\nabla^2 u,\nabla Q\in L_p(\R^n\times (0,T)),
\end{equation*}
and the following estimate holds true
\begin{equation}\label{Lp-est2}
\begin{split}
  \sup_{0\leq t\leq T} \|u(t)\|_{\dot B^{2-2/p}_{p,p}(\R^n)} +
  \|u_t, \nabla^2 u, \nabla Q\|_{L_p(\R^n\times (0,T))}  \leq \, C \left( \|f\|_{L_p(\R^n\times (0,T))}  + \|u_0\|_{\dot B^{2-2/p}_{p,p}(\R^n)}\right),
\end{split}
\end{equation}
where $C$ is a positive constant independent of $T$.
\end{lemma}

Before we present the next lemma, we introduce the following auxiliary functional space: denote by $\mathcal{N}_p(T)$
the sum of two spaces $L_{\frac{2p}{2p-n}}(0,T; L_{\frac{2p}{p+2}}(\R^n))+ L_p(0,T;L_2(\R^n))$ expressed as
\begin{equation}\label{Np}
  \mathcal{N}_p(T):= \left\{f \big| f=a+b,a\in L_{\frac{2p}{2p-n}}(0,T; L_{\frac{2p}{p+2}}(\R^n)) , b\in L_2(0,T;L_2(\R^n))\right\}
\end{equation}
with the norm
\begin{equation*}
  \|f\|_{\mathcal{N}_p(T)}:= \inf \left\{\|a\|_{L_{\frac{2p}{2p-n}}(0,T; L_{\frac{2p}{p+2}})} + \|b\|_{L_2(0,T;L_2)} \big|
  f= a +b, a\in L_{\frac{2p}{2p-n}}(0,T; L_{\frac{2p}{p+2}}) , b\in L_2(0,T;L_2) \right\}.
\end{equation*}
Clearly, $\mathcal{N}_p(T)$ with the norm $\|\cdot\|_{\mathcal{N}_p(T)}$ is a Banach space.
Then we give the following lemma on the linear twisted divergence equation, which plays a key role in the uniqueness part of Theorem \ref{thm:rough}.
\begin{lemma}\label{lem:divEq}
  Let $p>n$ be fixed, and $A$ be a matrix valued function on $\R^n\times [0,T]$ satisfying $\mathrm{det}\,A\equiv 1$.
Let $g:\R^n\times [0,T]\rightarrow \R$ be a function given by
\begin{equation}
  g=\divg R
\end{equation}
which satisfies the following conditions
\begin{equation}\label{g-cd}
  g\in L_2(\R^n\times [0,T])\;\; \textrm{and}\;\; R\in L_\infty(0,T; L_2(\R^n)) \;\;\textrm{and}\;\; R_t\in L_{\frac{2p}{2p-n}}(0,T; L_{\frac{2p}{p+2}}(\R^n)).
\end{equation}
There exists a constant $c>0$ depending only on $n$, such that if
\begin{equation}\label{A-cd}
  \|\mathrm{Id}-A\|_{L_\infty(\R^n\times(0,T))} + \|A_t\|_{L_2(0,T; L_\infty(\R^n))} \leq c,
\end{equation}
then the twisted divergence equation
\begin{equation}\label{tDIVeq}
  \divg(A\, z) = g\quad \textrm{in} \quad\R^n\times [0,T]
\end{equation}
admits a solution $z$ in the space 
\begin{equation}
  X_T:= \left\{f|\,f\in L_\infty(0,T;L_2(\R^n)),\, \nabla f\in L_2(0,T;L_2(\R^n)),\,\, f_t\in \mathcal{N}_p(T) \right\},
\end{equation}
which satisfies the following estimates for some constant $C=C(n)$:
\begin{equation}\label{DIVest}
\begin{split}
  & \|z\|_{L_\infty(0,T;L_2(\R^n))}\leq C \|R\|_{L_\infty(0,T;L_2(\R^n))},\quad \|\nabla z\|_{L_2(0,T;L_2(\R^n))}\leq C \|g\|_{L_2(0,T;L_2(\R^n))}, \\
  & \|z_t\|_{\mathcal{N}_p(T)} \leq C \|R\|_{L_\infty(0,T;L_2(\R^n))} + C \|R_t\|_{L_{\frac{2p}{2p-n}}(0,T;L_{\frac{2p}{p+2}(\R^n)})}.
\end{split}
\end{equation}
\end{lemma}

\begin{proof}[Proof of Lemma \ref{lem:divEq}]
  For the proof, we mainly use the spirit of the corresponding part in \cite{DanMdiv} (see also \cite[Lemma 5.2]{DanM17}).

For any $z\in X_T$, we define
\begin{equation*}
  \Psi(z)\equiv \nabla \Delta^{-1}\divg\big( (\mathrm{Id}-A)z + R\big).
\end{equation*}
It is easy to see that $\Psi(z)$ satisfies the following linear equation
\begin{equation*}
  \divg(\Psi(z))= \divg\big( (\mathrm{Id}-A)z + R\big).
\end{equation*}
First we prove that $\Psi$ maps $X_T$ to $X_T$. From the 0-order operator $\nabla\Delta^{-1}\divg$ maps $L^2(\R^n)$ to $L^2(\R^n)$ with norm 1, we get
\begin{equation*}
\begin{split}
  \|\Psi(z)\|_{L_\infty(0,T; L_2(\R^n))} & \leq \|(\mathrm{Id}-A)z + g\|_{L_\infty(0,T; L_2(\R^2))} \\
  & \leq \|\mathrm{Id}-A\|_{L_\infty(0,T;L_\infty(\R^n))} \|z\|_{L_\infty(0,T; L_2(\R^n))} + \|g\|_{L_\infty(0,T; L_2(\R^2))}.
\end{split}
\end{equation*}
Noting that $\mathrm{det}\,A\equiv 1$ implies that (see e.g. \cite[Appendix]{DanM12})
\begin{equation}
  \divg(Az)=A^{\textrm{T}}:\nabla z,
\end{equation}
we obtain
\begin{equation*}
\begin{split}
  \|\nabla \Psi(z)\|_{L^2(0,T;L^2(\R^n))} & \leq \|(\mathrm{Id}-A^{\textrm{T}}): \nabla z\|_{L_2(0,T;L_2(\R^n))} + \|g\|_{L_2(0,T; L_2(\R^n))} \\
  & \leq \|\mathrm{Id}-A\|_{L_\infty(0,T; L_\infty(\R^n))} \|\nabla z\|_{L_2(0,T; L_2(\R^n))} + \|g\|_{L_2(0,T;L_2(\R^n))}.
\end{split}
\end{equation*}
Finally, due to that $((\mathrm{Id}-A)z)_t = (\mathrm{Id}-A)z_t + A_t z$, we find
\begin{align*}
  \|(\Psi(z))_t\|_{\mathcal{N}_p(T)} \leq & \|(\mathrm{Id}-A)z_t \|_{\mathcal{N}_p(T)}
  + \|A_t z\|_{L_2(0,T; L_2)} + \|R_t\|_{L_{\frac{2p}{2p-n}}(0,T; L_{\frac{2p}{p+2}})} \\
  \leq & \|\mathrm{Id}-A\|_{L_\infty(0,T;L_\infty)} \|z_t\|_{\mathcal{N}_p(T)} +  \|A_t\|_{L_2(0,T; L_\infty)} \|z\|_{L_\infty(0,T;L_2)}
  + \|R_t\|_{L_{\frac{2p}{2p-n}}(0,T; L_{\frac{2p}{p+2}})}.
\end{align*}
Hence the above estimates guarantee $\Psi:X_T\rightarrow X_T$. Then, for any $(z_1,z_2)\in X_T\times X_T$,
from $\Psi(z_1)-\Psi(z_2)=\nabla \Delta^{-1}\divg\big( (\mathrm{Id}-A)(z_1-z_2)\big)$ and by arguing as the above deductions, we infer that provided $c$ in \eqref{A-cd} is small enough,
\begin{align*}
  \|\Psi(z_1)-\Psi(z_2)\|_{X_T}\leq C\big(\|\mathrm{Id}-A\|_{L_\infty(0,T;L_\infty)} + \|A_t\|_{L_2(0,T; L_\infty)}\big) \|z_1-z_2\|_{X_T} \leq \frac{1}{2}\|z_1-z_2\|_{X_T}.
\end{align*}
Therefore, the classical Banach contraction mapping theorem ensures that there is a solution in $X_T$ to the equation $\Psi(z)=z$, which moreover satisfies the divergence equation \eqref{tDIVeq}.
Furthermore, coming back to the above estimates in the case $\Psi(z)=z$ leads to the desired inequalities \eqref{DIVest}.
\end{proof}

\subsection{The Lagrangian coordinates}\label{subsec:Lag}

The use of Lagrange coordinates plays a fundamental role in our proof of Theorem \ref{thm:rough}.
In this subsection, we introduce some notations and basic results related to the Lagrangian coordinates.

Let $X_v(t,y)$ solve the following ordinary differential equation (treating $y$ as a parameter)
\begin{equation}\label{flow}
  \frac{\dd X_v(t,y)}{\dd t} = v (t, X_v(t,y)),\quad X_v(t,y)|_{t=0}=y,
\end{equation}
which leads to the following relation
\begin{equation}\label{flow2}
  X_v(t,y) = y + \int_0^t v(\tau, X_v(\tau,y))\dd \tau.
\end{equation}

We list some basic properties for the Lagrangian change of variables.
\begin{lemma}\label{lem:Lag}
  Let $\Omega=\R^n$, $n=2,3$. Assume that $v\in L_1(0,T; \dot  W^1_\infty(\Omega))$. Then the system \eqref{flow} has a unique solution $X_v(t,y)$ on the time interval $[0,T]$
satisfying $\nabla_y X_v\in L_\infty(0,T; L_\infty)$ with
\begin{equation}\label{DXvest}
  \|\nabla_y X_v(t)\|_{L_\infty(\Omega)} \leq \exp \set{\int_0^t \|\nabla_x v(\tau)\|_{L_\infty(\Omega)}\dd \tau}.
\end{equation}
Furthermore, denoting by $\bar v(t,y):= v(t,X_v(t,y))$, we have
\begin{equation}\label{Xv}
  X_v(t,y)=y+ \int_0^t\bar v(\tau,y)\dd\tau,
\end{equation}
so that
\begin{equation}\label{DXveq}
  \nabla_y X_v(t,y)= \mathrm{Id} + \int_0^t \nabla_y \bar v(\tau,y)\dd \tau.
\end{equation}
Let $Y(t,\cdot)$ be the inverse diffeomorphism of $X(t,\cdot)$, then
$\nabla_x Y_v(t,x)= \left( \nabla_y X_v(t,y)\right)^{-1}$ with $x=X_v(t,y)$,
and if
\begin{equation}\label{vbar-cd}
  \int_0^t\|\nabla_y \bar v(\tau)\|_{L_\infty(\Omega)}\dd \tau\leq \frac{1}{2},
\end{equation}
we have
\begin{equation}\label{Avbd1}
  |\nabla_x Y_v(t,x) -\Id|\leq 2 \int_0^t |\nabla_y \bar v(\tau, y)|\dd \tau.
\end{equation}
\end{lemma}

\begin{proof}[Proof of Lemma \ref{lem:Lag}]
The proof is standard, and one can refer to \cite[Proposition 1]{DanM13}. We only note that as long as $\nabla_y X_v -\Id=\int_0^t \nabla_y \bar v(\tau,y)\dd \tau$
is sufficiently small so that \eqref{vbar-cd} holds, we have
\begin{equation}\label{DxY}
  \nabla_x Y_v = \left( \Id + (\nabla_y X_v -\Id)\right)^{-1} = \sum_{k=0}^\infty (-1)^k \left( \int_0^t \nabla_y \bar v(\tau,y)\dd \tau\right)^k,
\end{equation}
which immediately leads to \eqref{Avbd1}.
\end{proof}

By using the Lagrangian coordinates introduced as above, we set
\begin{equation}
\begin{split}
  & \bar{h}(t,y):=h(t, X_v(t,y)),\quad \bar{w}(t,y):= w(t, X_v(t,y)),\quad \\
  & \overline{q}(t,y):=q(t,X_v(t,y)),\quad\bar F(t,y):= F(t, X_v(t,y)),
\end{split}
\end{equation}
then according to the deduction in \cite{DanM12} or \cite{DanM13}, the perturbed system \eqref{pertS1} recasts in
\begin{equation}\label{LpertS}
\begin{cases}
  \bar{h}_t =0, \\
  \bar{w}_t -\divg\left( A_v A_v^{\textrm{T}} \nabla_y \bar{w}\right) + A_v^{\textrm{T}}\nabla_y \overline{q} = \bar F, \\
  \divg_y\left( A_v \bar{w} \right) =0, \\
  \bar{h}|_{t=0}= h_0,\quad \bar{w}|_{t=0}= w_0,
\end{cases}
\end{equation}
where we adopt the following notation
\begin{equation}\label{Av}
  A_v(t,y):=(\nabla_y X_v(t,y))^{-1}.
\end{equation}

As pointed out by \cite{DanM12,DanM13}, under the condition \eqref{vbar-cd}, the system \eqref{LpertS} in the Lagrangian coordinates is equivalent to the system \eqref{pertS1} in the Eulerian coordinates.
\vskip0.2cm

In the sequel we also denote
\begin{equation}\label{}
  \overline{v^\mtd}(t,y):= v^\mtd(t,X_{v,\mh}(t,y)), \quad\textrm{and}\quad \bar v(t,y):= v(t,X_v(t,y)),
\end{equation}
with $X_{v,\mh}(t,y)=\left( X_{v,1}(t,y), X_{v,2}(t,y)\right)$, then
\begin{equation}\label{bar-v}
  \bar v(t,y)= \overline{v^\mtd}(t,y) + \bar{w}(t,y).
\end{equation}
The first equation of \eqref{LpertS} implies
\begin{equation}\label{bar-h}
\bar{h}(t,y)\equiv h_0(y),\quad \forall t\in [0,T],
\end{equation}
and then thanks to the formula $w_t(t,X_v(t,y))= \bar{w}_t(t,y)- (v\cdot\nabla  w)(t, X_v(t,y))$, the formula of $F$ \eqref{pertS2} and \eqref{bar-v}, we further have
\begin{equation}\label{barFi}
\begin{split}
  \bar F(t,y)=& -h_0(y)\, (v^\mtd)_t\left( t,X_{v,\mh}\right) -h_0(y)\,  w_t\left( t,X_v\right) \\
  & -h_0(y) \,( v\cdot\nabla v)(t,X_v)- \left(  w_\mh\cdot\nabla_\mh v^\mtd\right) \left( t,X_v\right) \\
  = & -h_0(y)\, (v^\mtd)_t\left( t,X_{v,\mh}\right) - h_0(y)\, \bar{w}_t\left( t,y\right) \\
  & -h_0(y) \,(v^\mtd_\mh\cdot \nabla_\mh v^\mtd)(t,X_{v,\mh})- \rho_0(y)\, \overline{w_\mh}(t,y)\,\cdot (\nabla_\mh v^\mtd)\left( t,X_{v,\mh}\right) \\
  := & \,\bar F^1(t,y) + \bar F^2(t,y) + \bar F^3(t,y) + \bar F^4(t,y).
\end{split}
\end{equation}

\smallskip
\section{Proof of Theorem \ref{thm:regu}: global stability result with regular density}\label{sec:thm1}

\subsection{A priori estimates}

In this subsection we construct the a priori estimates used in proving Theorem \ref{thm:regu}.

\subsubsection{A priori estimates for 2D (HNS) on $\R^2$}

For the solution of 2D (HNS) \eqref{HNS} with $L_2$-initial data, we have the following a priori results.
\begin{lemma}\label{lem:2DNS-L2}
  Let $v^\mtd_0= (v^\mtd_{1,0}, v^\mtd_{2,0}, v^\mtd_{3,0})\in L_2(\R^2)$, then there exists a unique strong solution $v^\mtd$ to the three-component 2D (HNS) system \eqref{HNS} on $(0,\infty)$ which is also smoothly regular for all $t>0$.
Moreover, the statements as follows hold true.
\begin{enumerate}[(1)]
\item
$v^\mtd$ satisfies the $L_2$-energy estimate
\begin{equation}\label{L2es2D}
  \|v^\mtd(t)\|_{L_2(\R^2)}^2 + 2 \|\nabla_\mh v^\mtd\|_{L_2(\R^2\times(0,t))}^2\leq \|v^\mtd_0\|_{L_2(\R^2)}^2,\quad \forall t\geq 0.
\end{equation}
\item
There is a generic constant $C>0$ such that
\begin{equation}\label{L2Linf2D}
\begin{split}
  \|v^\mtd\|_{L_2(\R^+; L_\infty(\R^2))}^2 
  \leq C \|v^\mtd_0\|_{L_2(\R^2)}^2 \left( 1 + \|v^\mtd_0\|_{L_2(\R^2)}^2 \log^2(e + \|v^\mtd_0\|_{L_2(\R^2)})\right).
\end{split}
\end{equation}
\item
$v^\mtd$ also satisfies the following energy type estimates that for every $t\geq 0$,
\begin{equation}\label{L2es2d-2}
  t \|\nabla_\mh v^\mtd(t)\|_{L_2(\R^2)}^2 + \int_0^t \tau  \big( \|\partial_\tau v^\mtd\|_{L_2(\R^2)}^2 + \|\nabla_\mh^2 v^\mtd\|_{L_2(\R^2)}^2 \big)\dd \tau
  \leq  C \|v^\mtd_0\|_{L_2(\R^2)}^2 e^{C \|v^\mtd_0\|_{L_2}^4},
\end{equation}
and
\begin{equation}\label{L2es2d-3}
  t^2 \|\partial_t v^\mtd(t)\|_{L_2(\R^2)}^2 + t^2 \|\nabla_\mh^2 v^\mtd\|_{L_2(\R^2)}^2 +\int_0^t \tau^2  \|\nabla_\mh\partial_\tau v^\mtd\|_{L_2}^2 \dd \tau
  \leq  C \|v^\mtd_0\|_{L_2(\R^2)}^2 e^{C (1+\|v^\mtd_0\|_{L_2}^4)},
\end{equation}
and
\begin{equation}\label{L2es2d-4}
  t^3 \|\nabla_\mh\partial_t v^\mtd(t)\|_{L_2(\R^2)}^2 +  t^4 \|\nabla_\mh^2\partial_t v^\mtd(t)\|_{L_2(\R^2)}^2
  \leq  C \|v^\mtd_0\|_{L_2(\R^2)}^2 e^{C(1+ \|v^\mtd_0\|_{L_2}^4)}.
\end{equation}
\item
There is an absolute constant $C>0$ such that for every $t>0$,
\begin{align}
  &\|v^\mtd(t)\|_{L_p(\R^2)}\leq C \|v^\mtd_0\|_{L_2(\R^2)} e^{C (1+\|v^\mtd_0\|_{L_2}^4)}\, t^{\frac{1}{p}- \frac{1}{2}},\qquad\;\;\; p\in [2,\infty], \label{decHNS2d1}\\
  & \|\nabla_\mh v^\mtd(t)\|_{L_p(\R^2)} \leq C \|v^\mtd_0\|_{L_2(\R^2)} e^{C (1+\|v^\mtd_0\|_{L_2}^4)}\, t^{\frac{1}{p}-1} ,\qquad  p\in [2,\infty],\label{decHNS2d3} \\
  & \|\partial_t v^\mtd(t)\|_{L_\infty(\R^2)} \leq C \|v^\mtd_0\|_{L_2(\R^2)} e^{C(1+\|v^\mtd_0\|_{L_2}^4)}\, t^{-\frac{3}{2}}.\label{decHNS2d4}
\end{align}
\end{enumerate}
\end{lemma}

\begin{remark}
  Up to an exponential function depending on $\|v^\mtd_0\|_{L_2(\R^2)}$,
the decay estimates \eqref{decHNS2d1}-\eqref{decHNS2d4} are the same with those for $e^{t\Delta_\mh} v^\mtd_0$, which is solution to the free heat equation.
We also note that these estimates \eqref{decHNS2d1}-\eqref{decHNS2d4} at $p=\infty$ case remove the additional logarithmic function $\sqrt{\log(1+t)}$
on the right-hand side of \cite[(1.3)-(1.9)]{KozO}.
\end{remark}

\begin{proof}[Proof of Lemma \ref{lem:2DNS-L2}]

For the existence, uniqueness and smoothness on $(0,\infty)$ issues of strong solution $v^\mtd$, the proof is classical and one can see \cite[Theorem]{KozO} for the details (noting that the treating of the three-component vector field $v^\mtd$ follows almost the same lines
with that of $v^\mtd_\mh$).
The energy estimate \eqref{L2es2D} for $v^\mtd$ can be deduced in a standard way, while the estimate \eqref{L2Linf2D} is just the same result as \cite[Theorem 3]{CG06}.
By imposing some suitable time weights in obtaining energy type estimates (for the original ideas see the past works \cite{Hof,PZZ} and references therein),
we can prove \eqref{L2es2d-2}-\eqref{L2es2d-4} in an elementary approach, and we place the proof in the appendix section.

The decay estimates \eqref{decHNS2d1}-\eqref{decHNS2d4} for $v^\mtd$ are immediately followed from \eqref{L2es2d-2}-\eqref{L2es2d-4}
and the interpolation inequalities $\|g\|_{L_p(\R^2)} \leq \|g\|_{L_2(\R^2)}^{2/p} \|\nabla_\mh g\|_{L_2(\R^2)}^{1-2/p}$ for $p\in [2,\infty[$
and $\|g\|_{L_\infty(\R^2)} \leq \|g\|_{L_2(\R^2)}^{1/2} \|\nabla_\mh^2 g\|_{L_2(\R^2)}^{1/2}$.

\end{proof}

If the initial data $v^\mtd_0$ is more regular, we also have the following refined a priori estimates.
\begin{lemma}\label{lem:2DNS-2}
  Let $v^\mtd_0\in L_2\cap \dot B^{3-2/p}_{p,p}(\R^2)$ with $p>3$, then the unique global smooth solution $v^\mtd=(v^\mtd_\mh, v^\mtd_3)$ of the 2D (HNS) system satisfies that
\begin{equation}\label{v2d-es1}
\begin{split}
  \sup_{t\geq 0}\|v^\mtd \|_{\dot B^{3-2/p}_{p,p}} + \|\partial_t v^\mtd, \nabla^2_\mh v^\mtd\|_{L_p(\R^+; W^1_p)}
  \leq\,   C \Big( \|v^\mtd_0\|_{L_2\cap \dot B^{3-2/p}_{p,p}}^3 + 1\Big)\|v^\mtd_0\|_{L_2\cap \dot B^{3-2/p}_{p,p}}e^{C(1+\|v^\mtd_0\|_{L_2}^4)}.
\end{split}
\end{equation}
\end{lemma}

\begin{proof}[Proof of Lemma \ref{lem:2DNS-2}]
We first consider the estimate of $\|\nabla_\mh v^\mtd(t)\|_{L_p(\R^2)}$. Noticing that the vorticity $\omega^\mtd_\mh :=\partial_1 v^\mtd_2 - \partial_2 v^\mtd_1$ satisfies the transport equation
\begin{equation}\label{omg2dEq}
  \partial_t \omega^\mtd_\mh + v^\mtd_\mh \cdot\nabla_\mh \omega^\mtd_\mh -\Delta_\mh \omega^\mtd_\mh =0,
\end{equation}
we multiply the above equation with $|\omega^\mtd_\mh|^{p-2}\omega^\mtd_\mh$ and integrate on the spatial variables to get
\begin{equation*}
  \frac{1}{p}\frac{\dd}{\dd t}\|\omega^\mtd_\mh \|_{L^p}^p + (p-1)\int_{\R^2} |\nabla \omega^\mtd_\mh |^2 |\omega^\mtd_\mh |^{p-2}\dd x_\mh =0,
\end{equation*}
which combined with the continuous embedding $L_2\cap\dot B^{3-2/p}_{p,p}(\R^2)\hookrightarrow \dot W^1_p(\R^2)$ leads to
\begin{equation}\label{bomgLpes}
  \|\omega^\mtd_\mh (t)\|_{L_p(\R^2)}\leq \|\omega^\mtd_0 \|_{L_p(\R^2)}\leq C \|v^\mtd_{\mh,0}\|_{L_2\cap \dot B^{3-2/p}_{p,p}(\R^2)}.
\end{equation}
By the Calder\'on-Zygmund theorem, we immediately obtain
\begin{equation}\label{bomgLpes2}
  \|\nabla_\mh v^\mtd_\mh(t)\|_{L_p(\R^2)}\leq C\|v^\mtd_{\mh,0}\|_{L_2\cap \dot B^{3-2/p}_{p,p}(\R^2)}.
\end{equation}
Now we turn to the estimate of $\|\nabla_\mh v^\mtd_3(t)\|_{L_p}$. From the maximum principle and $L_r$-estimate of the transport-diffusion equation (similar to \eqref{bomgLpes}), we get that for every $r\in [2,\infty]$,
\begin{equation}
  \|v^\mtd_3(t)\|_{L_r(\R^2)} \leq \|v^\mtd_{3,0}\|_{L_r(\R^2)},\quad \forall t\geq 0.
\end{equation}
Observe that
\begin{equation}\label{nab-v3Eq}
  \partial_t (\nabla_\mh v^\mtd_3) + v^\mtd_\mh \cdot\nabla_\mh (\nabla_\mh v^\mtd_3) - \Delta_\mh (\nabla_\mh v^\mtd_3) = - (\nabla_\mh v^\mtd_\mh) \cdot \nabla_\mh v^\mtd_3.
\end{equation}
By taking the scalar product of the above equation with $|\nabla_\mh v^\mtd_3|^{p-2} \nabla_\mh v^\mtd_3$, and using the divergence-free property of $v^\mtd_\mh$ and the integration by parts, we obtain
\begin{equation*}
\begin{split}
  \frac{1}{p} \frac{\dd }{\dd t} \|\nabla_\mh v^\mtd_3(t)\|_{L_p}^p & + (p-1) \int_{\R^2} |\nabla_\mh^2 v^\mtd_3|^2 |\nabla_\mh v^\mtd_3|^{p-2} \dd x_\mh
  \leq  \left|\int_{\R^2} |\nabla_\mh v^\mtd_3|^{p-2}\nabla_\mh v^\mtd_3 \cdot \nabla_\mh v^\mtd_\mh \cdot \nabla_\mh v^\mtd_3  \dd x\right|  \\
  & \leq (p-1) \int_{\R^2} |\nabla_\mh v^\mtd_3|^{p-2} |\nabla^2_\mh v^\mtd_3| |\nabla_\mh v^\mtd_\mh| |v^\mtd_3|(t,x_\mh) \dd x_\mh.
\end{split}
\end{equation*}
The Young inequality and H\"older's inequality lead to
\begin{equation*}
\begin{split}
  \frac{1}{p} \frac{\dd }{\dd t} \|\nabla_\mh v^\mtd_3\|_{L_p}^p + \frac{p-1}{2} \int_{\R^2} |\nabla_\mh^2 v^\mtd_3|^2|\nabla_\mh v^\mtd_3|^{p-2} & \dd x_\mh
  \leq C_0(p-1) \int_{\R^2} |\nabla_\mh v^\mtd_3|^{p-2} |\nabla_\mh v^\mtd_\mh|^2 |v^\mtd_3|^2 \dd x_\mh \\
  & \leq C_0 (p-1) \|\nabla_\mh v^\mtd_3\|_{L_p}^{p-2} \|\nabla_\mh v^\mtd_\mh\|_{L_p}^2 \|v^\mtd_{3,0}\|_{L_\infty}^2,
\end{split}
\end{equation*}
which implies
\begin{equation*}
  \frac{1}{2} \frac{\dd }{\dd t} \|\nabla_\mh v^\mtd_3\|_{L_p}^2 \leq C_0 (p-1) \|v^\mtd_{3,0}\|_{L_\infty}^2 \|\nabla_\mh v^\mtd_\mh(t)\|_{L_p}^2.
\end{equation*}
By using the following estimate (deduced from \eqref{decHNS2d3} and \eqref{bomgLpes})
\begin{equation}\label{fact5}
\begin{split}
  \int_0^\infty \|\nabla_\mh v^\mtd_\mh(t)\|_{L_p}^2 \dd t & \leq \int_0^1 \|\omega^\mtd_\mh(t)\|_{L_p}^2 \dd t + C \|v^\mtd_0\|_{L_2}^2 e^{C(1+\|v^\mtd_0\|_{L_2}^4)} \int_1^\infty t^{-(2-\frac{2}{p})}\dd t \\
  & \leq \|\omega^\mtd_{\mh,0}\|_{L_p}^2 + C \|v^\mtd_0\|_{L_2}^2 e^{C(1+\|v^\mtd_0\|_{L_2}^4)} \leq C \|v^\mtd_0\|_{L_2\cap \dot B^{3-2/p}_{p,p}(\R^2)}^2 e^{C(1+\|v^\mtd_0\|_{L_2}^4)} ,
\end{split}
\end{equation}
we integrate in the time variable to find
\begin{equation}\label{nab-v3Lp}
\begin{split}
  \|\nabla_\mh v^\mtd_3(t)\|_{L_p}^2 \leq &   \|\nabla_\mh v^\mtd_{3,0}\|_{L_p}^2  + 2 C_0(p-1) \|v^\mtd_{3,0}\|_{L_\infty}^2 \int_0^\infty \|\nabla_\mh v^\mtd_\mh(t)\|_{L_p}^2 \dd t \\
  \leq & C \Big(\|v^\mtd_0\|_{L_2\cap \dot B^{3-2/p}_{p,p}(\R^2)}^2 +1\Big)\|v^\mtd_0\|_{L_2\cap \dot B^{3-2/p}_{p,p}(\R^2)}^2 e^{C(1+\|v^\mtd_0\|_{L_2}^4)}.
\end{split}
\end{equation}

Next we consider the estimate of $\|v^\mtd(t)\|_{\dot B^{2-2/p}_{p,p}(\R^2)}$. From the 2D (HNS) system \eqref{HNS}, by using Lemma \ref{lem:StokesR}, we have
\begin{equation*}
\begin{split}
  \sup_{t>0}\|v^\mtd(t)\|_{\dot B^{2-2/p}_{p,p}(\R^2)} + \|\partial_t v^\mtd,\nabla^2 v^\mtd\|_{L_p(\R^2\times\R^+)}
  \leq \, C \left( \|v^\mtd_\mh \cdot\nabla v^\mtd\|_{L_p(\R^2 \times\R^+)} + \|v^\mtd_0\|_{\dot B^{2-2/p}_{p,p}(\R^2)}\right).
\end{split}
\end{equation*}
From \eqref{decHNS2d1} and \eqref{bomgLpes2}, we deduce that for every $p>3$,
\begin{equation}\label{fact3}
\begin{split}
  \|v^\mtd_\mh\|_{L_p(\R^+; L_\infty(\R^2))}
  \leq & \|v^\mtd_\mh\|_{L_p(0,1; L_\infty(\R^2))} + \|v^\mtd_\mh\|_{L_p(1,\infty; L_\infty(\R^2))}\\
  \leq &  \|v^\mtd_\mh\|_{L_\infty(0,1; L_2\cap \dot W^1_p)} +  C\|v^\mtd_0\|_{L_2}e^{C(1+\|v^\mtd_0\|_{L_2}^4)}\left(\int_1^\infty t^{-\frac{p}{2}} \dd t\right)^{1/p} \\
  \leq & C \|v^\mtd_{\mh,0}\|_{L_2\cap \dot B^{3-2/p}_{p,p}} e^{C(1+\|v^\mtd_0\|_{L_2}^4)},
\end{split}
\end{equation}
and
\begin{equation}\label{eq:est3}
\begin{split}
  \|v^\mtd_\mh\cdot\nabla_\mh v^\mtd\|_{L_p(\R^2\times \R^+)} & \leq \|v^\mtd_\mh\|_{L_p(\R^+; L_\infty(\R^2))} \|\nabla_\mh v^\mtd \|_{L_\infty(\R^+; L_p(\R^2))} \\
  & \leq C \left(\|v^\mtd_0\|_{L_2\cap \dot B^{3-2/p}_{p,p}(\R^2)}^2 +1\right)\|v^\mtd_0\|_{L_2\cap \dot B^{3-2/p}_{p,p}(\R^2)}e^{C(1+\|v^\mtd_0\|_{L_2}^4)}.
\end{split}
\end{equation}
Gathering the above estimates leads to
\begin{equation}\label{estma0}
\begin{split}
  \sup_{t\geq 0}\|v^\mtd\|_{\dot B^{2-2/p}_{p,p}} + \|\partial_t v^\mtd,\nabla_\mh^2 v^\mtd\|_{L_p(\R^2\times\R^+)}
  \leq C \left( \|v^\mtd_0\|_{L_2\cap \dot B^{3-2/p}_{p,p}}^2  + 1\right) \|v^\mtd_0\|_{L_2\cap \dot B^{3-2/p}_{p,p}}e^{C(1+\|v^\mtd_0\|_{L_2}^4)}.
\end{split}
\end{equation}
In particular, since a better estimate \eqref{bomgLpes2} holds and we can use it in \eqref{eq:est3}, we also infer that
\begin{equation}\label{estma0.2}
\begin{split}
  \sup_{t\geq 0}\|v^\mtd_\mh\|_{\dot B^{2-2/p}_{p,p}} + \|\partial_t v^\mtd_\mh,\nabla_\mh^2 v^\mtd_\mh\|_{L_p(\R^2\times\R^+)}
  \leq C \left( \|v^\mtd_0\|_{L_2\cap \dot B^{3-2/p}_{p,p}}  + 1\right) \|v^\mtd_{\mh,0}\|_{L_2\cap \dot B^{3-2/p}_{p,p}}e^{C(1+\|v^\mtd_0\|_{L_2}^4)}.
\end{split}
\end{equation}

Finally, we intend to estimate $\|v^\mtd(t)\|_{\dot B^{3-2/p}_{p,p}(\R^2)}$. Noting that
\begin{equation*}
  \partial_t (\nabla_\mh v^\mtd) - \Delta_\mh (\nabla_\mh v^\mtd) + \nabla (\nabla_\mh p^\mtd)
   = - v^\mtd_\mh \cdot\nabla_\mh (\nabla_\mh v^\mtd)- (\nabla_\mh v^\mtd_\mh) \cdot \nabla_\mh v^\mtd,
\end{equation*}
and by using Lemma \ref{lem:StokesR} again, we get
\begin{equation*}
\begin{split}
  & \sup_{t\geq 0}\|\nabla_\mh v^\mtd \|_{\dot B^{2-2/p}_{p,p}(\R^2)} + \|(\nabla_\mh\partial_t  v^\mtd , \nabla_\mh^3 v^\mtd )\|_{L_p(\R^2\times\R^+)} \\
  \leq &  C \left( \| v^\mtd_\mh\cdot\nabla_\mh^2 v^\mtd \|_{L_p( \R^2\times\R^+)}  + \|(\nabla_\mh v^\mtd_\mh)\cdot\nabla_\mh v^\mtd \|_{L_p( \R^2\times\R^+)}
  + \|v^\mtd_0 \|_{\dot B^{2-2/p}_{p,p}(\R^2)} \right).
\end{split}
\end{equation*}
By virtue of H\"older's inequality, the interpolation inequality, Cauchy's inequality and \eqref{fact3}, we deduce
\begin{equation*}
\begin{split}
  & \left(\int_0^\infty\|v^\mtd_\mh\cdot\nabla_\mh^2 v^\mtd(t) \|_{L_p}^p \dd t\right)^{1/p}
  \leq \left( \int_0^\infty \|v^\mtd_\mh\|_{L_\infty}^p \|\nabla_\mh v^\mtd \|_{L_p}^{p/2} \|\nabla_\mh^3 v^\mtd \|_{L_p}^{p/2} \dd t\right)^{1/p} \\
  & \leq \frac{1}{4C} \|\nabla^3_\mh v^\mtd\|_{L_p(\R^2\times\R^+)}
  + C \|v^\mtd_\mh\|_{L_{2p}(\R^+; L_\infty)}^2 \|\nabla_\mh v^\mtd\|_{L_\infty(\R^+; L_p)} \\
  & \leq \frac{1}{4C} \|\nabla^3_\mh v^\mtd\|_{L_p(\R^2\times\R^+)} + C \Big(\|v^\mtd_0\|_{L_2\cap \dot B^{3-2/p}_{p,p}(\R^2)}^3 +1\Big)
  \|v^\mtd_0\|_{L_2\cap \dot B^{3-2/p}_{p,p}(\R^2)} e^{C(1+\|v^\mtd_0\|_{L_2}^4)} .
\end{split}
\end{equation*}
From the continuous embedding $L_2\cap \dot W^2_p(\R^2)\hookrightarrow W^1_\infty(\R^2)$ and \eqref{decHNS2d3}, \eqref{estma0.2}, we find
\begin{align}\label{fact4}
  \Big(\int_0^\infty \|\nabla_\mh v^\mtd_\mh\|_{L_\infty(\R^2)}^p\dd t\Big)^{1/p} \nonumber 
  & \leq \|\nabla_\mh v^\mtd_\mh\|_{L_p(0,1; L_\infty(\R^2))} + \Big(\int_1^\infty \|\nabla_\mh v^\mtd_\mh(t)\|_{L_\infty(\R^2)}^p\dd t\Big)^{1/p} \nonumber\\
  & \leq C \| v^\mtd_\mh\|_{L_p(0,1; L_2\cap \dot W^2_p(\R^2))} + C \|v^\mtd_0\|_{L_2} e^{C(1+\|v^\mtd_0\|_{L_2}^4)}
  \Big(\int_1^\infty t^{-p} \dd t\Big)^{1/p} \nonumber \\
  & \leq C \Big( \|v^\mtd_0\|_{L_2\cap \dot B^{3-2/p}_{p,p}(\R^2)}  + 1\Big)\|v^\mtd_0\|_{L_2\cap \dot B^{3-2/p}_{p,p}} e^{ C(1+\|v^\mtd_0\|_{L_2}^4)},
\end{align}
which yields
\begin{equation*}
\begin{split}
  \|\nabla_\mh v^\mtd_\mh\cdot \nabla_\mh v^\mtd\|_{L_p(\R^+; L_p(\R^2))} & \leq \|\nabla_\mh v^\mtd_\mh \|_{L_p(\R^+; L_\infty(\R^2))} \|\nabla_\mh v^\mtd\|_{L_\infty(\R^+; L_p(\R^2))} \\
  & \leq C  \Big( \|v^\mtd_0\|_{L_2\cap \dot B^{3-2/p}_{p,p}(\R^2)}^3  + 1\Big) \|v^\mtd_0\|_{L_2\cap \dot B^{3-2/p}_{p,p}(\R^2)} e^{C(1+\|v^\mtd_0\|_{L_2}^4)}.
\end{split}
\end{equation*}
Thus we have
\begin{equation}\label{v2d3-es2}
  \sup_{t\geq 0}\|\nabla_\mh v^\mtd \|_{\dot B^{2-2/p}_{p,p}} + \|\partial_t v^\mtd , \nabla_\mh^2 v^\mtd \|_{L_p(\R^+; \dot W^1_p)}
  \leq   C  \Big( \|v^\mtd_0\|_{L_2\cap \dot B^{3-2/p}_{p,p}}^3  + 1\Big)\|v^\mtd_0\|_{L_2\cap \dot B^{3-2/p}_{p,p}} e^{C(1+\|v^\mtd_0\|_{L_2}^4)}.
\end{equation}

Hence, by combining \eqref{v2d3-es2} with \eqref{estma0} we conclude the desired estimate \eqref{v2d-es1}.

\end{proof}

\subsubsection{A priori estimates for the perturbed system \eqref{pertS1}-\eqref{pertS2}}

First is the a priori estimate of $(h,w)$ with initial data $(h_0,w_0)\in (L_2\cap L_\infty)\times L_2$.
\begin{proposition}\label{prop:eneR2}
  Let $h_0=\rho_0-1\in L_2\cap L_\infty(\R^3)$, $v^\mtd_0=(v^\mtd_{\mh,0}, v^\mtd_{3,0})\in L_2\cap \dot B^{3-2/p}_{p,p}(\R^2)$ ($p>3$) and $w_0=v_0-v^\mtd_0\in L_2(\R^3)$,
and let $(h,w)$ be a sufficiently smooth solution to the system \eqref{pertS1}-\eqref{pertS2} over $\R^3\times [0,T]$. Then under the condition that $\|h_0\|_{L_\infty(\R^3)}\leq \frac{1}{2}$, there holds for all $t\in [0,T]$,
\begin{equation}\label{hLp}
  \|h(t)\|_{L^p(\R^3)}= \|h_0\|_{L^p(\R^3)} \quad\text{for each}\quad p\in[2,\infty],
\end{equation}
and
\begin{equation}\label{wL2es}
  \sup_{0\leq t <T} \| w(t)\|_{L_2(\R^3)} + \left(\int_0^T \|\nabla  w(\tau)\|_{L_2(\R^3)}^2\,\dd\tau \right)^{1/2}\leq C \|( w_0,h_0)\|_{L_2(\R^3)} e^{B(v^\mtd_0)},
\end{equation}
where
\begin{equation}\label{B0}
  B(v^\mtd_0):=C  \Big( \|v^\mtd_0\|_{L_2\cap \dot B^{3-2/p}_{p,p}(\R^2)}^4 + 1 \Big)e^{C(1+\|v^\mtd_0\|_{L_2}^4)}.
\end{equation}
In above, the statement still holds replacing time interval $[0,T]$ by $[0,\infty)$.
\end{proposition}

\begin{proof}[Proof of Proposition \ref{prop:eneR2}]
From the first equation of \eqref{pertS1}, the $L_p$-conservation \eqref{hLp} is directly deduced from the property of the transport equation.
\vskip0.1cm

Next we prove \eqref{wL2es}. Taking the $L^2$-inner product of the second equation of \eqref{pertS1} with $w$, we immediately have
\begin{equation}\label{eq.L2-dv}
\begin{split}
  \frac{1}{2}\frac{\dd}{\dd t} \| w(t)\|^2_{L^2}+\|\nabla w(t)\|^2_{L^2}
  =& -\int_{\R^3}h (v^\mtd)_t\cdot  w\, \dd x - \int_{\R^3} h ( w)_t\cdot  w\, \dd x
  -\int_{\R^3}h (v\cdot\nabla) w\cdot  w \,\dd x\\
  &-\int_{\R^3}\rho( w_\mh \cdot\nabla_\mh)v^\mtd\cdot  w \,\dd x -\int_{\R^3}h (v^\mtd_\mh \cdot\nabla_\mh)v^\mtd\cdot  w \,\dd x.
\end{split}
\end{equation}
From H\"older's inequality and Cauchy's inequality, we get
\begin{equation*}
\begin{split}
  \left|\int_{\R^3} h (v^\mtd)_t\cdot w(t,x)\,\dd x\right| & \leq \|h\|_{L_2(\R^3)} \|\partial_t v^\mtd\|_{L_\infty(\R^2)} \|w\|_{L_2(\R^3)} \\
  & \leq \frac{1}{2}\|h_0\|_{L_2(\R^3)}^2 \|\partial_t v^\mtd\|_{L_\infty(\R^2)} + \frac{1}{2} \|\partial_t v^\mtd\|_{L_\infty(\R^2)}\|w\|_{L_2(\R^3)}^2,
\end{split}
\end{equation*}
and
\begin{equation*}
\begin{split}
  \left|\int_{\R^3}h\, v^\mtd_\mh\cdot\nabla_\mh v^\mtd\cdot w(t,x)\, \dd x\right|
  \leq \,& \|h\|_{L_2(\R^3)} \|v^\mtd_\mh\|_{L_\infty(\R^2)} \|\nabla_\mh v^\mtd\|_{L_\infty(\R^2)} \|w\|_{L_2(\R^3)} \\
  \leq \, & \frac{1}{2} \|h_0\|_{L_2(\R^3)}^2 \|v^\mtd_\mh\|_{L_\infty(\R^2)}^2 + \frac{1}{2} \|\nabla_\mh v^\mtd\|_{L_\infty(\R^2)}^2 \| w\|_{L_2(\R^3)}^2.
\end{split}
\end{equation*}
By virtue of the H\"older inequality, the following interpolation inequality
\begin{equation}\label{eq:IntP}
  \|f\|_{L_4(\R^2)}\leq C_0 \|f\|_{L_2(\R^2)}^{1/2} \|\nabla_\mh f\|_{L_2(\R^2)}^{1/2}
\end{equation}
and the Cauchy inequality, we find
\begin{equation*}
\begin{split}
  \left|\int_{\R^3} \rho\, ( w_\mh\cdot\nabla_\mh v^\mtd)\cdot w(t,x)\,\dd x \right| & \leq \|\rho\|_{L_\infty(\R^3)} \|\nabla_\mh v^\mtd\|_{L_2(\R^2)}\| w\|_{L_{4,x_{\mh}}L_{2,x_3}(\R^3)}^2 \\
  & \leq \|\rho_0\|_{L_\infty(\R^3)} \|\nabla_\mh v^\mtd\|_{L_2(\R^2)} \|\nabla w\|_{L_2(\R^3)} \| w\|_{L_2(\R^3)} \\
  & \leq \frac{1}{2} \|\nabla  w\|_{L_2(\R^3)}^2 + C \|\rho_0\|_{L_\infty(\R^3)}^2 \|\nabla_\mh v^\mtd\|_{L_2(\R^2)}^2 \| w\|_{L_2(\R^3)}^2.
\end{split}
\end{equation*}
Plunging these above estimates into \eqref{eq.L2-dv} leads to
\begin{equation*}
\begin{split}
  &\frac{\dd}{\dd t}\| w(t)\|^2_{L^2(\R^3)}+\|\nabla w(t)\|^2_{L^2(\R^3)}\\
  \leq& - 2\int_{\R^3}h ( w)_t \cdot w\, \dd x - 2\int_{\R^3} h (v \cdot\nabla) w \cdot w\,\dd x
  + \left( \|\partial_t v^\mtd\|_{L_\infty(\R^2)}
  + \|\nabla_\mh v^\mtd\|_{L_\infty}^2 \right) \| w\|_{L_2}^2 \\
  & + C \|\rho_0\|_{L_\infty}^2 \|\nabla_\mh v^\mtd\|_{L_2}^2 \| w\|_{L_2}^2 + \|h_0\|_{L_2}^2 \left( \|\partial_t v^\mtd\|_{L_\infty}
  + \|v^\mtd_\mh\|_{L_\infty}^2\right) .
\end{split}
\end{equation*}
Through integrating in time, and using the following inequality deduced from the integration by parts
\begin{equation*}
\begin{split}
  &-2\int_0^t\int_{\R^3}h\, w_\tau \cdot w\dd x\mathrm{d}\tau - 2\int_0^t\int_{\R^3} h (v\cdot\nabla) w\cdot  w\dd x\mathrm{d}\tau\\
  =&-\int_0^t\int_{\R^3} h \,\partial_\tau| w|^2\dd x\mathrm{d}\tau -\int_0^t\int_{\R^3} h\, (v \cdot\nabla)| w|^2\dd x\mathrm{d}\tau\\
  =& \int_{\R^3}h_0(x) |w_0(x)|^2 \dd x - \int_{\R^3} h(t,x) |w(t,x)|^2\dd x +\int_0^t\int_{\R^3}(\partial_\tau h+v\cdot\nabla h) |w|^2\dd x\mathrm{d}\tau \\
  \leq & \|h_0\|_{L_\infty}\|w_0\|_{L_2}^2 + \|h_0\|_{L_\infty}\|w(t)\|_{L_2}^2,
\end{split}
\end{equation*}
we have
\begin{equation}\label{wes-key}
\begin{split}
  &\| w(t)\|_{L_2(\R^3)}^2 + \int_0^t \|\nabla  w(\tau)\|_{L_2(\R^3)}^2\,\dd\tau \\
  \leq\, &(1+\|h_0\|_{L_\infty})\| w_0\|_{L_2}^2 + \|h_0\|_{L_\infty}\|w(t)\|_{L_2}^2 +\int_0^t \big( \|\partial_\tau v^\mtd(\tau)\|_{L_\infty}
  +  \|\nabla_\mh v^\mtd(\tau)\|_{L_\infty}^2 \big) \| w(\tau)\|_{L_2}^2 \dd\tau \\
  & + C \|\rho_0\|_{L_\infty}^2 \int_0^t\|\nabla_\mh v^\mtd(\tau)\|_{L_2}^2 \| w(\tau)\|_{L_2}^2\dd\tau + \|h_0\|_{L_2}^2 \int_0^t \left(  \|\partial_\tau v^\mtd\|_{L_\infty(\R^2)}
  +\|v^\mtd_\mh\|_{L_\infty(\R^2)}^2\right) \dd \tau.
\end{split}
\end{equation}
By letting
$  \|h_0\|_{L_\infty(\R^3)}\leq 1/2,$
it follows from Gr\"onwall's inequality that
\begin{equation}\label{L2estR}
\begin{split}
  \| w(t)\|_{L_2(\R^3)}^2 + \int_0^t\|\nabla  w(\tau)\|_{L_2(\R^3)}^2\,\dd\tau \leq  \widetilde{B}_1\, e^{\widetilde{B}_2},
\end{split}
\end{equation}
where
\begin{equation*}
\begin{split}
  & \widetilde{B}_1: = 3\| w_0\|_{L_2(\R^3)}^2 + 2\|h_0\|_{L_2(\R^3)}^2 \left( \|\partial_t v^\mtd\|_{L_1(\R^+; L_\infty(\R^2))} +  \|v^\mtd_\mh\|_{L_2(\R^+; L_\infty(\R^2))}^2 \right), \\
  & \widetilde{B}_2:=C \|\partial_t v^\mtd\|_{L_1(\R^+; L_\infty(\R^2))} + C \|\nabla_\mh v^\mtd\|_{L_2(\R^2\times \R^+)}^2 \|\rho_0\|_{L_\infty}^2 +  \|\nabla_\mh v^\mtd\|_{L_2(\R^+;L_\infty(\R^2))}^2.
\end{split}
\end{equation*}
Thanks to \eqref{decHNS2d4} and \eqref{v2d-es1} (from $W^1_p(\R^2)\hookrightarrow L_\infty(\R^2)$), we deduce that
\begin{equation*}
\begin{split}
  \int_0^\infty\|\partial_t v^\mtd\|_{L_\infty(\R^2)}\dd t  & \leq \|\partial_t v^\mtd\|_{L_p(0, 1; W^1_p(\R^2))}
  + C  \|v^\mtd_0\|_{L_2}  e^{C(1+\|v^\mtd_0\|_{L_2}^4)}\int_1^\infty \frac{1}{t^{3/2}} \dd t \\
  & \leq  C  \Big( \|v^\mtd_0\|_{L_2\cap \dot B^{3-2/p}_{p,p}(\R^2)}^4 + 1 \Big)  e^{C(1+\|v^\mtd_0\|_{L_2}^4)}.
\end{split}
\end{equation*}
Similarly as estimating \eqref{fact4} and using \eqref{estma0}, we find
\begin{align*}
  \|\nabla_\mh v^\mtd\|_{L_2(\R^+; L_\infty(\R^2))} 
  & \leq C  \| v^\mtd\|_{L_p(0,1; L_2\cap \dot W^2_p(\R^2))} + C \|v^\mtd_0\|_{L_2(\R^2)}e^{C(1+\|v^\mtd_0\|_{L_2}^4)} \bigg(\int_1^\infty t^{-2} \dd t\bigg)^{1/2}\\
  & \leq C  \Big( \|v^\mtd_0\|_{L_2\cap \dot B^{3-2/p}_{p,p}(\R^2)}^3 + 1 \Big)  e^{C(1+\|v^\mtd_0\|_{L_2}^4)}.
\end{align*}
Thus from \eqref{L2es2D}, \eqref{L2Linf2D}, \eqref{v2d-es1} and the continuous embedding $L_2\cap \dot B^{3-2/p}_{p,p}(\R^2)\hookrightarrow \dot W^1_\infty(\R^2)$ for $p>3$,
we further have
\begin{align}
  & \widetilde{B}_1\leq  3\| w_0\|_{L_2(\R^3)}^2 + C \|h_0\|_{L_2(\R^3)}^2\Big( \|v^\mtd_0\|_{L_2\cap \dot B^{3-2/p}_{p,p}(\R^2)}^4 + 1 \Big)e^{C(1+\|v^\mtd_0\|_{L_2}^4)}, \nonumber\\
  & \widetilde{B}_2\leq
  C \left( 1+ \|\rho_0\|_{L_\infty(\R^3)}\right)\Big( \|v^\mtd_0\|_{L_2\cap \dot B^{3-2/p}_{p,p}(\R^2)}^4 + 1 \Big)e^{C(1+\|v^\mtd_0\|_{L_2}^4)},\nonumber
\end{align}
then inserting into \eqref{L2estR} leads to the desired estimate \eqref{wL2es}.
\end{proof}

The next result is concerned with the crucial $L_p$-based a priori estimate of $w$ under more regular assumption of initial data $w_0$.
\begin{proposition}\label{prop:apR2}
  Let $h_0=\rho_0-1\in L_2\cap L_\infty(\R^3)$, $v^\mtd_0\in L_2\cap \dot B^{3-2/p}_{p,p}(\R^2)$ and $w_0=v_0-v^\mtd_0\in L_2\cap \dot B^{2-2/p}_{p,p}(\R^3)$ with $p>3$,
and let $(h,w)$ be a sufficiently smooth solution to the system \eqref{pertS1}-\eqref{pertS2} over $\R^3\times [0,T]$. There exists a small absolute constant $\bar{c}_*>0$ such that if $(h_0,w_0,v^\mtd_0)$ satisfies
\begin{equation}\label{wh-cd}
  \left(\|h_0\|_{L_2\cap L_\infty(\R^3)}+\| w_0\|_{L_2\cap \dot B^{2-2/p}_{p,p}(\R^3)} \right)
  \exp\left\{C' \left(\|v^\mtd_0\|_{L_2\cap \dot B^{3-2/p}_{p,p}(\R^2)}^{4p} +1\right)e^{C'(1+\|v^\mtd_0\|_{L_2}^4)}\right\} \leq \bar{c}_*,
\end{equation}
with $C'$ some absolute constant appearing in \eqref{B1}, then we have
\begin{equation}\label{w-es}
\begin{split}
  \sup_{t\leq T} \|w(t)\|_{\dot B^{2-\frac{2}{p}}_{p,p}(\mathbb{R}^3)} + \|w_t,\nabla^2 w,\nabla q\|_{L_p(\R^3\times(0,T))} \leq C \bar{c}_*.
\end{split}
\end{equation}
In the above, the time interval $[0,T]$ can be replaced by $[0,\infty)$.
\end{proposition}

\begin{proof}[Proof of Proposition \ref{prop:apR2}]
Applying \eqref{Lp-est2} to the second equation of the system \eqref{pertS1}, we have
\begin{equation}\label{est1}
\begin{split}
  \sup_{\tau\leq t} \|w\|_{\dot B^{2-\frac{2}{p}}_{p,p}(\mathbb{R}^3)}^p + \|\partial_\tau w,\nabla^2 w,\nabla q\|_{L_p(\R^3\times(0,t))}^p
  \leq \;  C \left( \sum_{i=1}^6\no{F_i}_{L_p(\mathbb{R}^3\times(0,t))}^p   +  \|w_0\|_{B^{2-2/p}_{p,p}(\mathbb{R}^3)}^p \right),
\end{split}
\end{equation}
with
\begin{equation}\label{Fia}
\begin{split}
  & F_1:=\rho(v^\mtd\cdot\nabla w), \quad F_2:=\rho(w\cdot\nabla w),\quad F_3:=h\, (v^\mtd)_\tau, \\
  & F_4:= h\, w_\tau,\quad F_5:=\rho(w_\mh\cdot\nabla_\mh v^\mtd),\quad F_6:=h (v^\mtd_\mh \cdot\nabla_\mh v^\mtd).
\end{split}
\end{equation}

We estimate the terms on the right-hand side one by one.
For $F_1$, in view of the interpolation inequality
\begin{equation}\label{InPes2}
  \|\nabla w\|_{L_p(\R^3)}\leq C \|w\|_{\dot B^{2-2/p}_{p,p}(\R^3)}^{\frac{5p-6}{7p-10}} \|w\|_{L_2(\R^3)}^{\frac{2p-4}{7p-10}}
  \leq C\left( \|w\|_{\dot B^{2-2/p}_{p,p}(\R^3)} + \|w\|_{L_2(\R^3)}\right),
\end{equation}
and using \eqref{wL2es}, the fact $\|v^\mtd\|_{L_p(\R^+;L_\infty (\R^2))}\leq C B(v^\mtd_0)$ (from \eqref{fact3} and \eqref{B0}), we have
\begin{equation*}{}
\begin{split}
  \|F_1\|_{L_p(\R^3\times (0,t))}^p 
  & \leq \int_0^t \|\rho(\tau)\|_{L_\infty(\R^3)}^p \|\nabla w(\tau)\|_{L_p(\R^3)}^p \|v^\mtd(\tau)\|_{L_\infty(\R^2)}^p \dd \tau \\
  & \leq C\|\rho_0\|_{L_\infty}^p \int_0^t \|w(\tau)\|_{\dot B^{2-2/p}_{p,p}}^p\|v^\mtd(\tau)\|_{L_\infty}^p\dd \tau
  + C \|\rho_0\|_{L_\infty}^p \|w\|_{L_\infty(0,t; L_2)}^p \int_0^t \|v^\mtd\|_{L_\infty}^p\dd \tau \\
  & \leq C\int_0^t \|w(\tau)\|_{\dot B^{2-2/p}_{p,p}}^p\|v^\mtd(\tau)\|_{L_\infty(\R^2)}^p\dd \tau +
  C \|( w_0,h_0)\|_{L_2}^p e^{2p B(v^\mtd_0)},
\end{split}
\end{equation*}
where in the last inequality we also used $\|\rho_0\|_{L_\infty(\R^3)}\leq 1+\|h_0\|_{L_\infty(\R^3)}\leq 2$.
For $F_2$, if $p> 5$, taking advantage of the following interpolation inequalities
\begin{equation*}
\begin{split}
  \|w\|_{L_\infty(0,t; L_\infty(\R^3))} & \leq C\|w\|_{L_\infty(0,t; L_2)}^{\frac{4p-10}{7p-10}} \|w\|_{L_\infty(0,t;\dot B^{2-2/p}_{p,p})}^{\frac{3p}{7p-10}},\\
  \|\nabla w\|_{L_p(0,t; L_p(\R^3))} & \leq C \|\nabla w\|_{L_2(0,t; L_2)}^{\frac{2}{p}} \|\nabla w\|_{L_\infty(0,t; L_\infty)}^{\frac{p-2}{p}}, \\
  &\leq \,C\|\nabla w\|_{L_2(0,t; L_2)}^{\frac{2}{p}} \|w\|_{L_\infty(0,t; L_2)}^{\frac{(2p-10)(p-2)}{(7p-10)p}} \|w\|_{L_\infty(0,t;\dot B^{2-2/p}_{p,p})}^{\frac{5(p-2)}{7p-10}},
\end{split}
\end{equation*}
we see that
\begin{equation*}
\begin{split}
  \|F_2\|_{L_p(\R^3\times (0,t))}^p & \leq \|\rho_0\|_{L_\infty(\R^3)}^p \|w\|_{L_\infty(0,t; L_\infty(\R^3))}^p \|\nabla w\|_{L_p(0,t; L_p(\R^3))}^p \\
  & \leq C \|w\|_{L_\infty (0,t; \dot B^{2-2/p}_{p,p})}^{\frac{8p-10}{7p-10}p} \| w\|_{L_\infty(0,t; L_2)}^{\frac{6p^2-24p+20}{7p-10}}\|\nabla w\|_{L_2(0,t; L_2)}^2 \\
  & \leq C \|w\|_{L_\infty (0,t; \dot B^{2-2/p}_{p,p})}^{\frac{8p-10}{7p-10}p} \left(  \| w\|_{L_\infty(0,t; L_2)} + \|\nabla w\|_{L_2(0,t; L_2)} \right)^{\frac{6p-10}{7p-10}p}\\
  & \leq C  \|w\|_{L_\infty (0,t; \dot B^{2-2/p}_{p,p})}^{\frac{8p-10}{7p-10}p} \left( \|( w_0,h_0)\|_{L_2} e^{B(v^\mtd_0)} \right)^{\frac{p(6p-10)}{7p-10}};
\end{split}
\end{equation*}
if $p=5$, we similarly get
\begin{equation*}
\begin{split}
  \|F_2\|_{L_5(\R^3\times (0,t))}^5 & \leq \|\rho_0\|_{L_\infty(\R^3)}^p\|w\|_{L_\infty(0,t; L_\infty(\R^3))}^5 \|\nabla w\|_{L_5(0,t; L_5(\R^3))}^5 \\
  & \leq C\|w\|_{L_\infty(0,t; L_2)}^2 \|w\|_{L_\infty(0,t;\dot B^{\frac{8}{5}}_{5,5})}^3
  \|\nabla w\|_{L_2(0,t; L_2)}^2 \|\nabla w\|_{L_\infty(0,t; \dot B^{\frac{3}{5}}_{5,5})}^3 \\
  & \leq C \|w\|_{L_\infty (0,t; \dot B^{8/5}_{5,5}(\R^3))}^6 \left(  \| w\|_{L_\infty(0,t; L_2)} + \|\nabla w\|_{L_2(0,t; L_2)} \right)^4 \\
  & \leq C \|w\|_{L_\infty (0,t; \dot B^{8/5}_{5,5}(\R^3))}^6  \left( \|( w_0,h_0)\|_{L_2}^4 e^{4B(v^\mtd_0)} \right);
\end{split}
\end{equation*}
while if $p\in ]3,5[$, thanks to the interpolation inequalities
\begin{equation*}
\begin{split}
  &\|\nabla w\|_{L_p(0,t; L_p(\R^3))} \leq C \|\nabla w\|_{L_2(0,t; L_2)}^{\frac{2p}{7p-10}}
  \|\nabla w\|_{L_{\frac{5p}{5-p}}(0,t; L_{\frac{5p}{5-p}})}^{\frac{5p-10}{7p-10}} ,\\
  & \|\nabla w\|_{L_{\frac{5p}{5-p}}(0,t; L_{\frac{5p}{5-p}})} \leq C \| w\|_{L_\infty(0,t; \dot B^{2-2/p}_{p,p})}^{\frac{p}{5}} \|\nabla^2 w\|_{L_p(0,t; L_p)}^{\frac{5-p}{5}},
\end{split}
\end{equation*}
we obtain that
\begin{align*}
  & \|F_2\|_{L_p(\R^3\times (0,t))}^p  \\
  \leq\, & \|\rho_0\|_{L_\infty(\R^3)}^p\|w\|_{L_\infty(0,t; L_\infty(\R^3))}^p \|\nabla w\|_{L_p(0,t; L_p(\R^3))}^p \\
  \leq\, & C\|w\|_{L_\infty(0,t; L_2)}^{\frac{4p-10}{7p-10}p} \|w\|_{L_\infty(0,t;\dot B^{2-2/p}_{p,p})}^{\frac{3p}{7p-10}p}
  \|\nabla w\|_{L_2(0,t; L_2)}^{\frac{2p^2}{7p-10}} \|\nabla w\|_{L_{\frac{5p}{5-p}}(0,t; L_{\frac{5p}{5-p}})}^{\frac{5p-10}{7p-10}p} \\
  \leq\, & C \left(  \|w\|_{L_\infty (0,t; \dot B^{2-2/p}_{p,p}(\R^3))}+ \|\nabla^2w\|_{L_p(0,t;L_p)}\right)^{\frac{8p-10}{7p-10}p}
  \left(\| w\|_{L_\infty(0,t; L_2)} + \|\nabla w\|_{L_2(0,t; L_2)} \right)^{\frac{6p-10}{7p-10}p} \\
  \leq \, & C  \left(  \|w\|_{L_\infty (0,t; \dot B^{2-2/p}_{p,p}(\R^3))}+ \|\nabla^2w\|_{L_p(0,t;L_p)}\right)^{\frac{8p-10}{7p-10}p}
  \left( \|( w_0,h_0)\|_{L_2} e^{B(v^\mtd_0)} \right)^{\frac{p(6p-10)}{7p-10}}.
\end{align*}
Thanks to \eqref{v2d-es1} and the continuous embedding $W^1_p(\R^2)\hookrightarrow L_\infty(\R^2)$, we estimate $F_3$ and $F_4$ as follows
\begin{equation*}
\begin{split}
  \|F_3\|_{L_p(\R^3\times (0,t))}^p & \leq \|h\|_{L_\infty(0,t; L_p(\R^3))}^p \|(v^\mtd)_\tau\|_{L_p(0,t; L_\infty(\R^2))}^p \\
  & \leq C\|h_0\|_{L_p(\R^3)}^p  \Big( \|v^\mtd_0\|_{L_2\cap \dot B^{3-2/p}_{p,p}(\R^2)}^{4p}  + 1 \Big)e^{C(1+\|v^\mtd_0\|_{L_2}^4)},
\end{split}
\end{equation*}
and
\begin{equation*}
  \|F_4\|_{L_p(\R^3\times(0,t))}^p \leq \|h_0\|_{L_\infty(\R^3)}^p \|w_\tau\|_{L_p(\R^3\times(0,t))}^p.
\end{equation*}
The treating of $F_5$ is similar to that of $F_1$, and by using the inequality
$\| w\|_{L_p} \leq C \big( \|w\|_{\dot B^{2-2/p}_{p,p}} + \|w\|_{L_2}\big)$,
we have
\begin{equation*}
\begin{split}
  \|F_5\|_{L_p(\R^3\times (0,t))}^p 
  & \leq \|\rho_0\|_{L_\infty(\R^3)} \int_0^t \|w(\tau)\|_{L_p(\R^3)}^p \|\nabla_\mh v^\mtd(\tau)\|_{L_\infty(\R^2)}^p \dd \tau \\
  & \leq C \int_0^t \|w(\tau)\|_{\dot B^{2-2/p}_{p,p}}^p\|\nabla_\mh v^\mtd(\tau)\|_{L_\infty(\R^2)}^p\dd \tau
  + C \|w\|_{L_\infty(0,t; L_2)}^p \int_0^t \|\nabla_\mh v^\mtd(\tau)\|_{L_\infty(\R^2)}^p\dd \tau.
\end{split}
\end{equation*}
By using the following estimate (deduced from \eqref{decHNS2d3} and \eqref{estma0})
\begin{equation}\label{fact6}
\begin{split}
  \|\nabla_\mh v^\mtd\|_{L_p(\R^+; L_\infty)} & \leq  \|v^\mtd\|_{L_p(0,1; L_2\cap \dot W^2_p)} +
  C \|v^\mtd_0\|_{L_2} e^{C(1+\|v^\mtd_0\|_{L_2}^4)} \left(\int_1^\infty \frac{1}{\tau^p} \dd \tau \right)^{1/p}\\
  & \leq C   \Big(\|v^\mtd_0\|_{L_2\cap \dot B^{3-2/p}_{p,p}(\R^2)}^3 +1\Big) e^{C(1+\|v^\mtd_0\|_{L_2}^4)},
\end{split}
\end{equation}
we get (recalling $B(v^\mtd_0)$ is defined by \eqref{B0})
\begin{equation*}
\begin{split}
  & \|F_5\|_{L_p(\R^3\times (0,t))}^p \\
  \leq & C \int_0^t \|w\|_{\dot B^{2-2/p}_{p,p}}^p\|\nabla_\mh v^\mtd\|_{L_\infty(\R^2)}^p\dd \tau
  + C \|( w_0,h_0)\|_{L_2}^p  e^{pB(v^\mtd_0)} \Big( \|v^\mtd_0\|_{L_2\cap \dot B^{3-2/p}_{p,p}}^{3p}  + 1\Big)e^{Cp(1+\|v^\mtd_0\|_{L_2}^4)} \\
  \leq & C \int_0^t \|w(\tau)\|_{\dot B^{2-2/p}_{p,p}}^p\|\nabla_\mh v^\mtd(\tau)\|_{L_\infty(\R^2)}^p\dd \tau
  + C \|( w_0,h_0)\|_{L_2}^p  e^{2 pB(v^\mtd_0)}.
\end{split}
\end{equation*}
For $F_6$, from \eqref{L2es2D}, \eqref{bomgLpes2}, \eqref{fact4} and \eqref{fact6}, we infer that
\begin{equation*}
\begin{split}
  \|F_6\|_{L_p(\R^3\times (0,t))}^p   \leq & \,\|h\|_{L_\infty(0,t; L_p(\R^3))}^p \|v^\mtd_\mh\cdot\nabla_\mh v^\mtd\|_{L_p(0,t; L_\infty(\R^2))}^p \\
  \leq &\, \|h_0\|_{L_p(\R^3)}^p \|v^\mtd_\mh\|_{L_\infty(0,t; L_2\cap \dot W^1_p(\R^2))}^p \|\nabla_\mh v^\mtd\|_{L_p(0,t; L_\infty(\R^2))}^p \\
  \leq & \, \|h_0\|_{L_p(\R^3)}^p  \Big( \|v^\mtd_0\|_{L_2\cap \dot B^{3-2/p}_{p,p}(\R^2)}^{4p} +1 \Big)  e^{C(1+\|v^\mtd_0\|_{L_2}^4)}.
\end{split}
\end{equation*}
Denoting by
\begin{equation}\label{XtYt}
  \mathcal{X}_w(t):=\sup_{\tau\leq t} \|w(\tau)\|_{\dot B^{2-\frac{2}{p}}_{p,p}(\mathbb{R}^3)}^p, \quad
  \mathcal{Y}_w(t):= \|w_\tau,\nabla^2 w,\nabla q\|_{L_p(\R^3\times(0,t))}^p,
\end{equation}
and assuming
\begin{equation}\label{hcd1}
  \|h_0\|_{L_\infty(\R^3)}\leq \left( \frac{1}{2C}\right)^{1/p},
\end{equation}
we collect the above estimates to find that for every $p>3$,
\begin{equation}\label{XYwes}
\begin{split}
  \mathcal{X}_w(t) + \mathcal{Y}_w(t) \leq & \, C\int_0^t \mathcal{X}_w(\tau) \left( \|v^\mtd(\tau)\|_{L_\infty(\R^2)}^p + \|\nabla_\mh v^\mtd(\tau)\|_{L_\infty(\R^2)}^p\right)\dd \tau + \\
  & \, + C_1 \big( \mathcal{X}_w(t) + \mathcal{Y}_w(t) \big)^{\frac{8p-10}{7p-10}}
  \left( \|( w_0,h_0)\|_{L_2} e^{B(v^\mtd_0)} \right)^{\frac{p(6p-10)}{7p-10}} +\\
  & \,+ C \| w_0\|_{\dot B^{2-\frac{2}{p}}_{p,p}(\R^3)}^p +   C \left( \| w_0\|_{L_2(\R^3)}^p  + \|h_0\|_{L_2\cap L_\infty(\R^3)}^p \right) e^{2pB(v^\mtd_0)} .
\end{split}
\end{equation}
We set
\begin{equation}
  T_*: = \sup \Big\{t>0: \mathcal{X}_w(t) + \mathcal{Y}_w(t)\leq (2C_1)^{-\frac{7p-10}{p}}\left( \|( w_0,h_0)\|_{L_2} e^{B(v^\mtd_0)}\right)^{-(6p-10)} \Big\} ,
\end{equation}
which satisfies $T_*>0$ from the local existence part. Then for every $t\leq T_*$, we have
\begin{equation*}
\begin{split}
  \mathcal{X}_w(t) + \mathcal{Y}_w(t) \leq & \, 2 C\int_0^t \mathcal{X}_w(\tau) \left( \|v^\mtd(\tau)\|_{L_\infty(\R^2)}^p + \|\nabla_\mh v^\mtd(\tau)\|_{L_\infty(\R^2)}^p\right)\dd \tau  \\
  & \,+ 2C \| w_0\|_{\dot B^{2-\frac{2}{p}}_{p,p}}^p +  2 C \left( \| w_0\|_{L_2}^p  + \|h_0\|_{L_2\cap L_\infty}^p \right) e^{2 p B(v^\mtd_0)} .
\end{split}
\end{equation*}
Gr\"onwall's inequality and \eqref{fact3}, \eqref{fact4}, \eqref{fact6} lead to
\begin{equation}
\begin{split}
  & \mathcal{X}_w(t) + \mathcal{Y}_w(t)  \\
  \leq\,& 2C \Big( \| w_0\|_{\dot B^{2-\frac{2}{p}}_{p,p}}^p + \left( \| w_0\|_{L_2}^p  + \|h_0\|_{L_2\cap L_\infty}^p \right) e^{2 pB(v^\mtd_0)}\Big)
  e^{2C \left(\|v^\mtd\|_{L_p(0,t; L_\infty)}^p + \|\nabla_\mh v^\mtd\|_{L_p(0,t; L_\infty)}^p \right)} \\
  \leq \, & 2 C\Big( \| w_0\|_{L_2\cap \dot B^{2-2/p}_{p,p}}^p  + \|h_0\|_{L_2\cap L_\infty}^p \Big) e^{p B_1(v^\mtd_0)} ,
\end{split}
\end{equation}
with
\begin{equation}
\begin{split}\label{B1}
  B_1(v^\mtd_0) :=& \,2 B(v^\mtd_0) + C \Big( \|v^\mtd_0\|_{L_2\cap \dot B^{3-2/p}_{p,p}(\R^2)}^{3p}  + 1\Big)e^{C(1+\|v^\mtd_0\|_{L_2}^4)} \\
  \leq & C' \Big( \|v^\mtd_0\|_{L_2\cap \dot B^{3-2/p}_{p,p}(\R^2)}^{4p}  + 1\Big)e^{C'(1+\|v^\mtd_0\|_{L_2}^4)} .
\end{split}
\end{equation}
Hence, if $( w_0,h_0)$ are small enough so that
\begin{equation*}
  2 C\Big( \| w_0\|_{L_2\cap \dot B^{2-2/p}_{p,p}}^p  + \|h_0\|_{L_2\cap L_\infty}^p \Big) e^{ p B_1(v^\mtd_0)} \leq \frac{1}{2} (2C_1)^{-\frac{7p-10}{p}} \|( w_0,h_0)\|_{L_2}^{-(6p-10)} e^{-(6p-10)B(v^\mtd_0)},
\end{equation*}
equivalently, if $(w_0,h_0)$ satisfies that
\begin{equation}\label{}
  \| w_0\|_{L_2\cap \dot B^{2-2/p}_{p,p}(\R^3)} + \|h_0\|_{L_2\cap L_\infty(\R^3)} \leq (4C)^{-\frac{1}{7p-10}} (2C_1)^{-\frac{1}{p}}  e^{-B_1(v^\mtd_0)},
\end{equation}
the bootstrapping method guarantees that $T_*=T$, and we have that for all $t\in [0,T]$,
\begin{equation}\label{XYest1}
  \mathcal{X}_w(t) + \mathcal{Y}_w(t)\leq 2 C\Big( \| w_0\|_{L_2\cap \dot B^{2-2/p}_{p,p}(\R^3)}^p  + \|h_0\|_{L_2\cap L_\infty(\R^3)}^p \Big) e^{p B_1(v^\mtd_0)}.
\end{equation}
Note that to construct (\ref{XYest1}) we did not use any information about regularity of the gradient of density.
\end{proof}

If the initial density has some gradient regularity, we moreover have the following regularity estimate on the density.
\begin{proposition}\label{prop:apR3}
  Under the assumptions of Proposition \ref{prop:apR2} and additionally assume that $\nabla h_0\in L_3(\R^3)$, then we have
\begin{equation}\label{nh-es}
  \sup_{t\in [0,T]} \|\nabla h(t)\|_{L_3(\mathbb{R}^3)} \leq  \|\nabla h_0\|_{L_3(\R^3)} e^{(1+ T)C(h_0,w_0,v^\mtd_0)},
\end{equation}
where $C(h_0,w_0,v^\mtd_0)$ is depending only on the initial data $h_0,w_0, v^\mtd_0$.
\end{proposition}

\begin{proof}[Proof of Proposition \ref{prop:apR3}]
From the first equation of \eqref{pertS1}, we see that
\begin{equation*}
  \partial_t (\nabla h) + v\cdot\nabla (\nabla h) =- (\nabla v)\cdot \nabla h.
\end{equation*}
By making the scalar product of both sides of the above equation with $((\partial_1 h)^2,(\partial_2 h)^2,(\partial_3 h)^2)$, and integrating on the spatial variables, we get
\begin{equation*}
 \frac{\dd}{\dd t} \|\nabla h\|_{L_3(\R^3)} \leq \|\nabla v\|_{L_\infty} \|\nabla h\|_{L_3(\R^3)}.
\end{equation*}
Integrating on the time interval $[0,T]$ leads to
\begin{equation}\label{eq:nhL3}
 \sup_{0\leq t\leq T} \|\nabla h\|_{L_3(\R^3)}(t) \leq \|\nabla h_0\|_{L_3} \exp \int_0^T \| \nabla v(t)\|_{L_\infty(\R^3)} \dd t.
\end{equation}
Thanks to the Sobolev embedding $\dot W^2_p\cap L_2(\R^n)\hookrightarrow \dot W^1_\infty(\R^n)$ ($n=2,3$) and the a priori estimates \eqref{L2es2D}, \eqref{estma0}, \eqref{wL2es} and \eqref{w-es}, we find
\begin{equation}\label{vL1Lip}
\begin{split}
  & \|\nabla v\|_{L_1(0,T; L_\infty(\R^3)} \leq \|w\|_{L_1(0,T;\dot W^1_\infty(\R^3))} + \| v^\mtd\|_{L_1(0,T; \dot W^1_\infty(\R^2))}  \\
  \leq\, & T\|w\|_{L_\infty(0,T;L_2(\R^3))} + T^{1-\frac{1}{p}} \|w\|_{L_p(0,T; \dot W^2_p)} +T\|v^\mtd\|_{L_\infty(0,T;L_2(\R^2))} + T^{1-\frac{1}{p}} \|v^\mtd\|_{L_p(0,T; \dot W^2_p(\R^2))}  \\
  \leq \, & (1+T)C(h_0,w_0,v^\mtd_0),
\end{split}
\end{equation}
which combined with \eqref{eq:nhL3} yields the desired inequality \eqref{nh-es}.
\end{proof}

\subsection{Global existence and uniqueness}\label{subsec:exiR}

\smallskip

We divide the whole proof of Theorem \ref{thm:regu} into five steps.

\smallskip

{\bf Step 1: Approximate system and uniform estimates.}

Let $v^\mtd$ be the unique strong solution to the three-component 2D (HNS) system \eqref{HNS} associated with $v^\mtd_0\in L_2\cap \dot B^{3-2/p}_{p,p}(\R^2)$,
then it is smooth for all $t>0$ and satisfies the estimates stated in Lemmas \ref{lem:2DNS-L2} - \ref{lem:2DNS-2}.
We construct $(w^{n+1},h^{n+1})$ ($n\in\N$) as the solutions to the following approximate system
\begin{equation}\label{appPertS}
 \begin{cases}
  h_t^{n+1}+v^n \cdot \nabla h^{n+1} =0, \\
  w_t^{n+1} + v^n \cdot \nabla w^{n+1} -\Delta w^{n+1} +\nabla p^{n+1} = -h^n (w^{n+1}_t +v^{n-1}\cdot\nabla w^{n+1})\\
  \qquad\qquad \qquad\qquad\qquad\qquad\qquad\qquad\qquad -(1+h^n)(w^{n+1}\cdot\nabla v^\mtd) + f(h^n,v^\mtd), \\
  \divg w^{n+1}=0, \\
  h^{n+1}|_{t=0}=h_0,\quad w^{n+1}|_{t=0}=w_0,
 \end{cases}
\end{equation}
with $v^n=v^\mtd + w^n$ and
\begin{equation}\label{eq:fhv}
  f(h^n,v^\mtd)=-h^n (v^\mtd)_t -h^n (v^\mtd\cdot\nabla v^\mtd).
\end{equation}
We also set $v^{-1}(t,x)\equiv 0$, $w^0(t,x)\equiv w_0(x)$, $h^0(t,x)\equiv h_0(x)$, $v^0(t,x)=v_0(x)=w_0+v^\mtd_0$.
The solvability of system (\ref{appPertS}) follows from Lemma \ref{lem:StokesR} and the Banach fixed point theorem. We treat the nonlinearity $v^n \cdot \nabla w^{n+1}$ as
a perturbation and find a solution via a contraction map for small time intervals. Solvability of the transport equation follows directly from the method of characteristics.
We omit explanation of this part, since one can find there no obstacles.

First we have $(h^1,w^1)$ solves
\begin{equation}
 \begin{cases}
  h_t^1+v_0 \cdot \nabla h^1 =0, \\
  w_t^1 + v_0 \cdot \nabla w^1 -\Delta w^1 +\nabla p^1 = -h_0 w^1_t -(1+h_0)(w^1_\mh\cdot\nabla_\mh v^\mtd) -h_0 (v^\mtd)_t -h_0 (v^\mtd_\mh\cdot\nabla_\mh v^\mtd), \\
  \divg w^1=0, \\
  h^1|_{t=0}=h_0,\quad w^1|_{t=0}=w_0.
 \end{cases}
\end{equation}
We see that $\|h^1(t)\|_{L_2\cap L_\infty(\R^3)}= \|h_0\|_{L_2\cap L_\infty(\R^3)}$ for any $t>0$, and by arguing as \eqref{wes-key}, we find
\begin{equation*}{}
\begin{split}
  &\| w^1(t)\|_{L_2(\R^3)}^2 + \int_0^t \|\nabla  w^1(\tau)\|_{L_2(\R^3)}^2\,\dd\tau \\
  \leq\, &(1+\|h_0\|_{L_\infty})\| w_0\|_{L_2}^2 + \|h_0\|_{L_\infty}\|w^1(t)\|_{L_2}^2 +\int_0^t \left( \|\partial_\tau v^\mtd(\tau)\|_{L_\infty}
  +  \|\nabla_\mh v^\mtd(\tau)\|_{L_\infty}^2 \right) \| w^1(\tau)\|_{L_2}^2 \dd\tau \\
  & + C(1+ \|h_0\|_{L_\infty}) \int_0^t\|\nabla_\mh v^\mtd(\tau)\|_{L_2}^2 \| w^1(\tau)\|_{L_2}^2\dd\tau + \|h_0\|_{L_2}^2 \int_0^t \left(  \|\partial_\tau v^\mtd(\tau)\|_{L_\infty}
  +\|v^\mtd_\mh(\tau)\|_{L_\infty}^2\right) \dd \tau,
\end{split}
\end{equation*}
which under the condition $\|h_0\|_{L_\infty(\R^3)}\leq \frac{1}{2}$ leads to that $w^1\in L_\infty(0,\infty;L_2(\R^3))\cap L_2(0,\infty; \dot W^1_2(\R^3))$ satisfying
\begin{equation}\label{w1ene}
  \| w^1(t)\|_{L_2(\R^3)} + \left(\int_0^t \|\nabla  w^1(\tau)\|_{L_2(\R^3)}^2\,\dd\tau \right)^{1/2}\leq C \|( w_0,h_0)\|_{L_2(\R^3)} e^{B(v^\mtd_0)}, \quad \forall t>0.
\end{equation}
Taking advantage of \eqref{Lp-est2} in Lemma \ref{lem:StokesR}, we have
\begin{equation*}\label{}
\begin{split}
  \sup_{\tau\leq t} \|w^1(\tau)\|_{\dot B^{2-\frac{2}{p}}_{p,p}(\mathbb{R}^3)}^p + \|w^1_\tau,\nabla^2 w^1\|_{L_p(\R^3\times(0,t))}^p
  \leq \;  C \bigg( \sum_{i=1}^6\|G_i^1\|_{L_p(\mathbb{R}^3\times(0,t))}^p   +  \|w_0\|_{B^{2-2/p}_{p,p}(\mathbb{R}^3)}^p \bigg),
\end{split}
\end{equation*}
with
\begin{equation*}
\begin{split}
  & G^1_1:= v^\mtd\cdot\nabla w^1, \quad G^1_2:=w_0\cdot\nabla w^1,\quad G^1_3:=h_0\, (v^\mtd)_\tau, \\
  & G^1_4:= h_0\, w_\tau,\quad G^1_5:=(1+h_0)(w^1_\mh\cdot\nabla_\mh v^\mtd),\quad G^1_6:= h_0 (v^\mtd_\mh\cdot\nabla_\mh v^\mtd).
\end{split}
\end{equation*}
Similarly as estimating \eqref{XYwes}, and from the condition $\|h_0\|_{L_\infty(\R^3)}\leq \left( \frac{1}{2C}\right)^{1/p}$ and the following estimate (as the treating of $F_2$ in Proposition \ref{prop:apR2})
\begin{equation*}
\begin{split}
  & \|G^1_2\|_{L_p(\R^3\times (0,t))}^p \leq \|w_0\|_{L_\infty(\R^3)}^p \|\nabla w^1\|_{L_p(\R^3\times (0,t))}^p \\
  \leq \,&C \|w_0\|_{L_\infty}^p \big( \mathcal{X}_{w^1}(t) + \mathcal{Y}_{w^1}(t) \big)^{\frac{5p-10}{7p-10}}
  \left(\| w^1\|_{L_\infty(0,t; L_2)} + \|\nabla w^1\|_{L_2(0,t; L_2)} \right)^{\frac{2p^2}{7p-10}} \\
  \leq \,& C\|w_0\|_{L_\infty}^p \big( \mathcal{X}_{w^1}(t) + \mathcal{Y}_{w^1}(t) \big)^{\frac{5p-10}{7p-10}}
  \left( \|( w_0,h_0)\|_{L_2} e^{B(v^\mtd_0)} \right)^{\frac{2p^2}{7p-10}},
\end{split}
\end{equation*}
we obtain
\begin{align*}
  \mathcal{X}_{w^1}(t) + \mathcal{Y}_{w^1}(t) \leq
  & \, C\int_0^t \mathcal{X}_{w^1}(\tau) \left( \|v^\mtd(\tau)\|_{L_\infty(\R^2)}^p + \|\nabla_\mh v^\mtd(\tau)\|_{L_\infty(\R^2)}^p\right)\dd \tau + \\
  & \, + C\|w_0\|_{L_\infty}^p \big( \mathcal{X}_{w^1}(t) + \mathcal{Y}_{w^1}(t) \big)^{\frac{5p-10}{7p-10}}
  \left( \|( w_0,h_0)\|_{L_2} e^{B(v^\mtd_0)} \right)^{\frac{2p^2}{7p-10}} +\\
  & \,+ C \| w_0\|_{\dot B^{2-\frac{2}{p}}_{p,p}}^p +   C \left( \| w_0\|_{L_2}^p  + \|h_0\|_{L_2\cap L_\infty}^p \right) e^{2pB(v^\mtd_0)} \\
  \leq & \,C\int_0^t \mathcal{X}_{w^1}(\tau) \left( \|v^\mtd(\tau)\|_{L_\infty(\R^2)}^p + \|\nabla_\mh v^\mtd(\tau)\|_{L_\infty(\R^2)}^p\right)\dd \tau
  + \frac{1}{2} \big( \mathcal{X}_{w^1}(t) + \mathcal{Y}_{w^1}(t) \big)\\
  & \, + C \Big(1+\| w_0\|_{L_\infty}^{\frac{7p-10}{2}}\Big) \Big( \| w_0\|_{L_2\cap \dot B^{2-2/p}_{p,p}}^p
  + \|h_0\|_{L_2\cap L_\infty}^p \Big) e^{2pB(v^\mtd_0)},
\end{align*}
with $\mathcal{X}_{w^1}$ and $\mathcal{Y}_{w^1}$ given by \eqref{XtYt}.
Hence, by assuming $\|w_0\|_{L_\infty(\R^3)}\leq 1$, we use Gr\"onwall's inequality to deduce that
\begin{equation}\label{w1ape}
\begin{split}
  & \sup_{t<\infty} \|w^1(t)\|_{\dot B^{2-2/p}_{p,p}(\mathbb{R}^3)}^p + \|w^1_t,\nabla^2 w^1\|_{L_p(\R^3\times(0,\infty))}^p\\
  \leq\,& 2C \Big( \| w_0\|_{L_2\cap \dot B^{2- 2/p}_{p,p}(\R^3)}^p
  + \|h_0\|_{L_2\cap L_\infty(\R^3)}^p \Big) \exp\set{C'\Big( \|v^\mtd_0\|_{L_2\cap \dot B^{3-2/p}_{p,p}(\R^2)}^{4p}  + 1\Big) e^{C'(1+\|v^\mtd_0\|_{L_2}^4)}}.
\end{split}
\end{equation}

Now under the condition \eqref{wh-cd}, that is, there is an absolute small constant $c_*>0$ so that
\begin{equation}\label{wh-cd2}
  \left(\| w_0\|_{L_2\cap \dot B^{2-2/p}_{p,p}(\R^3)} + \|h_0\|_{L_2\cap L_\infty(\R^3)} \right)
  \exp\left\{C' \left(\|v^\mtd_0\|_{L_2\cap \dot B^{3-2/p}_{p,p}}^{4p} +1\right)e^{C'(1+\|v^\mtd_0\|_{L_2}^4)}\right\}\leq c_*,
\end{equation}
we suppose that for each $n\in\N_+$ and $k\leq n$ we have
\begin{equation}\label{wk-ene}
  \| w^k(t)\|_{L_2(\R^3)} + \left(\int_0^t \|\nabla  w^k(\tau)\|_{L_2(\R^3)}^2\,\dd\tau \right)^{1/2}\leq C \|( w_0,h_0)\|_{L_2(\R^3)} e^{B(v^\mtd_0)}, \quad \forall t>0,
\end{equation}
and
\begin{equation}\label{wk-ape}
\begin{split}
  & \sup_{t<\infty} \|w^k(t)\|_{\dot B^{2-\frac{2}{p}}_{p,p}(\mathbb{R}^3)}^p + \|w^k_t,\nabla^2 w^k\|_{L_p(\R^3\times(0,\infty))}^p\\
  \leq\,& 2C \Big( \| w_0\|_{L_2\cap \dot B^{2-\frac{2}{p}}_{p,p}(\R^3)}^p
  + \|h_0\|_{L_2\cap L_\infty(\R^3)}^p \Big) \exp\set{C'\left( \|v^\mtd_0\|_{L_2\cap \dot B^{3-2/p}_{p,p}(\R^2)}^{4p}  + 1\right)e^{C'(1+\|v^\mtd_0\|_{L_2}^4)}},
\end{split}
\end{equation}
which in terms of the notation \eqref{XtYt} means that
\begin{equation}\label{wk-ape2}
  \mathcal{X}_{w^k}(t) + \mathcal{Y}_{w^k}(t) \leq 2 C c_*^p,\quad \forall t>0.
\end{equation}
We intend to derive the similar uniform estimates for $w^{n+1}$.
Since $v^n=w^n+v^\mtd$, by arguing as \eqref{vL1Lip}, we have $v^n\in L_1(0,t; \dot W^1_\infty(\R^3))$ for $t>0$ arbitrary. Thus the flow property of transport equation guarantees that
$\|h^{n+1}(t)\|_{L_2\cap L_\infty(\R^3)}=\|h_0\|_{L_2\cap L_\infty(\R^3)}$ for any $t>0$. For the system \eqref{appPertS}, in a similar way as obtaining \eqref{wes-key}, and according to the following identity formula
\begin{equation*}
\begin{split}
  & -2\int_0^t\int_{\R^3}h^n ( w^{n+1})_\tau \, \cdot w^{n+1}\dd x\mathrm{d}\tau -2 \int_0^t\int_{\R^3} h^n (v^{n-1}\cdot\nabla) w^{n+1}\, \cdot w^{n+1}\dd x\mathrm{d}\tau\\
  & =\int_{\R^3}h_0(x) |w_0(x)|^2 \dd x - \int_{\R^3} h(t,x) |w(t,x)|^2\dd x,
\end{split}
\end{equation*}
we get
\begin{align*}
  &\| w^{n+1}(t)\|_{L_2(\R^3)}^2 + \int_0^t \|\nabla  w^{n+1}(\tau)\|_{L_2(\R^3)}^2\,\dd\tau \\
  \leq\, &(1+\|h_0\|_{L_\infty})\| w_0\|_{L_2}^2 + \|h_0\|_{L_\infty}\|w(t)\|_{L_2}^2 +\int_0^t \left( \|\partial_\tau v^\mtd(\tau)\|_{L_\infty}
  +  \|\nabla_\mh v^\mtd(\tau)\|_{L_\infty}^2 \right) \| w^{n+1}(\tau)\|_{L_2}^2 \dd\tau \\
  & + C(1+ \|h_0\|_{L_\infty}) \int_0^t\|\nabla_\mh v^\mtd(\tau)\|_{L_2(\R^2)}^2 \| w^{n+1}(\tau)\|_{L_2}^2\dd\tau + \|h_0\|_{L_2}^2 \int_0^t \left(  \|\partial_\tau v^\mtd\|_{L_\infty}
  +\|v^\mtd_\mh\|_{L_\infty}^2\right) \dd \tau,
\end{align*}
which ensures that \eqref{wk-ene} holds with $k=n+1$ by using Gr\"onwall's inequality. Next we apply Lemma \ref{lem:StokesR} to the second equation of \eqref{appPertS} to see that
\begin{equation*}
\begin{split}
  \sup_{\tau\leq t} \|w^{n+1}\|_{\dot B^{2-\frac{2}{p}}_{p,p}}^p + \|w^{n+1}_\tau,\nabla^2 w^{n+1}\|_{L_p(\R^3\times(0,t))}^p
  \leq \;  C \bigg( \sum_{i=1}^6\|G^{n+1}_i\|_{L_p(\mathbb{R}^3\times(0,t))}^p   +  \|w_0\|_{B^{2-2/p}_{p,p}}^p \bigg),
\end{split}
\end{equation*}
with
\begin{equation*}
\begin{split}
  & G^{n+1}_1:=(1+h^n)(v^\mtd\cdot\nabla w^{n+1}), \quad G^{n+1}_2:=(w^n + h^n w^{n-1})\cdot\nabla w^{n+1},\quad G^{n+1}_3:=h^n\, (v^\mtd)_\tau, \\
  & G^{n+1}_4:= h^n\, w^{n+1}_\tau,\quad G^{n+1}_5:=(1+h^n)(w^{n+1}_\mh\cdot\nabla_\mh v^\mtd),\quad G^{n+1}_6:=h^n (v^\mtd_\mh\cdot\nabla_\mh v^\mtd).
\end{split}
\end{equation*}
By estimating as \eqref{XYwes}, and noting that (similar to the treating of $F_2$ in Proposition \ref{prop:apR2} and using \eqref{wk-ene} for $k=n-1,n,n+1$)
\begin{equation*}
\begin{split}
  & \|G^{n+1}_2\|_{L_p(\R^3\times (0,t))}^p \leq \left(\|w^n\|_{L_\infty(\R^3\times (0,t))} + \|h_0\|_{L_\infty}\|w^{n-1}\|_{L_\infty(\R^3\times (0,t))}\right) \|\nabla w^{n+1}\|_{L_p(\R^3\times (0,t))}^p \\
  \leq \,&C \Big(\sum_{k=n-1,n} \mathcal{X}_{w^k}(t) + \mathcal{Y}_{w^k}(t)\Big)^{\frac{3p}{7p-10}}\big( \mathcal{X}_{w^{n+1}}(t) + \mathcal{Y}_{w^{n+1}}(t) \big)^{\frac{5p-10}{7p-10}}
  \left( \|( w_0,h_0)\|_{L_2} e^{B(v^\mtd_0)} \right)^{\frac{p(6p-10)}{7p-10}},
\end{split}
\end{equation*}
we infer that by Young's inequality,
\begin{equation*}\label{}
\begin{split}
  & \mathcal{X}_{w^{n+1}}(t) + \mathcal{Y}_{w^{n+1}}(t) \\
  \leq &\, C\int_0^t \mathcal{X}_{w^{n+1}}(\tau) \left( \|v^\mtd(\tau)\|_{L_\infty(\R^2)}^p + \|\nabla_\mh v^\mtd(\tau)\|_{L_\infty(\R^2)}^p\right)\dd \tau + \\
  & \, + C \Big(\sum_{k=n-1,n} \mathcal{X}_{w^k}(t) + \mathcal{Y}_{w^k}(t)\Big)^{\frac{3p}{7p-10}}\big( \mathcal{X}_{w^{n+1}}(t) + \mathcal{Y}_{w^{n+1}}(t) \big)^{\frac{5p-10}{7p-10}}
  \left( \|( w_0,h_0)\|_{L_2} e^{B(v^\mtd_0)} \right)^{\frac{p(6p-10)}{7p-10}} +\\
  & \,+ C \| w_0\|_{\dot B^{2-\frac{2}{p}}_{p,p}}^p +   C \left( \| w_0\|_{L_2}^p  + \|h_0\|_{L_2\cap L_\infty}^p \right) e^{2pB(v^\mtd_0)} \\
  \leq &\, C\int_0^t \mathcal{X}_{w^{n+1}}(\tau) \left( \|v^\mtd(\tau)\|_{L_\infty(\R^2)}^p + \|\nabla v^\mtd(\tau)\|_{L_\infty(\R^2)}^p\right)\dd \tau + \frac{1}{3}\big( \mathcal{X}_{w^{n+1}}(t) + \mathcal{Y}_{w^{n+1}}(t) \big) \\
  & \, + C \Big(\sum_{k=n-1,n} \mathcal{X}_{w^k}(t) + \mathcal{Y}_{w^k}(t)\Big)^{\frac{3}{2}}  \left( \|( w_0,h_0)\|_{L_2} e^{B(v^\mtd_0)} \right)^{3p-5} + C \Big( \| w_0\|_{L_2\cap \dot B^{2-\frac{2}{p}}_{p,p}}^p  + \|h_0\|_{L_2\cap L_\infty}^p \Big) e^{2pB(v^\mtd_0)},
\end{split}
\end{equation*}
with $\mathcal{X}_{w^{n+1}}$ and $\mathcal{Y}_{w^{n+1}}$ given by \eqref{XtYt}.
In view of Gr\"onwall's inequality and the assumptions \eqref{wh-cd2}, \eqref{wk-ape}-\eqref{wk-ape2}, by letting $c_*$ be suitably small so that $(4C)^{3/2} c_*^{7p/2-5}\leq \frac{1}{3}$, we deduce
\begin{equation*}
\begin{split}
  & \mathcal{X}_{w^{n+1}}(t) + \mathcal{Y}_{w^{n+1}}(t) \\
  \leq & \,  \frac{3 C}{2} \Big(\sum_{k=n-1}^n \mathcal{X}_{w^k}(t) + \mathcal{Y}_{w^k}(t)\Big)^{\frac{3}{2}}  \|( w_0,h_0)\|_{L_2}^{3p-5} e^{(3p-5)B_1(v^\mtd_0)}  \\
  & \,+ \frac{3 C}{2}  \Big( \| w_0\|_{L_2\cap \dot B^{2-2/p}_{p,p}}^p  + \|h_0\|_{L_2\cap L_\infty}^p \Big) e^{pB_1(v^\mtd_0)} \\
  \leq & \,\frac{3 C}{2} \Big( \| w_0\|_{L_2\cap \dot B^{2-\frac{2}{p}}_{p,p}}^p + \|h_0\|_{L_2\cap L_\infty}^p \Big) e^{pB_1(v^\mtd_0)}\left( (4C)^{3/2} c_*^{7p/2-5}+1\right) \\
  \leq &\,2C \left( \| w_0\|_{L_2\cap \dot B^{2-2/p}_{p,p}}^p
  + \|h_0\|_{L_2\cap L_\infty}^p \right) \exp\set{C'\left( \|v^\mtd_0\|_{L_2\cap \dot B^{3-2/p}_{p,p}(\R^2)}^{2p}  + 1\right)e^{C'(1+\|v^\mtd_0\|_{L_2}^4)}} .
\end{split}
\end{equation*}

Therefore, the induction method guarantees that the uniform estimates of \eqref{wk-ene} and \eqref{wk-ape} indeed hold for every $k\in\N$ under the smallness condition \eqref{wh-cd2}.
Moreover, for any $T>0$, by arguing as Proposition \ref{prop:apR3} and using the uniform estimates \eqref{wk-ene} and \eqref{wk-ape}, we deduce that
\begin{equation}\label{nhn-es}
  \sup_{t\in [0,T]} \|\nabla h^{n+1}(t)\|_{L_3(\mathbb{R}^3)} \leq C \|\nabla h_0\|_{L_3(\R^3)} e^{(1+ T)C(h_0,w_0,v^\mtd_0)},\quad \forall n\in\N.
\end{equation}

\smallskip

{\bf Step 2: $L_2$-contraction of $\{(h^n,w^n)\}_{n\in\N}$} on a small interval $[0,T_0]$.

Now based on the uniform estimates \eqref{wk-ene}, \eqref{wk-ape} and \eqref{nhn-es}, we show that $\{(h^n,w^n)\}_{n\in\N}$ is a Cauchy sequence in the $L_2$-energy space on a small interval
$[0,T_0]$ with $T_0>0$.
Denoting by
\begin{equation}
  \delta h^n= h^n-h^{n-1},\quad \delta w^n = w^n -w^{n-1},\quad \delta p^n= p^n-p^{n-1},\quad n\in \N,
\end{equation}
with the convention $h^{-1}=p^{-1}=0$ and $w^{-1}=v^{-1}=0$, from \eqref{appPertS}, we write the equations of $(\delta h^{n+1},\delta w^{n+1})$ as
\begin{equation}\label{PertSdif}
 \begin{cases}
  (\delta h^{n+1})_t+v^n \cdot\nabla \delta h^{n+1} = -\delta w^n\cdot \nabla h^n, \\
  (\delta w^{n+1})_t + v^n \cdot \nabla \delta w^{n+1} -\Delta \delta w^{n+1} +\nabla \delta p^{n+1} = H, \\
  \divg \delta w^{n+1}=0, \\
  \delta h^{n+1}|_{t=0}=0,\quad \delta w^{n+1}|_{t=0}=0,
 \end{cases}
\end{equation}
where $H=\sum_{i=1}^{10}H_i$ with
\begin{equation*}\label{}
\begin{split}
  & H_1:=-\delta w^n\cdot\nabla w^n, \;\; H_2:=-\delta h^n \,w^{n+1}_t,\;\; H_3:=-h^{n-1}\, (\delta w^{n+1})_t,
  \;\, H_4:= -\delta h^n (v^{n-1}\cdot\nabla w^{n+1}), \\
  & H_5:=-h^{n-1}\,(\delta w^{n-1}\cdot\nabla w^{n+1}),\quad H_6:=-h^{n-1} (v^{n-2} \cdot\nabla \delta w^{n+1}),
  \quad H_7:=-\delta h^n ( w^{n+1}_\mh\cdot\nabla_\mh v^\mtd), \\
  & H_8:=-(1+h^{n-1})\delta w^{n+1}_\mh\cdot\nabla_\mh v^\mtd,\quad H_9:= -\delta h^n (v^\mtd_\mh\cdot\nabla_\mh v^\mtd), \quad H_{10}:=-\delta h^n (v^\mtd)_t.
\end{split}
\end{equation*}
First we consider the $L_2$-estimate of $\delta h^{n+1}$. Multiplying both sides of the first equation of \eqref{PertSdif} with $\delta h^{n+1}$ and integrating over the spatial variables, we get
\begin{equation*}
  \frac{1}{2}\frac{\dd}{\dd t} \|\delta h^{n+1}(t)\|_{L_2}^2 \leq \int_{\R^3} |\delta w^n ||\nabla h^n||\delta h^{n+1}(t,x)|\dd x .
\end{equation*}
H\"older's inequality leads to
\begin{equation*}
 \frac{\dd}{\dd t} \|\delta h^{n+1}(t)\|_{L_2} \leq \|\delta w^n(t)\|_{L_6} \|\nabla h^n(t)\|_{L_3}.
\end{equation*}
By virtue of integration on time over $[0,t]$ and H\"older's inequality, we obtain
\begin{equation}\label{dhnL2}
  \|\delta h^{n+1}(t)\|_{L_2}\leq \|\nabla h^n\|_{L_\infty(0,t;L_3)} t^{1/2}\left(\int_0^t\int_{\R^3} |\nabla \delta w^n|^2\dd x \dd \tau\right)^{1/2}.
\end{equation}
Next we turn to the $L_2$-estimate of $\delta w^{n+1}$. Through taking the inner product of the second equation of \eqref{PertSdif} with $\delta w^{n+1}$, we have
\begin{equation}\label{dwn-es1}
  \frac{1}{2}\frac{\dd}{\dd t}\|\delta w^{n+1}(t)\|_{L_2}^2 + \|\nabla \delta w^{n+1}(t)\|_{L_2}^2 = \sum_{i=1}^{10}\int_{\R^3} H_i \, \delta w^{n+1}(t,x)\dd x .
\end{equation}
Thanks to H\"older's inequality, the interpolation inequality and Young's inequality, we respectively estimate the terms on the right-hand side of \eqref{dwn-es1} (except the terms containing $H_3$, $H_6$) as
\begin{equation*}
  \int_{\R^3} |H_1| \,|\delta w^{n+1}| \dd x\leq \|\delta w^n\|_{L_2} \|\nabla w^n\|_{L_\infty} \|\delta w^{n+1}\|_{L_2}
  \leq \frac{1}{2} \|\nabla w^n\|_{L_\infty}^2 \|\delta w^n\|_{L_2}^2 +\frac{1}{2}\|\delta w^{n+1}\|_{L_2}^2,
\end{equation*}
and
\begin{align*}
  \int_{\R^3} |H_2| \,|\delta w^{n+1}| \dd x & \leq \|\delta h^n\|_{L_2} \| w^{n+1}_t\|_{L_p} \|\delta w^{n+1}\|_{L_{\frac{2p}{p-2}}} \\
  & \leq C \|\delta h^n\|_{L_2} \| w^{n+1}_t\|_{L_p}  \|\delta w^{n+1}\|_{L_2}^{\frac{p-3}{p}} \|\nabla\delta w^{n+1}\|_{L_2}^{\frac{3}{p}} \\
  & \leq C \|\delta h^n\|_{L_2}^{\frac{2p}{2p-3}} \|w^{n+1}_t\|_{L_p}^{\frac{2p}{2p-3}} \|\delta w^{n+1}\|_{L_2}^{\frac{2p-6}{2p-3}}+ \frac{1}{2}\|\nabla\delta w^{n+1}\|_{L_2}^2 \\
  & \leq C \|w^{n+1}_t\|_{L_p}^2\|\delta h^n\|_{L_2}^2  + \frac{1}{2} \|\delta w^{n+1}\|_{L_2}^2 + \frac{1}{2}\|\nabla\delta w^{n+1}\|_{L_2}^2,
\end{align*}
and
\begin{equation*}
\begin{split}
  \int_{\R^3} |H_4|\, |\delta w^{n+1}|\dd x  &\,\leq \|\delta h^n\|_{L_2} \|v^{n-1}\|_{L_\infty} \|\nabla w^{n+1}\|_{L_\infty} \|\delta w^{n+1}\|_{L_2} \\
  &\,\leq \frac{1}{2} \|v^{n-1}\|_{L_\infty}^2 \|\nabla w^{n+1}\|_{L_\infty}^2 \|\delta h^n\|_{L_2}^2 +  \frac{1}{2}\|\delta w^{n+1}\|_{L_2}^2,
\end{split}
\end{equation*}
and
\begin{equation*}
\begin{split}
  \int_{\R^3} |H_5|\, |\delta w^{n+1}|\dd x  & \leq \|h^{n-1}\|_{L_\infty} \|\delta w^{n-1}\|_{L_2} \|\nabla w^{n+1}\|_{L_\infty} \|\delta w^{n+1}\|_{L_2} \\
  & \leq \frac{1}{2}\|h_0\|_{L_\infty}^2  \|\nabla w^{n+1}\|_{L_\infty}^2 \|\delta w^{n-1}\|_{L_2}^2 + \frac{1}{2} \|\delta w^{n+1}\|_{L_2}^2,
\end{split}
\end{equation*}
and
\begin{equation*}
\begin{split}
  \int_{\R^3} |H_7|\, |\delta w^{n+1}|\dd x  & \leq \|\delta h^n\|_{L_2} \|w^{n+1}\|_{L_\infty} \|\nabla_\mh v^\mtd\|_{L_\infty} \|\delta w^{n+1}\|_{L_2} \\
  & \leq \frac{1}{2}\|w^{n+1}\|_{L_\infty}^2 \|\nabla_\mh v^\mtd\|_{L_\infty}^2 \|\delta h^n\|_{L_2}^2 +\frac{1}{2} \|\delta w^{n+1}\|_{L_2}^2,
\end{split}
\end{equation*}
and
\begin{equation*}
  \int_{\R^3} |H_8|\, |\delta w^{n+1}|\dd x  \leq (1+\| h^{n-1}\|_{L_\infty}) \|\nabla_\mh v^\mtd\|_{L_\infty} \|\delta w^{n+1}\|_{L_2}^2 ,
\end{equation*}
and
\begin{equation*}
\begin{split}
  \int_{\R^3} |H_9|\, |\delta w^{n+1}|\dd x  & \leq  \|\delta h^n\|_{L_2} \|v^\mtd_\mh\|_{L_\infty}\|\nabla_\mh v^\mtd\|_{L_\infty} \|\delta w^{n+1}\|_{L_2} \\
  & \leq \frac{1}{2}\|v^\mtd_\mh\|_{L_\infty}^2 \|\nabla_\mh v^\mtd\|_{L_\infty}^2 \|\delta h^n\|_{L_2}^2 +\frac{1}{2} \|\delta w^{n+1}\|_{L_2}^2  ,
\end{split}
\end{equation*}
and
\begin{equation*}
  \int_{\R^3} |H_{10}|\, |\delta w^{n+1}|\dd x  \leq \|\delta h^n\|_{L_2} \|(v^\mtd)_t\|_{L_\infty} \|\delta w^{n+1}\|_{L_2}\leq \frac{1}{2} \|(v^\mtd)_t\|_{L_\infty}^2 \|\delta h^n\|_{L_2}^2 +\frac{1}{2}\|\delta w^{n+1}\|_{L_2}^2.
\end{equation*}
Integrating on the time variable and using the following formula (concerning terms $H_3$ and $H_6$)
\begin{equation*}
\begin{split}
  & -\int_0^t\int_{\R^3}h^{n-1} (\delta w^{n+1})_\tau \, \cdot\delta w^{n+1}\dd x\mathrm{d}\tau -\int_0^t\int_{\R^3} h^{n-1} (v^{n-2}\cdot\nabla) \delta w^{n+1}\, \cdot \delta w^{n+1}\dd x\mathrm{d}\tau\\
  & =- \frac{1}{2}\int_{\R^3} h^{n-1}(t,x)\, |\delta w^{n+1}|^2(t,x)\dd x,
\end{split}
\end{equation*}
we gather the above estimates to obtain that
\begin{align}\label{dwnL2}
  & \|\delta w^{n+1}(t)\|_{L_2}^2 + \int_0^t \int_{\R^3} |\nabla \delta w^{n+1}(x,\tau)|^2\dd x\dd \tau \nonumber \\
  \leq & \|h_0\|_{L_\infty} \|\delta w^{n+1}(t)\|_{L_2}^2 +\left(7 + 2(1+ \|h_0\|_{L_\infty})\|\nabla_\mh v^\mtd\|_{L_\infty(\Omega_t)} \right)\, t\,\|\delta w^{n+1}\|_{L_\infty(0,t; L_2)}^2 \nonumber \\
  &+ t^{\frac{p-2}{p}}\|\nabla w^n\|_{L_p(0,t;L_\infty)}^2 \|\delta w^n\|_{L_\infty(0,t;L_2)}^2 + t^{\frac{p-2}{p}} \|h_0\|_{L_\infty}^2  \|\nabla w^{n+1}\|_{L_p(0,t;L_\infty)}^2 \|\delta w^{n-1}\|_{L_\infty(0,t;L_2)}^2 \\
  & + t^{\frac{p-2}{p}} \|v^{n-1}\|_{L_\infty(\Omega_t)}^2 \|\nabla w^{n+1}\|_{L_p(0,t;L_\infty)}^2 \|\delta h^n\|_{L_\infty(0,t;L_2)}^2
  +C t^{\frac{p-2}{p}}\|w^{n+1}_t\|_{L_p(\Omega_t)}^2 \|\delta h^n\|_{L_\infty(0,t;L_2)}^2 \nonumber\\
  & + t \left(\|w^{n+1}\|_{L_\infty(\Omega_t)}^2 \|\nabla_\mh v^\mtd\|_{L_\infty(\Omega_t)}^2 + \|v^\mtd_\mh\|_{L_\infty(\Omega_t)}^2 \|\nabla_\mh v^\mtd\|_{L_\infty(\Omega_t)}^2
  + \|(v^\mtd)_t\|_{L_\infty(\Omega_t)}^2\right)\|\delta h^n\|_{L_\infty(0,t;L_2)}^2, \nonumber
\end{align}
with $\Omega_t:= \R^3\times(0,t)$. Denoting by
\begin{equation}
  I_n(t):= \sup_{\tau\in[0,t]}\|\delta h^n(\tau)\|_{L_2}^2 + \sup_{\tau\in[0,t]}\|\delta w^n(\tau)\|_{L_2}^2
  + \int_0^t\|\nabla \delta w^n(\tau)\|_{L_2}^2\dd \tau,
\end{equation}
and relied on \eqref{dhnL2}, \eqref{dwnL2} and the smallness condition \eqref{wh-cd2}, Lemmas \ref{lem:2DNS-L2} - \ref{lem:2DNS-2}
and the uniform estimates \eqref{wk-ene}, \eqref{wk-ape}, \eqref{nhn-es},
we can let $t$ be small enough, i.e. $t\leq T_0$ with $T_0$ depending only on $p$, $\| w_0\|_{L_2\cap \dot B^{2-2/p}_{p,p}(\R^3)}$, $\|h_0\|_{L_2\cap L_\infty\cap \dot W^1_3(\R^3)}$ and $\|v^\mtd_0\|_{L_2\cap\dot B^{3-2/p}_{p,p}(\R^2)}$, such that
\begin{equation}
  I_{n+1}(T_0)\leq \frac{1}{4} I_n(T_0) +\frac{1}{8} I_{n-1}(T_0),
\end{equation}
which implies that
\begin{equation*}
  I_{n+1}(T_0)+\frac{1}{4}I_n(T_0)\leq \frac{1}{2}\Big(I_n(T_0) + \frac 14 I_{n-1}(T_0)\Big)\leq \cdots \leq \frac{1}{2^n}\Big(I_1(T_0) + \frac 14 I_0(T_0)\Big)\leq \frac{C_0}{2^n}.
\end{equation*}
Hence we conclude that $\{h^n,w^n\}_{n\in\N}$ is a Cauchy sequence in the space $L_\infty(0,T_0;L_2(\R^3))$.

\smallskip

{\bf Step 3: Strong Convergence.}

According to Step 2, there exists some functions $h\in L_\infty(0,T_0;L_2(\R^3))$ and $w\in L_\infty(0,T_0;L_2(\R^3))\cap L_2(0,T_0; \dot W^1_2(\R^3))$
such that $h^n\rightarrow h$ and $w^n\rightarrow w$ in $L_\infty(0,T_0;L_2(\R^3))$. By virtue of the uniform estimate \eqref{wk-ape}, \eqref{nhn-es} and the interpolation inequality,
we further get the strong convergence  $w^n\rightarrow w$ in $L_\infty(0,T_0; L_2\cap L_\infty\cap \dot B^{s_1}_{p,p}(\R^3))$ and $h^n\rightarrow h$ in $L_\infty(0,T_0; L_2\cap \dot W^{s_2}_3)$ with $s_1<2-2/p$ and $s_2<1$ arbitrary.
Therefore, we can pass the limit $n\rightarrow \infty$ in the system \eqref{appPertS} to deduce that $(h,w)$ is a distributional solution of the perturbed system \eqref{pertS1}-\eqref{pertS2}.
The limits of nonlinearities, since they are quasilinear, are well defined.
It is also guaranteed that $h\in L_\infty(0,T_0; \dot W^1_3(\R^3))$ and $w\in L_\infty(0,T_0;\dot B^{2-2/p}_{p,p}(\R^3))$, $w_t\in L_p(\Omega_{T_0})$, $\nabla^2 w\in L_p(\Omega_{T_0})$.

\smallskip

{\bf Step 4: Uniqueness.}

For any $T>0$, let $h_i\in L_\infty(0,T; L_2\cap L_\infty\cap \dot W^1_3)$, $w_i\in L_\infty(0,T; L_2\cap \dot B^{2-2/p}_{p,p})\cap \dot W^1_p(0,T; L_p)\cap L_p(0,T; \dot W^2_p)$,
$i=1,2$ be two solutions of the perturbed system \eqref{pertS1}-\eqref{pertS2} associated with the same initial data $(h_0,w_0)$,
which additionally satisfy that $\|h_i\|_{L_\infty(\Omega_T)}\leq \frac{1}{2}$. Denoting by
\begin{equation*}
  \delta h := h_1-h_2,\quad \delta w := w_1 -w_2, \quad \delta p:= p_1 - p_2,
\end{equation*}
and $v_i(t,x):= v^\mtd(t,x_\mh)+ w_i(t,x)$ $(i=1,2)$,
we obtain the equations of $(\delta h,\delta w)$ as follows
\begin{equation}\label{PertSuni}
 \begin{cases}
  (\delta h)_t+v_2 \cdot\nabla \delta h = -\delta w \cdot \nabla h_1, \\
  (1+h_2)(\delta w)_t + (1+h_2) (v_2 \cdot \nabla \delta w) -\Delta \delta w +\nabla \delta p = K, \\
  \divg \delta w=0, \\
  \delta h|_{t=0}=0,\quad \delta w|_{t=0}=0,
 \end{cases}
\end{equation}
where $K=\sum_{j=1}^7 K_j$ with
\begin{equation*}
\begin{split}
  & K_1:=-(1+h_2)\delta w \cdot\nabla w_1, \quad K_2:=-\delta h \,(w_1)_t,\quad K_3:= -\delta h (v_1\cdot\nabla w_1), \quad K_4:=-\delta h ( w_{1,\mh}\cdot\nabla_\mh v^\mtd), \\
  & K_5:=-(1+h_2)\delta w_\mh \cdot\nabla_\mh v^\mtd,\quad K_6:=-\delta h\, (v^\mtd)_t,\quad K_7:=-\delta h (v^\mtd_\mh \cdot\nabla_\mh v^\mtd).
\end{split}
\end{equation*}
In a similar way as the deduction in Step 2, we get that
\begin{equation*}
  \|\delta h(t)\|_{L_2}\leq \|\nabla h_1\|_{L_\infty(0,t;L_3)} t^{1/2}\left(\int_0^t\int_{\R^3} |\nabla \delta w|^2\dd x \dd \tau\right)^{1/2},
\end{equation*}
and
\begin{align*}
  & \|\delta w(t)\|_{L_2}^2 + \int_0^t \int_{\R^3} |\nabla \delta w(x,\tau)|^2\dd x\dd \tau \\
  \leq & \|h_2(t)\|_{L_\infty} \|\delta w(t)\|_{L_2}^2 +\left(7 + 2(1+ \|h_2\|_{L_\infty(\Omega_t)})\|\nabla_\mh v^\mtd\|_{L_\infty(\Omega_t)} \right)\, t\,\|\delta w\|_{L_\infty(0,t; L_2)}^2 \\
  &+ t^{\frac{p-2}{p}} \left(1+ \|h_2\|_{L_\infty(\Omega_t)}\right)\|\nabla w_1\|_{L_p(0,t;L_\infty)}^2 \|\delta w\|_{L_\infty(0,t;L_2)}^2  \\
  & + t^{\frac{p-2}{p}} \left(\|v_1\|_{L_\infty(\Omega_t)}^2 \|\nabla w_1\|_{L_p(0,t;L_\infty)}^2+\|(w_1)_t\|_{L_p(\Omega_t)}^2  \right) \|\delta h\|_{L_\infty(0,t;L_2)}^2 \\
  & + t \left(\|w_1\|_{L_\infty(\Omega_t)}^2 \|\nabla_\mh v^\mtd\|_{L_\infty(\Omega_t)}^2 + \|v^\mtd_\mh\|_{L_\infty(\Omega_t)}^2 \|\nabla_\mh v^\mtd\|_{L_\infty(\Omega_t)}^2 + \|(v^\mtd)_t\|_{L_\infty(\Omega_t)}^2\right)\|\delta h\|_{L_\infty(0,t;L_2)}^2.
\end{align*}
From the condition $\|h_2\|_{L_\infty(\Omega_T)}\leq \frac{1}{2}$, and by letting $t$ be small enough, i.e. for $t\leq T_1$ with $T_1$ depending only on $p$, $\| w_i\|_{L_\infty(0,T;L_2\cap \dot B^{2-2/p}_{p,p})}$, $\|( w_i)_t,\nabla^2 w_i\|_{L_p(\Omega_T)}$, $\|h_i\|_{L_\infty(0,T;L_2\cap L_\infty\cap \dot W^1_3)}$ and $\|v^\mtd_0\|_{L_2\cap\dot B^{3-2/p}_{p,p}(\R^2)}$, we find
\begin{equation}
  \|\delta h\|_{L_\infty(0,T_1; L_2(\R^3))}^2 + \|\delta w\|_{L_\infty(0,T_1; L_2(\R^3))}^2 + \int_0^{T_1}\int_{\R^3} |\nabla \delta w(t)|^2 \dd x\dd t \leq 0,
\end{equation}
which implies $(\delta h,\delta w)\equiv 0$ on $\R^3\times [0,T_1]$. By repeating the above process on the time interval $[T_1,2T_1]$, $[2T_1,3T_1]\,\cdots$, we conclude that
$(\delta h,\delta w)\equiv 0$ on all the spacetime domain $\R^3\times [0,T]$, which implies the uniqueness.

\smallskip

{\bf Step 5: the maximal time $T_*$ can equal $\infty$.}

Now we consider the time interval $[0,\infty)$ instead of $[0,T]$. We let $h\in L_\infty([0,T_*); L_2\cap\dot W^1_3(\R^3))$ and $w\in L_\infty([0,T_*); L_2\cap \dot B^{2-2/p}_{p,p})\cap \dot W^1_p([0,T_*); L_p)\cap L_p([0,T_*);\dot W^2_p)$
be the maximal lifespan solution of the perturbed system \eqref{pertS1}-\eqref{pertS2} constructed as above.
Suppose that $T_*<\infty$, we intend to show a contradiction.

Since $(h,w)$ is now regular enough on $(0,T_*)$ to satisfy the assumptions of Propositions
\ref{prop:eneR2}--\ref{prop:apR3}, we infer that
\begin{equation}
  \sup_{t\in[0,T_*)} \|h(t)\|_{L_2\cap L_\infty(\R^3)} + \sup_{t\in [0,T_*)} \|w(t)\|_{L_2\cap \dot B^{2-2/p}_{p,p}(\mathbb{R}^3)} \leq C \bar{c}_*,
\end{equation}
and
\begin{equation}
  \sup_{t\in [0,T_*)} \|\nabla h(t)\|_{L_3(\mathbb{R}^3)} \leq  \|\nabla h_0\|_{L_3(\R^3)} e^{(1+ T_*)C(h_0,w_0,v^\mtd_0)},
\end{equation}
where $\bar{c}_*$ is the absolute constant in \eqref{wh-cd} and $C>0$ is a generic constant depending only on $p$.
Hence if we assume
$$\bar{c}_*<\frac{c_*}{C}\exp\left\{-C' \left(\|v^\mtd_0\|_{L_2\cap \dot B^{2-2/p}_{p,p}(\R^2)}^{4p} +1\right)e^{C'(1+\|v^\mtd_0\|_{L_2}^4)}\right\},$$
with $c_*$ is the constant in \eqref{wh-cd2}, we have
\begin{equation}\label{wh-cd3}
  \left(\| w\|_{L_\infty([0,T_*);L_2\cap \dot B^{2-2/p}_{p,p})} + \|h\|_{L_\infty([0,T_*);L_2\cap L_\infty)} \right)
  \exp\left\{C' \left(\|v^\mtd_0\|_{L_2\cap \dot B^{2-2/p}_{p,p}}^{7p} + 1\right)e^{C'(1+\|v^\mtd_0\|_{L_2}^4)}\right\}\leq c_*,
\end{equation}
and we can repeat the above process in Step 1-3 from some time $t_0<T_*$ that can be arbitrarily close to $T_*$. Since the proceeding time $T_0$ is depending only on
$p$, $T_*$, $\| w\|_{L_\infty(0,T_*;L_2\cap \dot B^{2-2/p}_{p,p})} $,
$\|h\|_{L_\infty(0,T_*;L_2\cap L_\infty\cap \dot W^1_3)}$ and $\|v^\mtd_0\|_{L_2\cap\dot B^{3-2/p}_{p,p}(\R^2)}$, which in turn implies that $T_0$ is depending only on $p$, $T_*$ and the initial data $(h_0,w_0,v^\mtd_0)$,
we conclude that the maximal time $T_*$ can be proceeded beyond and this is a contradiction. Hence we obtain $T_*=\infty$. Then Theorem \ref{thm:regu} is proved.
\smallskip

\section{Proof of Theorem \ref{thm:rough}: global stability result with rough density}\label{sec:thm2}

The proof of Theorem \ref{thm:rough} is split into two parts: existence and uniqueness. We follow the following ideas.
Having already proved Theorem \ref{thm:regu}, we are able to find a class of suitable approximation solutions with regular density. Then since the a priori estimates (\ref{t2}) and (\ref{t3}) do not contain any dependence of $\epsilon$, we are able to find a suitable
sequence tending to a rough solution to the system (INS) with rough density. However this approach does not provide the uniqueness of solutions. To perform the uniqueness issue,
we apply the method from \cite{DanM12,DanM13} to use the approach via the Lagrangian coordinates setting, which is also the main part of the proof.

\subsection{Existence}

Take $\epsilon >0$ and define $\rho_0^\epsilon := \pi_\epsilon \ast \rho_0$, where $\pi_\epsilon$ stands for the standard mollifier. Then we are ensured that $\nabla \rho_0^\epsilon \in L_3(\R^3)$.
Next we consider the (INS) system (\ref{INS}) for $(\rho^\epsilon,v^\epsilon)$ with initial data $v^\epsilon|_{t=0}=v_0$ and $\rho^\epsilon|_{t=0}=\rho^\epsilon_0$.
Theorem \ref{thm:regu} implies that under the uniform-in-$\epsilon$ condition \eqref{t1},
the approximate system generates a unique regular solution $(\rho^\epsilon,v^\epsilon)$ satisfying estimates (\ref{t2}) and (\ref{t3}) with rhs's independent of $\epsilon$. Then we are allowed to subtract a subsequence $\epsilon_k \to 0$
(write $\epsilon \to 0$ for short) such that
\begin{equation}
  \rho^\epsilon -1 \rightharpoonup^* \rho -1  \mbox{ \ \ in \ \ } L_\infty(0,\infty;L_2\cap L_\infty(\R^3))
\end{equation}
and
\begin{equation}
  v^\epsilon - v^\mtd \rightharpoonup^* v-v^\mtd \mbox{ \ \ in  \ \ } L_\infty(0,\infty;L_2\cap \dot B^{2-2/p}_{p,p}(\R^3)) \cap \dot W^1_p(0,\infty;L_p(\R^3))\cap L_p(0,\infty;\dot W^2_p(\R^3)).
\end{equation}
In addition, by the argument of diagonal method together with Rellich type theorems applied for the compact (spacetime) subsets of $\R^3 \times \R_+$, we conclude
\begin{equation}
  v^\epsilon \to v\; \mbox{ a.e. pointwisely in } \R^3 \times (0,\infty).
\end{equation}
The above convergences are sufficient to pass to the limit in the equations (\ref{pertS1}) and prove that $(\rho,v)$ is indeed the solution to the (INS) system (\ref{INS}) (and also (\ref{pertS1}).
The only problematic term is $\rho^\epsilon \d_t v^\epsilon$, since here directly we have just weak convergence of single terms. However one can use the continuity equation of $\rho^\epsilon$ to represent as follows
\begin{equation}
  \rho^\epsilon \d_t v^\epsilon + \rho^\epsilon v^\epsilon \cdot \nabla v^\epsilon = \d_t(\rho^\epsilon v^\epsilon) + \divg (\rho^\epsilon v^\epsilon \otimes v^\epsilon).
\end{equation}
The rhs of the above relation allows to pass to the limit using its distributional form and strong convergence of the velocity. Hence the part concerning existence of Theorem \ref{thm:rough} is proved.

\subsection{Uniqueness}

We consider two solutions $ (h_1, w_1, q_1)$ and $(h_2, w_2, q_2)$ to the perturbed system \eqref{pertS1} starting from the same initial data $(h_0,v^\mtd_0,w_0)$ with
$h_0\in L_2\cap L_\infty(\R^3)$, $v^\mtd_0\in L_2\cap \dot B^{4-2/p}_{p,p}(\R^2)$
and $w_0\in L_2\cap \dot B^{2-2/p}_{p,p}(\R^3)$. According to \eqref{t3}, we have
\begin{equation}\label{w-apes}
  \sup_{t<\infty} \|w_i(t)\|_{L_2\cap \dot B^{2-2/p}_{p,p}(\mathbb{R}^3)} + \|(w_i)_t,\nabla^2 w_i,\nabla q_i\|_{L_p(\R^3\times(0,\infty))} \leq C \,c_0,
\end{equation}
which combined with \eqref{L2es2D} and \eqref{v2d-es1} leads to that for any $T>0$,
\begin{equation}\label{eq:vLinf}
\begin{split}
  \|v_i\|_{L_\infty(\R^3\times(0,T))} & \leq \|w_i\|_{L_\infty(0,T; L_\infty(\R^3))} + \|v^\mtd\|_{L_\infty(0,T; L_\infty(\R^2))}  \\
  & \leq \|w_i\|_{L_\infty(0,T; L_2\cap \dot B^{2-2/p}_{p,p}(\R^3))} + \|v^\mtd\|_{L_\infty(0,T; L_2\cap \dot B^{3-2/p}_{p,p}(\R^2))} \leq C,
\end{split}
\end{equation}
and (similar to \eqref{vL1Lip})
\begin{equation}\label{eq:Lip}
\begin{split}
  \int_0^T \|\nabla v_i(t)\|_{L_\infty(\R^3)}\dd t 
  & \leq C_0 T^{1-\frac{1}{p}} \|w_i\|_{L_p(0,T; L_2\cap \dot W^2_p(\R^3))} +  C_0 T \|v^\mtd\|_{L_\infty(0,T; L_2\cap \dot B^{3-2/p}_{p,p}(\R^2))} \\
  & \leq C (T^{1-\frac{1}{p}}+ T),
\end{split}
\end{equation}
with $C$ some constant depending on the initial data. Note that due to $v^\mtd_0\in L_2\cap \dot B^{4-2/p}_{p,p}(\R^2)$, we also have the estimate \eqref{v2d-es2}, which is more regular than \eqref{v2d-es1}.

By adopting the notations introduced in the subsection \ref{subsec:Lag}, the system of $(h_i,w_i)$ ($i=1,2$) in the Lagrangian coordinates is written as
\begin{equation}\label{LpertS3}
\begin{cases}
  \partial_t \bar{h}_i =0, \\
  \partial_t \bar{w}_i -\divg\left( A_{v_i} A_{v_i}^{\textrm{T}} \nabla_y \bar{w}_i\right) + A_{v_i}^{\textrm{T}}\nabla_y \overline{q}_i = \bar F_i, \\
  \divg_y\left( A_{v_i} \bar{w}_i \right) =0, \\
  \bar{h}_i|_{t=0}= h_0,\quad \bar{w}_i|_{t=0}= w_0,
\end{cases}
\end{equation}
with $\bar{F}_i(t,y) = F_i(t, X_{v_i}(t,y))$ and
\begin{equation}\label{Fi-exp}
  F_i:= -h (v^\mtd)_t - h \,\partial_t w_i -h(v_i\cdot\nabla w_i) -\rho( w_{i,\mh} \cdot\nabla_\mh v^\mtd) -h (v^\mtd_\mh\cdot\nabla_\mh v^\mtd).
\end{equation}
Note that from \eqref{bar-h}, the density in the framework of Lagrangian coordinates are the same, that is,
$$\bar{h}_1(t,y)=\bar{h}_2(t,y)\equiv h_0(y).$$
We see the difference equation of $\bar{w}_1-\bar{w}_2=: \delta \bar w$ reads as follows
\begin{equation}\label{LpertS-del}
\begin{cases}
  \delta \bar{w}_t -\divg\big( A_{v_1} A_{v_1}^{\textrm{T}} \nabla_y \delta \bar{w}\big) + A_{v_1}^{\textrm{T}}\nabla_y \delta \overline{q} =  \delta \bar{F} + \divg\left( (A_{v_1} A_{v_1}^{\textrm{T}} -
  A_{v_2}A_{v_2}^{\textrm{T}})\nabla_y  \bar{w}_2 \right)-(A_{v_1}^{\textrm{T}}-A_{v_2}^{\textrm{T}})\nabla_y  \overline{q}_2 , \\
  \divg_y\left( A_{v_1} \delta\bar{w} \right) =\divg\big((A_{v_1}-A_{v_2})\bar{w}_2\big) , \\
  \delta \bar{w}|_{t=0}=0,
\end{cases}
\end{equation}
where $\delta \bar q :=\bar q_1 - \bar q_2$, and $\delta\bar F:= \bar F_1 - \bar F_2$ is decomposed as $\sum_{i=1}^5 \delta\bar{F}^i$ with
\begin{equation}\label{delFi}
\begin{split}
  & \delta\bar{F}^1:=- h_0(y) \big(\partial_t v^\mtd(t, X_{v_1,\mh}) - \partial_t v^\mtd(t,X_{v_2,\mh})\big),\qquad \delta \bar{F}^2 := h_0(y) \delta \bar{w}_t, \\
  &\delta\bar{F}^3 := - h_0(y) \big((v^\mtd_\mh\cdot \nabla_\mh v^\mtd)(t,X_{v_1,\mh})-(v^\mtd_\mh\cdot \nabla_\mh v^\mtd)(t,X_{v_2,\mh}) \big), \\
  &\delta \bar{F}^4 := - \rho_0(y)\, \delta\overline{w_\mh}(t,y)\,\cdot \nabla_\mh v^\mtd\left( t,X_{v_1,\mh}\right), \\
  &\delta \bar{F}^5 :=- \rho_0(y)\, \overline{w_{2,\mh}}(t,y)\,\cdot \big(\nabla_\mh v^\mtd( t,X_{v_1,\mh})-\nabla_\mh v^\mtd( t,X_{v_2,\mh}) \big).
\end{split}
\end{equation}

We want to show that the norm $\|\delta \bar w(t)\|_{L_2}$ is zero, by the energy type estimates. The basic problem is related to the nonhomogeneous right-hand side of the second equation.
Because of it, we are not allowed to test the equation \eqref{LpertS-del} by $\delta \bar w$. We instead have to split it into two parts
\begin{equation}\label{z1z2}
  \delta \bar w = z^1 +z^2,
\end{equation}
where $z^1$ is given as a solution to the following divergence equation
\begin{equation}\label{z1def}
  \divg (A_{v_1} z^1)= \divg((A_{v_1}-A_{v_2})\bar{w}_2) = (A_{v_1}-A_{v_2}):\nabla_y \bar{w}_2.
\end{equation}
The construction of such a field $z^1$ can be done by the method from \cite{DanMdiv}, and one can see Lemma \ref{lem:divEq} for details.
We below need to verify the conditions \eqref{g-cd} and \eqref{A-cd} in Lemma \ref{lem:divEq}.
Note that from \eqref{DXvest} and \eqref{eq:Lip}, by letting $T$ be small enough we have
\begin{align}
  \int_0^T \|\nabla_y \bar{v}_1(t)\|_{L_\infty(\R^3)}\dd t & \leq \int_0^T \|\nabla_x v_1(t)\|_{L_\infty(\R^3)} \|\nabla_y X_{v_1}(t)\|_{L_\infty(\R^3)}\dd t \nonumber\\
  & \leq \int_0^T \|\nabla v_1\|_{L_\infty(\R^3)}\dd t \exp\left\{\int_0^T \|\nabla v_1\|_{L_\infty(\R^3)}\dd t\right\} \label{eq:fact} \\
  & \leq C T^{1-1/p} e^{CT^{1-1/p}} \leq \min\big\{\frac{c}{2},\frac{1}{4}\big\}, \nonumber
\end{align}
with $c>0$ the constant in \eqref{A-cd}, thus thanks to \eqref{Avbd1},
\begin{align}\label{eq:fact2}
  \|\mathrm{Id}-A_{v_1}\|_{L_\infty(0,T;L_\infty(\R^3))}\leq 2\int_0^T \|\nabla_y \bar{v}_1(t)\|_{L_\infty(\R^3)}\dd t \leq \min\big\{c,\frac{1}{2}\big\},
\end{align}
and
\begin{align}\label{eq:fact3}
  \|A_{v_1}\|_{L_\infty(0,T; L_\infty(\R^3))} \leq 2.
\end{align}
Taking advantage of \eqref{Avbd1} again, we see that
\begin{equation*}
  (A_{v_1})_t = \sum_{k=1}^\infty (-1)^k k \left( \int_0^t \nabla_y \bar v_1(\tau,y)\dd \tau\right)^{k-1} \big(\nabla_y \bar v_1(t,y)\big),
\end{equation*}
thus using \eqref{eq:fact}, \eqref{eq:Lip} and by letting $T>0$ be small enough we get
\begin{equation}\label{eq:fact4}
\begin{split}
  \|(A_{v_1})_t\|_{L_2(0,T; L_\infty(\R^3))} & \leq C \|\nabla_y \bar v_1\|_{L_2(0,T;L_\infty(\R^3))}
  \leq C T^{\frac{1}{2}-\frac{1}{p}} e^{C T^{1-1/p}} \leq c,
\end{split}
\end{equation}
which combined with \eqref{eq:fact2} ensures \eqref{A-cd}.
As for the condition \eqref{g-cd}, recalling $\mathcal{N}_p(T)$ is the function space defined in \eqref{Np}, we have to justify that
\begin{equation}\label{cdTarg}
  (A_{v_1}-A_{v_2})\bar{w}_2\in L_\infty(0,T; L_2),\;\; (A_{v_1}-A_{v_2}):\nabla_y \bar{w}_2 \in L_2(0,T; L_2),\;\;
  \big((A_{v_1}-A_{v_2})\bar{w}_2\big)_t \in \mathcal{N}_p(T),
\end{equation}
then we find
\begin{equation}\label{z1es1}
 \|(A_{v_1}-A_{v_2})\bar{w}_2\|_{L_\infty(0,T;L_2(\R^3))} \leq \|A_{v_1}-A_{v_2}\|_{L_\infty(0,T;L_2(\R^3))} \|\bar w_2\|_{L_\infty(0,T;L_\infty(\R^3))},
\end{equation}
and
\begin{equation}\label{z1es2}
  \|(A_{v_1}-A_{v_2}):\nabla_y \bar{w}_2\|_{L_2(0,T;L_2(\R^3))} \leq \|A_{v_1}-A_{v_2}\|_{L_\infty(0,T;L_2(\R^3))} \|\bar w_2\|_{L_2(0,T;\dot W^1_\infty(\R^3))},
\end{equation}
and
\begin{align}\label{z1es3}
  \|\big((A_{v_1}-A_{v_2})\bar{w}_2\big)_t\|_{\mathcal{N}(T)} \leq & \, \|(A_{v_1}-A_{v_2})\partial_t\bar w_2\|_{L_{\frac{2p}{2p-3}}(0,T;L_{\frac{2p}{p+2}})}
  + \|(A_{v_1}-A_{v_2})_t\bar w_2\|_{L_{\frac{2p}{2p-3}}(0,T;L_{\frac{2p}{p+2}})} \nonumber \\
  \leq & \, T^{\frac{p-3}{2p}}\|A_{v_1}-A_{v_2}\|_{L_\infty(0,T;L_2(\R^3))} \|\partial_t\bar w_2\|_{L_2(0,T;L_{p}(\R^3))} \nonumber \\
  & + T^{\frac{p-3}{2p}} \|(A_{v_1}-A_{v_2})_t\|_{L_2(0,T;L_2(\R^3))} \|\bar w_2\|_{L_\infty(0,T;L_{p}(\R^3))}.
\end{align}
Thanks to \eqref{DXvest}, \eqref{w-apes}-\eqref{eq:Lip}, we have
\begin{equation*}
  \|\bar w_2\|_{L_\infty(0,T;L_\infty\cap L_p(\R^3))}\leq \|w_2\|_{L_\infty(0,T; L_\infty\cap L_p(\R^3))} \leq C \|w_2\|_{L_\infty(0,T; L_2\cap \dot B^{2-2/p}_{p,p})} \leq C,
\end{equation*}
and
\begin{equation*}
\begin{split}
  \|\nabla_y\bar w_2\|_{L_2(0,T;L_\infty(\R^3))} & \leq \|\nabla_y X_{v_2}\|_{L_\infty(\R^3\times(0,T))} \|\nabla_x w_2\|_{L_2(0,T;L_\infty(\R^3))} \\
  & \leq C T^{\frac{p-2}{2p}} e^{\|\nabla v_2\|_{L_1(0,T; L_\infty)}} \|w_2\|_{L_p(0,T; L_2\cap \dot W^2_p(\R^3))} \leq C T^{\frac{p-2}{2p}},
\end{split}
\end{equation*}
and
\begin{equation*}
\begin{split}
  & \|\partial_t\bar w_2\|_{L_2(0,T;L_p(\R^3))} \leq \| \partial_t w_2\|_{L_2(0,T;L_p(\R^3))} + \|\nabla_x w_2\|_{L_\infty(0,T; L_p(\R^3))} \|\partial_t X_{v_2}\|_{L_2(0,T; L_\infty(\R^3))} \\
  & \leq C T^{\frac{p-2}{2p}} \| \partial_t w_2\|_{L_p(\R^3\times(0,T))} + T^{\frac{1}{2}} \|w_2\|_{L_\infty(0,T; L_2\cap \dot B^{2-2/p}_{p,p}(\R^3))} \|v_2\|_{L_\infty(\R^3\times(0,T))}
  \leq C T^{\frac{p-2}{2p}},
\end{split}
\end{equation*}
thus in order to obtain \eqref{cdTarg}, it suffices to control $\|A_{v_1}-A_{v_2}\|_{L_\infty(0,T;L_2(\R^3))}$ and $\|(A_{v_1}-A_{v_2})_t\|_{L_2(0,T;L_2(\R^3))}$.
Next from \eqref{DxY} observe that
\begin{equation*}
\begin{split}
  A_{v_1}-A_{v_2} =\,&\sum_{k=1}^\infty \sum_{j=0}^{k-1} (-1)^k (C_{v_1}(t,y))^j (C_{v_2}(t,y))^{k-1-j} \int_0^t \nabla_y \delta\bar v(\tau,y)\dd y, \\
  (A_{v_1}- A_{v_2})_t=\,& \sum_{k=1}^\infty (-1)^k k \left((C_{v_1}(t,y))^{k-1} \nabla \bar{v}_1(t,y) - (C_{v_2}(t,y))^{k-1} \nabla \bar{v}_2(t,y) \right) \\
  = \,& \nabla \delta \bar v(t,y) + \nabla\delta \bar v(t,y) \sum_{k=2}^\infty (-1)^k k (C_{v_2}(t,y))^{k-1} \\
  & + \nabla \bar{v}_1(t,y)\sum_{k=2}^\infty\sum_{j=0}^{k-2} (-1)^k k (C_{v_1}(t,y))^j (C_{v_2}(t,y))^{k-2-j} \int_0^t \nabla_y \delta\bar v(\tau,y)\dd y ,
\end{split}
\end{equation*}
with $C_{v_i}(t,y)= \int_0^t \nabla \bar{v}_i(\tau,y)\dd y$, $i=1,2$, by letting $T>0$ small enough so that \eqref{eq:fact} holds, we get
\begin{equation}\label{Av-es1}
  \|A_{v_1}-A_{v_2}\|_{L_\infty(0,T;L_2(\R^3))} \leq CT^{1/2}\big(\|\nabla\big((\overline{v^\mtd})_1 -(\overline{v^\mtd})_2\big)\|_{L_2(0,T;L_2(\R^3))} + \|\nabla \delta \bar w\|_{L_2(0,T;L_2(\R^3))}\big),
\end{equation}
and
\begin{equation}\label{Av-es2}
  \|(A_{v_1}-A_{v_2})_t\|_{L_2(0,T;L_2(\R^3))} \leq C\big(\|\nabla\big((\overline{v^\mtd})_1 -(\overline{v^\mtd})_2\big)\|_{L_2(0,T;L_2(\R^3))} + \|\nabla \delta \bar w\|_{L_2(0,T;L_2(\R^3))}\big),
\end{equation}
where $\overline{v^\mtd}^i= v^\mtd (t,X_{v_i,\mh}(t,y))$, $i=1,2$.
Noting that
\begin{equation}\label{bar-v2d}
\begin{split}
  (\overline{v^\mtd})_1(t,y)- (\overline{v^\mtd})_2(t,y) = (X_{v_1,\mh}(t,y)-X_{v_2,\mh}(t,y))\cdot\int_0^1\nabla_{x_\mh} v^\mtd\big(t,s X_{v_1,\mh}+(1-s)X_{v_2,\mh}\big)\dd s& \\
  = \int_0^t \big((\overline{v^\mtd_\mh})_1(\tau,y)- (\overline{v^\mtd_\mh})_2(\tau,y) + \delta\overline{w_\mh}(\tau,y)\big)\dd \tau\cdot\int_0^1\nabla_{x_\mh} v^\mtd\big(t,s X_{v_1,\mh}+(1-s)X_{v_2,\mh}\big)\dd s& ,
\end{split}
\end{equation}
and
\begin{equation}
\begin{split}
  & \nabla_y\big((\overline{v^\mtd})_1(t,y)- (\overline{v^\mtd})_2(t,y)\big) \\
  = &\nabla_y X_{v_1,\mh}(t,y)\cdot\nabla_{x_\mh} v^\mtd\big(t,X_{v_1,\mh}(t,y)\big) -\nabla_y X_{v_2,\mh}(t,y)\cdot \nabla_{x_\mh} v^\mtd\big(t,X_{v_2,\mh}(t,y)\big) \\
  = & \nabla_y X_{v_1,\mh}(t,y) \cdot \big(\nabla_{x_\mh} v^\mtd(t,X_{v_1,\mh}(t,y))-\nabla_{x_\mh} v^\mtd(t,X_{v_2,\mh}(t,y))\big)  \\
  & + \left(\nabla_y X_{v_1,\mh}(t,y)-\nabla_y X_{v_2,\mh}(t,y)\right) \cdot \nabla_{x_\mh} v^\mtd(t,X_{v_2,\mh}(t,y))\\
  = &\nabla_y X_{v_1,\mh}(t,y)\cdot\int_0^1\nabla^2_{x_\mh} v^\mtd\big(t,s X_{v_1,\mh}+(1-s)X_{v_2,\mh}\big)\dd s\, \cdot(X_{v_1,\mh}(t,y)-X_{v_2,\mh}(t,y)) \\
  & +\int_0^t\Big(\nabla_y (\overline{v^\mtd_\mh})_1(\tau,y)- \nabla_y (\overline{v^\mtd_\mh})_2(\tau,y) + \nabla\delta\overline{w_\mh}(\tau,y)\Big)\dd \tau\cdot\nabla_{x_\mh} v^\mtd(t,X_{v_2,\mh}(t,y)) ,
\end{split}
\end{equation}
and due to that $\nabla_x v^\mtd$ is Lipschitzian and bounded (see \eqref{v2d-es2} below), so by letting time $T$ small enough we have
\begin{equation}\label{b-v2des}
  \|(\overline{v^\mtd})_1 - (\overline{v^\mtd})_2\|_{L_\infty(0,T;L_2(\R^3))}   \leq CT \|\delta \bar w\|_{L_\infty(0,T;L_2(\R^3))},
\end{equation}
and
\begin{equation}\label{b-v2des2}
   \|\nabla\big((\overline{v^\mtd})_1 - (\overline{v^\mtd})_2\big)\|_{L_2(0,T;L_2(\R^3))}
  \leq C T^{1/2}\big(\|\delta \bar w\|_{L_\infty(0,T;L_2(\R^3))} + \|\nabla \delta \bar w\|_{L_2(0,T;L_2(\R^3))}\big).
\end{equation}
Inserting \eqref{b-v2des}-\eqref{b-v2des2} into \eqref{Av-es1}-\eqref{Av-es2} leads to that for sufficiently small $T$,
\begin{equation}\label{Av-es3}
  \|A_{v_1}-A_{v_2}\|_{L_\infty(0,T;L_2(\R^3))} \leq CT^{1/2}(\|\delta \bar w\|_{L_\infty(0,T;L_2(\R^3))} + \|\nabla \delta \bar w\|_{L_2(0,T;L_2(\R^3))}),
\end{equation}
and
\begin{equation}\label{Av-es4}
  \|(A_{v_1}-A_{v_2})_t\|_{L_2(0,T;L_2(\R^3))} \leq C(\|\delta \bar w\|_{L_\infty(0,T;L_2(\R^3))} + \|\nabla \delta \bar w\|_{L_2(0,T;L_2(\R^3))}).
\end{equation}
By collecting the above estimates, we thus verify the condition \eqref{cdTarg}. Hence, Lemma \ref{lem:divEq} and the above estimates ensure that
\begin{equation}\label{z1es-key}
  \|z^1\|_{L_\infty(0,T;L_2(\R^3))} + \|\nabla z^1\|_{L_2(\R^3\times(0,T))} \leq C T^{1/2}\big(\|\delta \bar w\|_{L_\infty(0,T;L_2(\R^3))} + \|\nabla \delta \bar w\|_{L_2(0,T;L_2(\R^3))}\big),
\end{equation}
and
\begin{equation}\label{z1es5}
  \|\partial_t z^1\|_{\mathcal{N}(T)} \leq C T^{\frac{p-3}{2p}}\big(\|\delta \bar w\|_{L_\infty(0,T;L_2(\R^3))} + \|\nabla \delta \bar w\|_{L_2(0,T;L_2(\R^3))}\big).
\end{equation}
\vskip0.2cm

Now we look at the equation on $z^2$ which satisfies that
\begin{equation}\label{LpertS-z2}
\begin{cases}
  \partial_t z^2 -\divg\left( A_{v_1} A_{v_1}^{\textrm{T}} \nabla_y z^2\right) + A_{v_1}^{\textrm{T}}\nabla_y \delta \overline{q} =  \delta\bar F + \sum_{i=1}^4 \bar{L}^i \\
  \divg_y\left( A_{v_1} z^2\right) =0, \\
  z^2|_{t=0}=0,
\end{cases}
\end{equation}
with $\delta \bar{F}:= \sum_{i=1}^5 \delta \bar{F}^i$ given by \eqref{delFi}, and
\begin{equation}
\begin{split}
  & \bar{L}^1 = \divg\left( (A_{v_1} A_{v_1}^{\textrm{T}} -  A_{v_2}A_{v_2}^{\textrm{T}})\nabla_y  \bar{w}_2 \right),\quad \bar{L}^2=-(A_{v_1}^{\textrm{T}}-A_{v_2}^{\textrm{T}})\nabla_y  \overline{q}_2,\\
  & \bar{L}^3 =-\partial_t z^1,\qquad\qquad \bar{L}^4=\divg\left( A_{v_1} A_{v_1}^{\textrm{T}} \nabla_y z^1\right).
\end{split}
\end{equation}
We test the equation \eqref{LpertS-z2} by $z^2$, and noticing
\begin{equation}
  \int_{\R^3} A_{v_1}^{\textrm{T}}\nabla_y \delta \overline{q}(t,y)\, z^2(t,y) \dd y = -\int_{\R^3} \delta \overline{q}(t,y)\,\divg(A_{v_1} z^2)(t,y) \dd y =0,
\end{equation}
we derive that
\begin{equation}\label{z2-es1}
  \frac{1}{2} \frac{\dd}{\dd t} \|z^2(t)\|_{L_2}^2 + \|A_{v_1}^{\textrm{T}} \nabla_y z^2(t)\|_{L_2(\R^3)}^2 \leq \sum_{i=1}^5 \int_{\R^3} \delta \bar{F}^i(t,y) z^2(t,y)\dd y
  + \sum_{i=1}^4 \int_{\R^3} \bar{L}^i(t,y) \,z^2(t,y)\dd y.
\end{equation}
By virtue of the estimate $\|\mathrm{Id}-A_{v_1}\|_{L_\infty(0,T;L_\infty(\R^3))}\leq \frac{1}{2}$ (from \eqref{eq:fact2}), it induces that
\begin{equation}
\begin{split}
  \|A_{v_1}^{\textrm{T}} \nabla_y z^2(t)\|_{L_2(\R^3)}^2 & \geq \Big(\|\nabla_y z^2(t)\|_{L_2(\R^3)} - \|\mathrm{Id}-A_{v_1}\|_{L_\infty(\R^3\times(0,T))} \|\nabla_y z^2(t)\|_{L_2(\R^3)} \Big)^2 \\
  & \geq \frac{1}{2}\|\nabla_y z^2(t)\|_{L_2(\R^3)}^2 -  \|\mathrm{Id}-A_{v_1}\|_{L_\infty(\R^3\times(0,T))}^2 \|\nabla_y z^2(t)\|_{L_2(\R^3)}^2 \\
  & \geq \frac{1}{4}\|\nabla_y z^2(t)\|_{L_2(\R^3)}^2.
\end{split}
\end{equation}
From the integration by parts, H\"older's inequality and \eqref{eq:fact3}, we see that
\begin{equation}
\begin{split}
  \left|\int_{\R^3} \bar{L}^1(t,y) \,z^2(t,y)\dd y\right| & = \left|\int_{\R^3}  (A_{v_1} A_{v_1}^{\textrm{T}} -  A_{v_2}A_{v_2}^{\textrm{T}})\nabla_y  \bar{w}_2
  \cdot\nabla z^2(t,y) \dd y\right| \\
  & \leq \|A_{v_1} A_{v_1}^{\textrm{T}} -A_{v_2} A_{v_2}^{\textrm{T}}\|_{L_\infty(0,T; L_2(\R^3))} \|\nabla_y \bar{w}_2(t) \|_{L_\infty(\R^3)} \|\nabla z^2(t)\|_{L_2(\R^3)} \\
  & \leq \frac{1}{32} \|\nabla z^2(t)\|_{L_2(\R^3)}^2 + C \|A_{v_1}-A_{v_2}\|_{L_\infty(0,T; L_2)}^2 \|\nabla_y \bar{w}_2(t) \|_{L_\infty(\R^3)}^2.
\end{split}
\end{equation}
Using H\"older's inequality and the interpolation inequality, it follows
\begin{equation}
\begin{split}
  \left|\int_{\R^3} \bar{L}^2(t,y) \,z^2(t,y) \dd y\right| & \leq  \|A_{v_1}-A_{v_2}\|_{L_\infty(0,T:L_2(\R^3))} \|\nabla_y\bar{q}_2(t)\|_{L_p(\R^3)} \|z^2(t)\|_{L_{\frac{2p}{p-2}}(\R^3)}  \\
  & \leq  \|A_{v_1}-A_{v_2}\|_{L_\infty(0,T; L_2(\R^3))}\|\nabla_y\bar{q}_2(t)\|_{L_p}  \|z^2(t)\|_{L_2(\R^3)}^{\frac{p-3}{p}} \| \nabla z^2(t)\|_{L_2(\R^3)}^{\frac{3}{p}}\\
  & \leq \frac{1}{32} \|\nabla z^2(t)\|_{L_2}^2 + C \|A_{v_1}-A_{v_2}\|_{L_\infty(0,T; L_2)}^{\frac{2p}{2p-3}} \|\nabla\bar{q}_2(t)\|_{L_p}^{\frac{2p}{2p-3}} \|z^2(t)\|_{L_2}^{\frac{2p-6}{2p-3}} .
\end{split}
\end{equation}
For the right-hand side of \eqref{z2-es1} containing $\bar{L}^3= -\partial_t z^1$, since $\partial_t z^1\in \mathcal{N}_p(T)$ with $\mathcal{N}_p(T)$ defined by \eqref{Np}, for any $\epsilon>0$ we take
$$
\partial_tz^1= a_\epsilon + b_\epsilon, \mbox{ \ \ with \ \ } a_\epsilon \in L_{\frac{2p}{2p-3}}(0,T;L_{\frac{2p}{p+2}}(\R^3)), \; b_\epsilon \in L_2(0,T;L_2(\R^3)),
$$
so that
\begin{equation}\label{ab-es}
  \|a_\epsilon\|_{L_{\frac{2p}{2p-3}}(0,T;L_{\frac{2p}{p+2}}(\R^3))} + \|b_\epsilon\|_{L_2(0,T;L_2(\R^3))} \leq \|\partial_t z^1\|_{\mathcal{N}_p(T)} + \epsilon,
\end{equation}
thus by virtue of H\"older's inequality and the interpolation inequality, we infer that for $p\in (3,\infty)$,
\begin{equation}\label{barL3es}
\begin{split}
  \left|\int_{\R^3} \bar{L}^3(t,y) \,z^2(t,y) \dd y\right| & \leq \|a_\epsilon(t)\|_{L_{\frac{2p}{p+2}}(\R^3)}  \|z^2(t)\|_{L_{\frac{2p}{p-2}}(\R^3)} +\|b_\epsilon(t)\|_{L_2(\R^3)}\|z^2(t)\|_{L_2(\R^3)} \\
  & \leq \|a_\epsilon(t)\|_{L_{\frac{2p}{p+2}}} \|z^2(t)\|_{L_2(\R^3)}^{\frac{p-3}{p}} \|\nabla z^2(t)\|_{L_2(\R^3)}^{\frac{3}{p}} +\|b_\epsilon(t)\|_{L_2}\|z^2(t)\|_{L_2(\R^3)}  \\
  & \leq \frac{1}{32} \|\nabla z^2(t)\|_{L_2}^2 + C \|a_\epsilon\|_{L_{\frac{2p}{p+2}}}^{\frac{2p}{2p-3}} \|z^2(t)\|_{L_2(\R^3)}^{\frac{2p-6}{2p-3}} +\|b_\epsilon(t)\|_{L_2}\|z^2(t)\|_{L_2(\R^3)} .
\end{split}
\end{equation}
For the right-hand side of \eqref{z2-es1} containing $\bar{L}^4$, we integrate by parts and use \eqref{eq:fact3} to show that
\begin{equation}
\begin{split}
  \left|\int_{\R^3} \bar{L}^4(t,y) \,z^2(t,y) \dd y\right| & = \left| \int_{\R^3} A_{v_1} A_{v_1}^{\mathrm{T}} \nabla z^1(t,y)\cdot\nabla z^2(t,y)\dd y\right| \\
  & \leq \|A_{v_1}\|_{L_\infty(\R^3\times (0,T))} \|\nabla z^1(t)\|_{L_2(\R^3)} \|\nabla z^2(t)\|_{L_2(\R^3)}  \\
  & \leq \frac{1}{32} \|\nabla z^2(t)\|_{L_2(\R^3)}^2 + C \|\nabla z^1(t)\|_{L_2(\R^3)}^2 .
\end{split}
\end{equation}
Next we consider the right-hand-side terms of \eqref{z2-es1} containing $\delta\bar F^i$ $(1\leq i\leq5)$ which are given by \eqref{delFi}. By H\"older's inequality and
\begin{equation*}
  \partial_t v^\mtd(t,X_{v_1,\mh})-\partial_t v^\mtd(t,X_{v_2,\mh})= \int_0^1 (X_{v_1,\mh}-X_{v_2,\mh})\cdot \nabla_{x_\mh}\partial_t v^\mtd\big(t,sX_{v_1,\mh}+(1-s)X_{v_2,\mh}\big)\dd s,
\end{equation*}
we get
\begin{equation}
\begin{split}
  \left|\int_{\R^3} \delta \bar F^1(t,y) z^2(t,y) \dd y\right| & \leq \|h_0\|_{L_\infty} \|\partial_t v^\mtd (t,X_{v_1,\mh})-\partial_t v^\mtd (t,X_{v_2,\mh})\|_{L_2(\R^3)} \|z^2(t)\|_{L_2(\R^3)} \\
  & \leq \|h_0\|_{L_\infty} \|\partial_t\nabla_\mh v^\mtd (t)\|_{L_\infty(\R^2)} \|X_{v_1}-X_{v_2}\|_{L_\infty(0,T;L_2(\R^3))} \|z^2(t)\|_{L_2(\R^3)}.
\end{split}
\end{equation}
It follows from \eqref{Xv} that
\begin{equation*}
  X_{v_1}(t,y)-X_{v_2}(t,y) = \int_0^t \left((\overline{v^\mtd})_1(\tau,y)-(\overline{v^\mtd})_2(\tau,y) + \delta \bar w(\tau,y)\right)\dd \tau ,
\end{equation*}
and by \eqref{b-v2des},
\begin{equation*}
  \|X_{v_1}-X_{v_2}\|_{L_\infty(0,T;L_2(\R^3))} \leq C T \|\delta \bar w\|_{L_\infty(0,T; L_2(\R^3))} ,
\end{equation*}
thus we obtain
\begin{equation}
  \left|\int_{\R^3} \delta \bar F^1(t,y) z^2(t,y) \dd y\right|\leq C T\|h_0\|_{L_\infty} \|\partial_t\nabla_\mh v^\mtd (t)\|_{L_\infty(\R^2)} \|\delta \bar w\|_{L_\infty(0,T;L_2(\R^3))} \|z^2(t)\|_{L_2(\R^3)}.
\end{equation}
For the right-hand-side terms of \eqref{z2-es1} containing $\delta\bar F^2$, in a similar way as the deduction in \eqref{barL3es}, we have
\begin{equation}
\begin{split}
  & \int_{\R^3} \delta \bar F^2(t,y) z^2(t,y) \dd y = \frac{1}{2}\frac{\dd}{\dd t} \int_{\R^3} h_0(y) |z^2(t,y)|^2\dd y + \int_{\R^3} h_0(y) \partial_t z^1(t,y) z^2(t,y) \dd y \\
  \leq & \frac{1}{2}\frac{\dd}{\dd t} \int_{\R^3} h_0(y) |z^2(t,y)|^2\dd y + \|h_0\|_{L_\infty} \|a_\epsilon(t)\|_{L_{\frac{2p}{p+2}}}  \|z^2(t)\|_{L_{\frac{2p}{p-2}}}
  + \|h_0\|_{L_\infty} \|b_\epsilon(t)\|_{L_2} \|z^2(t)\|_{L_2}  \\
  \leq & \frac{1}{2}\frac{\dd}{\dd t} \int_{\R^3} h_0(y) |z^2(t,y)|^2\dd y + \frac{1}{32} \|\nabla z^2(t)\|_{L_2(\R^3)}^2
  +  C \|h_0\|_{L_\infty(\R^3)}^{\frac{2p}{2p-3}} \|a_\epsilon(t)\|_{L_{\frac{2p}{p+2}}(\R^3)}^{\frac{2p}{2p-3}} \|z^2(t)\|_{L_2(\R^3)}^{\frac{2p-6}{2p-3}} \\
  & + \|h_0\|_{L_\infty(\R^3)} \|b_\epsilon(t)\|_{L_2(\R^3)} \|z^2(t)\|_{L_2(\R^3)} .
\end{split}
\end{equation}
By arguing as above, we find that
\begin{equation}
\begin{split}
  \bigg|\int_{\R^3} \delta \bar F^3\, z^2(t,y) & \dd y \bigg|\leq \|h_0\|_{L_\infty} \|v^\mtd_\mh\cdot\nabla_\mh v^\mtd (t, X_{v_1,\mh})-v^\mtd_\mh\cdot\nabla_\mh v^\mtd(t,X_{v_2,\mh})\|_{L_2(\R^3)}
  \|z^2(t)\|_{L_2} \\
  & \leq \|h_0\|_{L_\infty} \|\nabla_\mh(v^\mtd_\mh\cdot\nabla_\mh v^\mtd)(t)\|_{L_\infty(\R^2)} \|X_{v_1}- X_{v_2}\|_{L_\infty(0,T;L_2(\R^3))} \|z^2(t)\|_{L_2} \\
  & \leq C T \|h_0\|_{L_\infty} \|\nabla_\mh(v^\mtd_\mh\cdot\nabla_\mh v^\mtd)(t)\|_{L_\infty(\R^2)} \|\delta \bar w\|_{L_\infty(0,T;L_2(\R^3))} \|z^2(t)\|_{L_2(\R^3)},
\end{split}
\end{equation}
and
\begin{equation}
\begin{split}
  \left|\int_{\R^3} \delta \bar F^4(t,y) z^2(t,y) \dd y\right| 
  \leq \|\rho_0\|_{L_\infty(\R^3)}\|\nabla_\mh v^\mtd(t)\|_{L_\infty(\R^2)} \|\delta \bar w(t)\|_{L_2(\R^3)} \|z^2(t)\|_{L_2(\R^3)},
\end{split}
\end{equation}
and
\begin{equation}
\begin{split}
  \left| \int_{\R^3} \delta \bar F^5(t,y) z^2(t,y) \dd y\right| & \leq \|\rho_0\|_{L_\infty} \|\bar w_2(t)\|_{L_\infty} \|\nabla_\mh v^\mtd(t,X_{v_1,\mh})-\nabla_\mh v^\mtd(t,X_{v_2,\mh})\|_{L_2}
  \|z^2(t)\|_{L_2}\\
  & \leq C T \|\rho_0\|_{L_\infty}  \|w_2\|_{L_\infty(\R^3\times (0,T))} \|\nabla^2_\mh v^\mtd(t)\|_{L_\infty} \|\delta \bar w\|_{L_\infty(0,T;L_2)}\|z^2(t)\|_{L_2}.
\end{split}
\end{equation}
Gathering \eqref{z2-es1} and the above estimates yields
\begin{equation*}
\begin{split}
  & \frac{\dd }{\dd t} \|z^2(t)\|_{L_2}^2 + \frac{3}{16} \|\nabla z^2(t)\|_{L_2}^2 - \frac{\dd }{\dd t}\int_{\R^3} h_0(y) |z^2(t,y)|^2\dd y \\
  \leq\, & C \|A_{v_1}-A_{v_2}\|_{L_\infty(0,T; L_2)}^2 \|\nabla_y \bar{w}_2(t) \|_{L_\infty}^2
  + C \|A_{v_1}-A_{v_2}\|_{L_\infty(0,T; L_2)}^{\frac{2p}{2p-3}} \|\nabla\bar{q}_2(t)\|_{L_p}^{\frac{2p}{2p-3}} \|z^2(t)\|_{L_2}^{\frac{2p-6}{2p-3}} \\
  & + C (1+\|h_0\|_{L_\infty}^{\frac{2p}{2p-3}})\|a_\epsilon(t)\|_{L_{\frac{2p}{p+2}}}^{\frac{2p}{2p-3}} \|z^2(t)\|_{L_2}^{\frac{2p-6}{2p-3}}
  +C(1+ \|h_0\|_{L_\infty}) \|b_\epsilon(t)\|_{L_2}\|z^2(t)\|_{L_2} + C\|\nabla z^1(t)\|_{L_2}^2 \\
  & + C T\|h_0\|_{L_\infty} \left(\|\partial_t\nabla_\mh v^\mtd (t)\|_{L_\infty(\R^2)}+ \|\nabla_\mh(v^\mtd_\mh\cdot\nabla_\mh v^\mtd)(t)\|_{L_\infty(\R^2)}\right) \|\delta \bar w\|_{L_\infty(0,T;L_2)} \|z^2(t)\|_{L_2} \\
  & + C \|\rho_0\|_{L_\infty} \left(\|\nabla_\mh v^\mtd(t)\|_{L_\infty(\R^2)} + T\| w_2\|_{L_\infty(\R^3\times (0,T))} \|\nabla^2_\mh v^\mtd(t)\|_{L_\infty}\right) \|\delta \bar w\|_{L_\infty(0,T;L_2)}\|z^2(t)\|_{L_2}.
\end{split}
\end{equation*}
Noting that according to \eqref{DXvest}, \eqref{w-apes}, \eqref{eq:Lip} and \eqref{v2d-es2},
\begin{equation*}
\begin{split}
  \|\nabla_y \bar{w}_2\|_{L_2(0,T; L_\infty(\R^3))} & \leq \|\nabla_y X_{v_2}\|_{L_\infty(\R^3\times (0,T))} \|\nabla_x w_2\|_{L_2(0,T; L_\infty(\R^3))} \\
  & \leq C e^{\|\nabla v_2\|_{L_1(0,T; L_\infty)}} T^{\frac{p-2}{2p}}\| w_2\|_{L_p(0,T; L_2\cap \dot W^2_p(\R^3))} \leq C T^{\frac{p-2}{2p}},
\end{split}
\end{equation*}
and
\begin{equation*}
\begin{split}
  \|\nabla_y \bar{q}_2\|_{L_{\frac{2p}{2p-3}}(0,T; L_p(\R^3))} & \leq \|\nabla_y X_{v_2}\|_{L_\infty(\R^3\times (0,T))} \|\nabla_x q_2\|_{L_{\frac{2p}{2p-3}}(0,T; L_p(\R^3))} \\
  & \leq C T^{\frac{2p-5}{2p}} e^{\|\nabla v_2\|_{L_1(0,T; L_\infty)}} \|\nabla q_2\|_{L_p(\R^3\times (0,T))} \leq C T^{\frac{2p-5}{2p}},
\end{split}
\end{equation*}
and
\begin{equation*}
\begin{split}
  \|\partial_t \nabla_\mh v^\mtd\|_{L_1(0,T; L_\infty(\R^2))} & \leq C T^{\frac{p-1}{p}}\|\partial_t v^\mtd\|_{L_p(0,T; W^2_p(\R^2))} \leq C T^{\frac{p-1}{p}}, \\
  \|\nabla_\mh(v^\mtd_\mh\cdot\nabla_\mh v^\mtd)\|_{L_1(0,T; L_\infty(\R^2))} & \leq \|\nabla_\mh v^\mtd\|_{L_1(0,T;L_\infty)}^2 + \|v^\mtd_\mh\|_{L_\infty(\R^2\times (0,T))} \|\nabla^2_\mh v^\mtd\|_{L_1(0,T; L_\infty)} \\
  & \leq C T^2 \|v^\mtd\|_{L_\infty(0,T; L_2\cap \dot B^{3-2/p}_{p,p})}^2 + C T \|v^\mtd\|_{L_\infty(0,T;L_2\cap \dot B^{4-2/p}_{p,p})}^2 \leq C T,
\end{split}
\end{equation*}
we integrate on the time variable and set $\|h_0\|_{L_\infty(\R^3)}$ small enough to deduce that
\begin{equation}\label{z2es-key}
\begin{split}
  & \|z^2\|_{L_\infty(0,T; L_2(\R^3))}^2 + \|\nabla z^2\|_{L_2(\R^3\times(0,T))}^2 \\
  \leq \, & C T^{\frac{p-2}{p}}\|A_{v_1}-A_{v_2}\|_{L_\infty(0,T; L_2)}^2 +
  C T^{\frac{2p-5}{2p-3}}\|A_{v_1}-A_{v_2}\|_{L_\infty(0,T; L_2)}^{\frac{2p}{2p-3}}  \|z^2\|_{L_\infty(0,T;L_2)}^{\frac{2p-6}{2p-3}} \\
  & + C \|a_\epsilon \|_{L_{\frac{2p}{2p-3}}(0,T;L_{\frac{2p}{p+2}}(\R^3))}^{\frac{2p}{2p-3}} \|z^2\|_{L_\infty(0,T; L_2(\R^3))}^{\frac{2p-6}{2p-3}}+ CT^{1/2}\|b_\epsilon\|_{L_2(0,T;L_2(\R^3))}\|z^2\|_{L_\infty(0,T;L_2(\R^3))} \\
  & + C \|\nabla z^1\|_{L_2(\R^3\times(0,T))}^2
  + C T \|\delta \bar w\|_{L_\infty(0,T;L_2(\R^3))} \|z^2\|_{L_\infty(0,T;L_2(\R^3))} .
\end{split}
\end{equation}
Recalling \eqref{Av-es3}, \eqref{z1es5} and \eqref{ab-es}, we combine \eqref{z2es-key} with \eqref{z1es-key}-\eqref{z1es5} to get
\begin{equation}\label{z2es-key2}
\begin{split}
  &  \|z^1\|_{L_\infty(0,T;L_2(\R^3))}^2 + \|\nabla z^1\|_{L_2(\R^3\times(0,T))}^2 + \|z^2\|_{L_\infty(0,T; L_2(\R^3))}^2 + \|\nabla z^2\|_{L_2(\R^3\times(0,T))}^2 \\
  \leq \, & C T^{\frac{p-3}{2p-3}} \big(\|\delta \bar w\|_{L_\infty(0,T;L_2(\R^3))}^2 + \|\nabla \delta \bar w\|_{L_2(\R^3\times(0,T))}^2 + \|z^2\|_{L_\infty(0,T;L_2(\R^3))}^2 + \epsilon^2 \big).
\end{split}
\end{equation}
By passing $\epsilon$ to 0 and letting $T>0$ be small enough, we conclude that $z^1=z^2\equiv 0$ and $\delta \bar w =z^1+ z^2\equiv 0$ on $\R^3\times [0,T]$.
In light of \eqref{b-v2des} and \eqref{Xv}, we also get $(\overline{v^\mtd})_1\equiv (\overline{v^\mtd})_2$ and $X_{v_1}(t,y)\equiv X_{v_2}(t,y)$ on $\R^3\times [0,T]$.
Hence, by \eqref{eq:fact}, coming back to the Eulerian coordinates, we infer that $(h_1,w_1)\equiv (h_2, w_2)$ on $\R^3\times [0,T]$,
which corresponds to the uniqueness on the small interval $[0,T]$.

Now suppose that the solution to \eqref{pertS1} is unique on the time interval $[0,T']$, with $T'>0$ a fixed time.
Let $(h_i,w_i,q_i)$ ($i=1,2$) be two solutions to \eqref{pertS1} starting from the same initial data $(h_0, v^\mtd_0, w_0)$, and from the assumption: $(h_1,w_1,q_1)\equiv (h_2,w_2,q_2)$ on $[0,T']$.
We next introduce another Lagrangian coordinate $\widetilde{X}_{v_i}(t,y)$ $(i=1,2)$ defined by
\begin{equation}
  \frac{\dd \widetilde{X}_{v_i}(t,y)}{\dd t} = v_i(t,\widetilde{X}_{v_i}(t,y)),\quad \widetilde{X}_{v_i}(T',y)=y,
\end{equation}
which corresponds to
\begin{equation}
  \widetilde{X}_{v_i}(t,y) = y + \int_{T'}^t v_i(\tau, \widetilde{X}_{v_i}(\tau,y)) \dd \tau.
\end{equation}
In terms of this Lagrangian coordinate, and denoting by
\begin{equation*}
\begin{split}
  & \tilde{h}_i(t,y):= h_i(t, \widetilde{X}_{v_i}(t,y)),\quad \tilde{w}_i(t,y):= w_i(t, \widetilde{X}_{v_i}(t,y)),\quad \\
  & \tilde{q}_i(t,y):= q_i(t,\widetilde{X}_{v_i}(t,y)),\quad \tilde F_i(t,y):= F_i(t, \widetilde{X}_{v_i}(t,y)),
\end{split}
\end{equation*}
the perturbed system \eqref{pertS1} corresponding to $(h_i,w_i)$ ($i=1,2$) can be written as
\begin{equation}\label{LpertS2}
\begin{cases}
  \partial_t\tilde{h}_i =0, \\
  \partial_t\tilde{w}_i -\divg\big( \tilde{A}_{v_i} \tilde{A}_{v_i}^{\textrm{T}} \nabla_y \tilde{w}_i\big) + \tilde{A}_{v_i}^{\textrm{T}}\nabla_y \tilde{q}_i = \tilde F_i, \\
  \divg_y\big( \tilde{A}_{v_i} \tilde{w}_i \big) =0, \\
  \tilde{h}|_{t=T'}= h(T',y),\quad \tilde{w}_i|_{t=T'}= w(T',y),
\end{cases}
\end{equation}
where $h(T',y)=h_1(T',y)=h_2(T',y)$, $w(T',y)=w_1(T',y)=w_2(T',y)$ (from the uniqueness assumption) and
\begin{equation}\label{Av2}
  \tilde{A}_{v_i}(t,y):=(\nabla_y \widetilde{X}_{v_i}(t,y))^{-1}.
\end{equation}
We also set
\begin{equation*}
  (\widetilde{v^\mtd})_i(t,y):= v^\mtd(t,\widetilde{X}_{v_i,\mh}(t,y)), \quad\textrm{and}\quad \tilde v_i(t,y):= v_i(t,\widetilde{X}_{v_i}(t,y)),
\end{equation*}
with $\widetilde{X}_{v_i,\mh}(t,y)=\big( \widetilde{X}_{v_i,1}(t,y), \widetilde{X}_{v_i,2}(t,y)\big)$, then
\begin{equation}\label{tilde-v}
  \tilde v_i(t,y)= (\widetilde{v^\mtd})_i(t,y) + \tilde{w}_i(t,y).
\end{equation}
From the first equation of \eqref{LpertS2}, we see that
\begin{equation}
  \tilde{h}_i(t,y)\equiv h_i(T',y)= h(T',y),\quad \forall t\in [T',T'+T],
\end{equation}
and similarly as \eqref{barFi}, we have
\begin{equation*}\label{tildeFi}
\begin{split}
  \tilde F_i(t,y)
  = & -h(T',y)\, (v^\mtd)_t\left( t,\widetilde{X}_{v_i,\mh}\right) - h(T',y)\, \partial_t\tilde{w}_i\left( t,y\right) \\
  & -h(T',y) \,(v^\mtd_\mh\cdot \nabla_\mh v^\mtd)(t,\widetilde{X}_{v_i,\mh})- \rho(T',y)\, \tilde{w}_{i,\mh}(t,y)\,\cdot (\nabla_\mh v^\mtd)\left( t,\widetilde{X}_{v_i,\mh}\right).
\end{split}
\end{equation*}
The difference equation of $\delta w:=\tilde{w}^1-\tilde{w}^2$ reads as follows
\begin{equation}\label{LpertS-del2}
\begin{cases}
  \delta \tilde{w}_t -\divg\big( \tilde{A}_{v_1} \tilde{A}_{v_1}^{\textrm{T}} \nabla \delta \tilde{w}\big) + \tilde{A}_{v_1}^{\textrm{T}}\nabla \delta \tilde{q} =  \delta \tilde{F} + \divg\Big( (\tilde{A}_{v_1} \tilde{A}_{v_1}^{\textrm{T}} -
  \tilde{A}_{v_2}\tilde{A}_{v_2}^{\textrm{T}})\nabla  \tilde{w}^2 \Big)-(\tilde{A}_{v_1}^{\textrm{T}}-\tilde{A}_{v_2}^{\textrm{T}})\nabla  \tilde{q}_2 , \\
  \divg_y\left( \tilde{A}_{v_1} \delta\tilde{w} \right) =\divg\big((\tilde{A}_{v_1}-\tilde{A}_{v_2})\tilde{w}^2\big) , \\
  \delta \tilde{w}|_{t=T'}=0,
\end{cases}
\end{equation}
where $\delta \tilde q := \tilde q_1 - \tilde q_2$, and $\delta\tilde F:= \tilde F_1 - \tilde F_2$ is decomposed as $\sum_{i=1}^5 \delta\tilde{F}^i$ with
\begin{equation*}
\begin{split}
  & \delta\tilde{F}^1=- h(T',y) \big(\partial_t v^\mtd(t, \widetilde{X}_{v_1,\mh}) - \partial_t v^\mtd(t,\widetilde{X}_{v_2,\mh})\big),\qquad \delta \tilde{F}^2 = h(T',y) \delta \tilde{w}_t, \\
  &\delta\tilde{F}^3= - h(T',y) \big((v^\mtd_\mh\cdot \nabla_\mh v^\mtd)(t,\widetilde{X}_{v_1,\mh})-(v^\mtd_\mh\cdot \nabla_\mh v^\mtd)(t,\widetilde{X}_{v_2,\mh}) \big), \\
  &\delta \tilde{F}^4 = - \rho(T',y)\, \delta\widetilde{w_\mh}(t,y)\,\cdot \nabla_\mh v^\mtd\left( t,\widetilde{X}_{v_1,\mh}\right), \\
  &\delta \tilde{F}^5 =- \rho(T',y)\, \tilde{w}_{2,\mh}(t,y)\,\cdot \big(\nabla_\mh v^\mtd( t,\widetilde{X}_{v_1,\mh})-\nabla_\mh v^\mtd( t,\widetilde{X}_{v_2,\mh}) \big).
\end{split}
\end{equation*}
We split $\delta\tilde{w}$ into two parts
\begin{equation}\label{z1z2-2}
  \delta \tilde w = \tilde z^1 + \tilde z^2,
\end{equation}
where $\tilde z^1$ is given as the solution to the following equation (from Lemma \ref{lem:divEq})
\begin{equation}\label{z1def2}
  \divg (\tilde{A}_{v_1} \tilde{z}^1)= \divg((\tilde{A}_{v_1}-\tilde{A}_{v_2})\tilde{w}_2) = (\tilde{A}_{v_1}-\tilde{A}_{v_2}):\nabla_y \tilde{w}_2. 
\end{equation}
In view of \eqref{eq:Lip} and \eqref{eq:fact}-\eqref{eq:fact4}, we have that for $T>0$ small enough,
\begin{equation}\label{Xv-Linf}
  \int_{T'}^{T'+T} \|\nabla_y \tilde{v}_i(\tau)\|_{L_\infty(\R^3)}\dd \tau
  \leq \int_{T'}^{T'+T} \|\nabla_y \widetilde{X}_{v_i}(\tau)\|_{L_\infty}\|\nabla v_i(\tau)\|_{L_\infty(\R^3)} \dd \tau\leq \min\big\{\frac{c}{2},\frac{1}{4} \big\},
\end{equation}
and
\begin{equation*}
  \|\mathrm{Id}- \tilde{A}_{v_i}(t,y)\|_{L_\infty(T',T'+T;L_\infty(\R^3))} \leq 2 \int_{T'}^{T'+T} \|\nabla\tilde{v}_i(\tau)\|_{L_\infty(\R^3)} \dd \tau \leq \min\big\{c, \frac{1}{2}\big\},
\end{equation*}
and
\begin{equation*}
  \|(\tilde{A}_{v_i})_t\|_{L_2(T',T'+T;L_\infty(\R^3))}\leq C \|\nabla_y \tilde v_i\|_{L_2(T',T'+T;L_\infty(\R^3))} \leq c.
\end{equation*}
By arguing as the above deduction on the small interval $[0,T]$, and from $\|h(T')\|_{L_\infty}\leq \|h_0\|_{L_\infty} \ll1$, we find that
\begin{equation*}
\begin{split}
  &  \|\tilde{z}^1\|_{L_\infty(T',T'+T;L_2(\R^3))}^2 + \|\nabla \tilde{z}^1\|_{L_2(\R^3\times(T',T'+T))}^2
  + \|\tilde{z}^2\|_{L_\infty(T',T'+T; L_2(\R^3))}^2 + \|\nabla \tilde{z}^2\|_{L_2(\R^3\times(T',T'+T))}^2 \\
  & \leq \, C T^{\frac{p-3}{2p-3}} \big(\|\delta \tilde w\|_{L_\infty(T',T'+T;L_2(\R^3))}^2 + \|\nabla \delta \tilde w\|_{L_2(\R^3\times(T',T'+T))}^2
  + \|\tilde{z}^2\|_{L_\infty(T',T'+T;L_2(\R^3))}^2 \big).
\end{split}
\end{equation*}

Therefore, for any large number $T_*>T'>0$, there exists a sufficiently small constant $T>0$ depending only on the initial data and $T_*$
so that $\tilde{z}^1=\tilde{z}^2\equiv 0$ and $\delta\tilde{w}\equiv 0$ on $\R^3\times [T',T'+T]$, and moreover $\widetilde{X}_{v_1}(t,y)\equiv \widetilde{X}_{v_2}(t,y)$,
which combined with \eqref{Xv-Linf} implies $(h_1,w_1)\equiv (h_2,w_2)$ on $\R^3\times [T',T'+T]$.
By standard connectivity, we get $\delta w\equiv 0$ on $\R^3\times [0,T_*]$ and from the arbitraries of $T_*$, we conclude the uniqueness on the whole $\R^3\times [0,\infty)$.

\section{Appendix}

First, we give the proof of the energy type estimates \eqref{L2es2d-2}-\eqref{L2es2d-4} appearing in Lemma \ref{lem:2DNS-L2}.
\begin{proof}[Proof of \eqref{L2es2d-2}-\eqref{L2es2d-4} in Lemma \ref{lem:2DNS-L2}]
We take the inner product of the 2D (HNS) system \eqref{HNS} with the vector $\partial_t v^\mtd$,
and by the divergence-free property of $v^\mtd_\mh$ and the integration by parts, we get
\begin{equation*}
  \|\partial_t v^\mtd\|_{L_2}^2 + \frac{1}{2}\frac{\dd }{\dd t}\|\nabla_\mh v^\mtd\|_{L_2}^2 =
  \left|\int_{\R^2}  (v^\mtd_\mh \cdot\nabla_\mh v^\mtd) \cdot \partial_t v^\mtd \dd x \right| .
\end{equation*}
Multiplying both sides of the above equality with $t$ and integrating on the time interval $[0,t]$, we use H\"older's inequality to find
\begin{equation*}
  \frac{1}{2}t \|\nabla_\mh v^\mtd\|_{L^2}^2 - \frac{1}{2} \int_0^t \|\nabla_\mh v^\mtd\|_{L_2}^2\dd \tau + \int_0^t \tau \|\partial_\tau v^\mtd\|_{L_2}^2\dd \tau
  \leq \int_0^t \tau \|\nabla_\mh v^\mtd_\mh\|_{L_4} \|v^\mtd\|_{L_4} \|\partial_\tau v^\mtd\|_{L_2} \dd \tau.
\end{equation*}
Applying the interpolation inequality \eqref{eq:IntP}
and Young's inequality, it follows that
\begin{equation*}
\begin{split}
  \frac{1}{2}t \|\nabla_\mh v^\mtd\|_{L^2}^2 + \int_0^t \tau \|\partial_\tau v^\mtd\|_{L_2}^2\dd \tau
  \leq \frac{1}{2}\|v^\mtd_0\|_{L_2}^2 + C \int_0^t \tau\|v^\mtd\|_{L_2}^{1/2} \|\nabla_\mh v^\mtd\|_{L_2} \|\nabla_\mh^2 v^\mtd\|_{L_2}^{1/2} \|\partial_\tau v^\mtd\|_{L_2} \dd \tau & \\
  \leq \frac{1}{2}\|v^\mtd_0\|_{L_2}^2 + C\|v^\mtd_0\|_{L_2} \int_0^t \tau \|\nabla_\mh v^\mtd\|_{L_2}^2 \|\nabla_\mh^2 v^\mtd\|_{L_2} \dd \tau + \frac{1}{2}\int_0^t\tau \|\partial_\tau v^\mtd\|_{L_2}^2 \dd \tau ,&
\end{split}
\end{equation*}
and thus
\begin{equation*}
\begin{split}
  t \|\nabla_\mh v^\mtd\|_{L^2}^2 + \int_0^t \tau \|\partial_\tau v^\mtd\|_{L_2}^2\dd \tau
  \leq \|v^\mtd_0\|_{L_2}^2 + C\|v^\mtd_0\|_{L_2} \int_0^t \tau \|\nabla_\mh v^\mtd\|_{L_2}^2 \|\nabla_\mh^2 v^\mtd\|_{L_2} \dd \tau.
\end{split}
\end{equation*}
We write the (HNS) system \eqref{HNS} as
\begin{equation}\label{eq:HNS-3}
  \Delta_\mh v^\mtd + \nabla p^\mtd = -\partial_t v^\mtd - v^\mtd_\mh\cdot\nabla_\mh v^\mtd,
\end{equation}
and from the classical property of the Stokes system, we infer that
\begin{equation*}
\begin{split}
  \|\nabla_\mh^2 v^\mtd(t)\|_{L_2} + \|\nabla_\mh p^\mtd(t)\|_{L_2} & \leq C \|\partial_t v^\mtd(t)\|_{L_2} + C \|v^\mtd_\mh \cdot\nabla_\mh v^\mtd(t)\|_{L_2} \\
  & \leq C \|\partial_t v^\mtd(t)\|_{L_2} + C \|v^\mtd(t)\|_{L_2}^{1/2} \|\nabla_\mh v^\mtd(t)\|_{L_2} \|\nabla^2_\mh v^\mtd(t)\|_{L_2}^{1/2} \\
  & \leq C \|\partial_t v^\mtd(t)\|_{L_2} + C \|v^\mtd_0\|_{L_2} \|\nabla_\mh v^\mtd\|_{L_2}^2 + \frac{1}{2} \|\nabla_\mh^2 v^\mtd(t)\|_{L_2} ,
\end{split}
\end{equation*}
which implies that
\begin{equation}\label{eq:StokesEs}
    \|\nabla_\mh^2 v^\mtd(t)\|_{L_2} + \|\nabla_\mh p^\mtd(t)\|_{L_2} \leq C \|\partial_t v^\mtd(t)\|_{L_2} + C \|v^\mtd_0\|_{L_2} \|\nabla_\mh v^\mtd\|_{L_2}^2 .
\end{equation}
Thus we obtain
\begin{align*}
  &  t \|\nabla_\mh v^\mtd(t)\|_{L^2}^2 + \int_0^t \tau \|\partial_\tau v^\mtd(\tau)\|_{L_2}^2\dd \tau \\
  \leq &
  \|v^\mtd_0\|_{L_2}^2 + C \|v^\mtd_0\|_{L_2} \int_0^t \big(\tau \|\nabla_\mh v^\mtd(\tau)\|_{L_2}^2\big) \|\partial_\tau v^\mtd\|_{L_2} \dd \tau
  + C \|v^\mtd_0\|_{L_2}^2 \int_0^t \tau \|\nabla_\mh v^\mtd(\tau)\|_{L_2}^4 \dd \tau \\
  \leq & \|v^\mtd_0\|_{L_2}^2 + C \|v^\mtd_0\|_{L_2}^2 \int_0^t \big(\tau \|\nabla_\mh v^\mtd(\tau)\|_{L_2}^2\big) \|\nabla_\mh v^\mtd(\tau)\|_{L_2}^2 \dd \tau
  + \frac{1}{2}\int_0^t \tau \|\partial_\tau v^\mtd(\tau)\|_{L_2}^2\dd \tau.
\end{align*}
Gr\"onwall's inequality and \eqref{L2es2D} lead to that
\begin{equation}\label{L2es2d-2-1}
  t \|\nabla_\mh v^\mtd(t)\|_{L_2(\R^2)}^2 + \int_0^t \tau  \|\partial_\tau v^\mtd\|_{L_2(\R^2)}^2 \dd \tau
  \leq  C \|v^\mtd_0\|_{L_2(\R^2)}^2 e^{C \|v^\mtd_0\|_{L_2(\R^2)}^4}.
\end{equation}
Together with \eqref{eq:StokesEs}, we deduce
\begin{align}\label{L2es2d-2-2}
  \int_0^t \tau \|\nabla_\mh^2 v^\mtd(\tau)\|_{L_2}^2 \dd \tau & \leq C \int_0^t \tau \|\partial_\tau v^\mtd\|_{L_2}^2 \dd \tau
  + C \|v^\mtd_0\|_{L_2}^2 \sup_{\tau\in[0,t]}\Big(\tau \|\nabla_\mh v^\mtd(\tau)\|_{L_2}^2 \Big)\int_0^t \|\nabla_\mh v^\mtd(\tau)\|_{L_2}^2 \dd \tau \nonumber \\
  & \leq C ( \|v^\mtd_0\|_{L_2}^2 + \|v^\mtd_0\|_{L_2}^6) e^{C\|v^\mtd_0\|_{L_2}^4}  \leq C \|v^\mtd_0\|_{L_2(\R^2)}^2 e^{C\|v^\mtd_0\|_{L_2(\R^2)}^4} .
\end{align}
Combining \eqref{L2es2d-2-2} with \eqref{L2es2d-2-1} yields the desired estimate \eqref{L2es2d-2}.
\vskip0.1cm

Now we turn to \eqref{L2es2d-3}. Observing that
\begin{equation}\label{eq:HNS-2}
  \partial_{tt} v^\mtd +  v^\mtd_\mh\cdot\nabla_\mh (\partial_t v^\mtd) -\Delta_\mh (\partial_t v^\mtd) + \nabla(\partial_t p^\mtd)=- \partial_t v^\mtd_\mh \cdot\nabla_\mh v^\mtd£¬
\end{equation}
and taking the inner product with $\partial_t v^\mtd$, we have
\begin{align*}
  \frac{1}{2}\frac{\dd}{\dd t} \|\partial_t v^\mtd\|_{L_2}^2 + \|\nabla_\mh \partial_t v^\mtd\|_{L_2}^2 = \left| \int_{\R^2} (\partial_t v^\mtd_\mh \cdot \nabla_\mh v^\mtd) \cdot \partial_t v^\mtd \dd x \right|.
\end{align*}
By multiplying both sides of the above equation with $t^2$ and integrating on the time variable, we get
\begin{align*}
  \frac{1}{2} t^2 \|\partial_t v^\mtd\|_{L_2}^2 & + \int_0^t \tau^2 \|\nabla_\mh \partial_\tau v^\mtd\|_{L_2}^2 \dd \tau
  \leq \int_0^t \tau \|\partial_\tau v^\mtd\|_{L_2}^2 \dd \tau
  + \int_0^t \tau^2 \|\partial_\tau v^\mtd\|_{L_4}^2 \|\nabla_\mh v^\mtd\|_{L_2} \dd \tau \\
  & \leq C\|v^\mtd_0\|_{L_2}^2 e^{C\|v^\mtd_0\|_{L_2}^4} + \int_0^t \tau^2 \|\partial_\tau v^\mtd\|_{L_2} \|\nabla_\mh \partial_\tau v^\mtd\|_{L_2} \|\nabla_\mh v^\mtd\|_{L_2} \dd \tau \\
  & \leq C\|v^\mtd_0\|_{L_2}^2 e^{C\|v^\mtd_0\|_{L_2}^4} + \int_0^t \tau^2 \|\partial_\tau v^\mtd\|_{L_2}^2 \|\nabla_\mh v^\mtd\|_{L_2}^2 \dd \tau
  + \frac{1}{2} \int_0^t \tau^2\|\nabla_\mh \partial_\tau v^\mtd\|_{L_2}^2 \dd\tau.
\end{align*}
Gr\"onwall's inequality guarantees that
\begin{equation}\label{L2es2d-3-2}
   t^2 \|\partial_t v^\mtd\|_{L_2}^2  + \int_0^t \tau^2 \|\nabla_\mh \partial_\tau v^\mtd\|_{L_2}^2 \dd \tau
   \leq C\|v^\mtd_0\|_{L_2}^2 e^{C\|v^\mtd_0\|_{L_2}^4+ C\|v_0\|_{L_2}^2 } \leq C\|v^\mtd_0\|_{L_2}^2 e^{C(1+\|v^\mtd_0\|_{L_2}^4)}.
\end{equation}
By virtue of \eqref{eq:StokesEs} and \eqref{L2es2D}, we see that
\begin{equation}
\begin{split}
  t \|\nabla_\mh^2 v^\mtd(t)\|_{L_2}   \leq C \big(t \|\partial_t v^\mtd(t)\|_{L_2}\big) + C \|v^\mtd_0\|_{L_2} \big(t\|\nabla_\mh v^\mtd\|_{L_2}^2\big)
  \leq C  \|v^\mtd_0\|_{L_2}^2 e^{C(1+\|v^\mtd_0\|_{L_2}^4)},
\end{split}
\end{equation}
which combined with \eqref{L2es2d-3-2} yields \eqref{L2es2d-3}, as desired.
\vskip0.1cm

Next we treat \eqref{L2es2d-4}. Differentiating the equation \eqref{eq:HNS-2} leads to
\begin{equation}\label{eq:HNS-4}
\begin{split}
  \partial_{tt} \nabla_\mh v^\mtd +  v^\mtd_\mh\cdot\nabla_\mh (\nabla_\mh\partial_t v^\mtd) & -\Delta_\mh (\nabla_\mh\partial_t v^\mtd) + \nabla(\nabla_\mh\partial_t p^\mtd)\\
  & = - \nabla_\mh v^\mtd_\mh \cdot\nabla_\mh \partial_t v^\mtd - \nabla_\mh \partial_t v^\mtd_\mh \cdot\nabla_\mh v^\mtd -\partial_t v^\mtd_\mh \cdot \nabla_\mh^2 v^\mtd.
\end{split}
\end{equation}
By taking the inner product of this equation with $\nabla_\mh\partial_t v^\mtd$, and using H\"older's inequality and the integration by parts, we have
\begin{align}\label{eq:est1}
  \frac{1}{2}\frac{\dd}{\dd t} \|\nabla_\mh \partial_t v^\mtd\|_{L_2}^2 + \|\nabla_\mh^2 \partial_t v^\mtd\|_{L_2}^2
  \leq 2  \|\nabla_\mh \partial_t v^\mtd \|_{L_4}^2 \|\nabla_\mh v^\mtd \|_{L_2} + \|\partial_t v^\mtd\|_{L_4} \|\nabla_\mh v^\mtd\|_{L_4} \|\nabla_\mh^2 \partial_t v^\mtd\|_{L_2}.
\end{align}
Multiplying both sides with $t^3$ and integrating on the time variable, we find
\begin{align*}
  & \frac{1}{2} t^3 \|\nabla_\mh\partial_t v^\mtd\|_{L_2}^2 + \int_0^t \tau^3 \|\nabla_\mh^2 \partial_\tau v^\mtd\|_{L_2}^2 \dd \tau \\
  \leq & \frac{3}{2}\int_0^t \tau^2 \|\nabla_\mh\partial_\tau v^\mtd\|_{L_2}^2 \dd \tau
  +2 \int_0^t \tau^3 \|\nabla_\mh\partial_\tau v^\mtd\|_{L_4}^2 \|\nabla_\mh v^\mtd\|_{L_2} \dd \tau
  + \int_0^t \tau^3  \|\partial_\tau v^\mtd\|_{L_4} \|\nabla_\mh v^\mtd\|_{L_4} \|\nabla_\mh^2 \partial_\tau v^\mtd\|_{L_2} \dd \tau  \\
  \leq & C\|v^\mtd_0\|_{L_2}^2 e^{C(1+\|v^\mtd_0\|_{L_2}^4)}
  + C \int_0^t \tau^3 \|\nabla_\mh\partial_\tau v^\mtd\|_{L_2} \|\nabla_\mh^2 \partial_\tau v^\mtd\|_{L_2} \|\nabla_\mh v^\mtd\|_{L_2} \dd \tau  \\
  & + C\int_0^t \tau^3  \|\partial_\tau v^\mtd\|_{L_2}^{\frac{1}{2}} \|\nabla_\mh \partial_\tau v^\mtd\|_{L_2}^{\frac{1}{2}} \|\nabla_\mh v^\mtd\|_{L_2}^{\frac{1}{2}}
  \|\nabla_\mh^2 v^\mtd\|_{L_2}^{\frac{1}{2}} \|\nabla_\mh^2 \partial_\tau v^\mtd\|_{L_2} \dd \tau \\
  \leq & C\|v^\mtd_0\|_{L_2}^2 e^{C(1+\|v^\mtd_0\|_{L_2}^4)}
  + \frac{1}{2}\int_0^t \tau^3 \|\nabla_\mh^2 \partial_\tau v^\mtd\|_{L_2}^2 \dd \tau
  + C \int_0^t \tau^3 \|\nabla_\mh\partial_\tau v^\mtd\|_{L_2}^2\|\nabla_\mh v^\mtd\|_{L_2}^2 \dd \tau  \\
  & + C \int_0^t \tau^3  \|\partial_\tau v^\mtd\|_{L_2} \|\nabla_\mh \partial_\tau v^\mtd\|_{L_2} \|\nabla_\mh v^\mtd\|_{L_2}
  \|\nabla_\mh^2 v^\mtd\|_{L_2} \dd \tau .
\end{align*}
It is clear to see that
\begin{align*}
  & \int_0^t \tau^3  \|\partial_\tau v^\mtd\|_{L_2} \|\nabla_\mh \partial_\tau v^\mtd\|_{L_2} \|\nabla_\mh v^\mtd\|_{L_2}
  \|\nabla_\mh^2 v^\mtd\|_{L_2} \dd \tau \\
  \leq & \int_0^t \tau^4  \|\partial_\tau v^\mtd\|_{L_2}^2 \|\nabla_\mh v^\mtd\|_{L_2}^2 \|\nabla_\mh^2 v^\mtd\|_{L_2}^2 \dd \tau
  + \int_0^t \tau^2\|\nabla_\mh \partial_\tau v^\mtd\|_{L_2}^2 \dd \tau \\
  \leq & \sup_{\tau \in[0,t]} \big( \tau^2\|\partial_\tau v^\mtd\|_{L_2}^2\big) \sup_{\tau \in[0,t]} \big( \tau^2\|\nabla_\mh^2 v^\mtd\|_{L_2}^2\big)
  \int_0^t \|\nabla_\mh v^\mtd\|_{L_2}^2\dd\tau + C  \|v^\mtd_0\|_{L_2}^2 e^{C(1+\|v^\mtd_0\|_{L_2}^4)} \\
  \leq & C \|v^\mtd_0\|_{L_2}^2 \big(\|v^\mtd_0\|_{L_2}^4 + 1\big) e^{C(1+\|v^\mtd_0\|_{L_2}^4)}
  \leq C \|v^\mtd_0\|_{L_2}^2 e^{C(1+\|v^\mtd_0\|_{L_2}^4)} .
\end{align*}
We use Gr\"onwall's inequality to conclude that
\begin{equation}\label{L2es2d-4-1}
  t^3 \|\nabla_\mh\partial_t v^\mtd\|_{L_2}^2 + \int_0^t \tau^3 \|\nabla_\mh^2 \partial_\tau v^\mtd\|_{L_2}^2 \dd \tau
  \leq C \|v^\mtd_0\|_{L_2}^2 e^{C(1+\|v^\mtd_0\|_{L_2}^4)}. 
\end{equation}
From \eqref{eq:HNS-3} we get
\begin{equation*}
  \Delta_\mh \nabla_\mh v^\mtd - \nabla (\nabla_\mh p^\mtd) = \nabla_\mh \partial_t v^\mtd + (\nabla_\mh v^\mtd_\mh) \cdot \nabla_\mh v^\mtd
  + v^\mtd_\mh \cdot \nabla_\mh (\nabla_\mh v^\mtd),
\end{equation*}
and the classical estimate of Stokes system ensures that
\begin{align*}
  \|\nabla_\mh^3 v^\mtd\|_{L_2} & + \|\nabla_\mh^2 p^\mtd\|_{L_2} \leq C \|\nabla_\mh \partial_t v^\mtd\|_{L_2}
  + C \|\nabla_\mh v^\mtd\|_{L_4}^2 + C \|v^\mtd_\mh\|_{L_4} \|\nabla_\mh^2 v^\mtd\|_{L_4} \\
  & \leq C \|\nabla_\mh \partial_t v^\mtd\|_{L_2}
  + C \|\nabla_\mh v^\mtd\|_{L_2}^{\frac{3}{2}} \|\nabla_\mh^3 v^\mtd\|_{L_2}^{\frac{1}{2}}
  + C \|v^\mtd_\mh\|_{L_2}^{\frac{1}{2}} \|\nabla_\mh v^\mtd\|_{L_2}^{\frac{1}{2}} \|\nabla_\mh^2 v^\mtd\|_{L_2}^{\frac{1}{2}} \|\nabla_\mh^3 v^\mtd\|_{L_2}^{\frac{1}{2}} \\
  & \leq C \|\nabla_\mh \partial_t v^\mtd\|_{L_2} + \frac{1}{2} \|\nabla_\mh^3 v^\mtd\|_{L_2}
  + C \|\nabla_\mh v^\mtd\|_{L_2}^3 + C \|v^\mtd_0\|_{L_2}\|\nabla_\mh v^\mtd\|_{L_2} \|\nabla_\mh^2 v^\mtd\|_{L_2}.
\end{align*}
Combining the above estimate with \eqref{L2es2d-2}, \eqref{L2es2d-3} and \eqref{L2es2d-4-1} leads to
\begin{align}\label{L2es2d-4-2}
  t^{\frac{3}{2}}\|\nabla_\mh^3 v^\mtd(t)\|_{L_2} + t^{\frac{3}{2}} \|\nabla_\mh^2 p^\mtd(t)\|_{L_2}
  \leq C \|v^\mtd_0\|_{L_2}^2 e^{C(1+\|v^\mtd_0\|_{L_2}^4)}.
\end{align}
\vskip0.1cm

We then differentiate \eqref{eq:HNS-4} to get
\begin{equation}\label{eq:HNS-4}
\begin{split}
  \partial_{tt} (\nabla_\mh^2 v^\mtd) +  v^\mtd_\mh\cdot\nabla_\mh (\nabla_\mh^2\partial_t v^\mtd) -& \Delta_\mh (\nabla_\mh^2\partial_t v^\mtd) + \nabla(\nabla_\mh^2\partial_t p^\mtd)\\
  =& -2 \nabla_\mh v^\mtd_\mh \cdot\nabla_\mh^2 \partial_t v^\mtd - \nabla_\mh^2 v^\mtd_\mh \cdot\nabla_\mh \partial_t v^\mtd- \nabla_\mh^2 \partial_t v^\mtd_\mh \cdot\nabla_\mh v^\mtd \\
  & -2 \nabla_\mh \partial_t v^\mtd_\mh \cdot\nabla_\mh^2 v^\mtd - \partial_t v^\mtd_\mh \cdot \nabla_\mh^3 v^\mtd.
\end{split}
\end{equation}
Similarly as obtaining \eqref{eq:est1}, we get
\begin{align*}
  \frac{1}{2}\frac{\dd}{\dd t} \|\nabla_\mh^2 \partial_t v^\mtd\|_{L_2}^2 + \|\nabla_\mh^3 \partial_t v^\mtd\|_{L_2}^2
  \leq & 3 \|\nabla_\mh^2 \partial_t v^\mtd \|_{L_4}^2 \|\nabla_\mh v^\mtd \|_{L_2} + 2\|\partial_t v^\mtd\|_{L_4} \|\nabla_\mh^2 v^\mtd\|_{L_4} \|\nabla_\mh^3 \partial_t v^\mtd\|_{L_2} \\
  & + 2 \|\nabla_\mh \partial_t v^\mtd\|_{L_4} \|\nabla_\mh v^\mtd\|_{L_4} \|\nabla_\mh^3 \partial_t v^\mtd\|_{L_2} .
\end{align*}
We multiply both sides with $t^4$ and integrate on the time variable, and it follows that
\begin{align*}
  & \frac{1}{2} t^4 \|\nabla_\mh^2 \partial_t v^\mtd\|_{L_2}^2 + \int_0^t \tau^4 \|\nabla_\mh^3 \partial_\tau v^\mtd\|_{L_2}^2 \dd \tau \\
  \leq & 2\int_0^t \tau^3 \|\nabla_\mh^2 \partial_\tau v^\mtd\|_{L_2}^2 \dd \tau
  +3 \int_0^t \tau^4 \|\nabla_\mh^2 \partial_\tau v^\mtd\|_{L_4}^2 \|\nabla_\mh v^\mtd\|_{L_2} \dd \tau \\
  & + 2 \int_0^t \tau^4  \|\partial_\tau v^\mtd\|_{L_4} \|\nabla_\mh^2 v^\mtd\|_{L_4} \|\nabla_\mh^3 \partial_\tau v^\mtd\|_{L_2} \dd \tau
  +  2 \int_0^t \tau^4  \|\nabla_\mh \partial_\tau v^\mtd\|_{L_4} \|\nabla_\mh v^\mtd\|_{L_4} \|\nabla_\mh^2 \partial_\tau v^\mtd\|_{L_2} \dd \tau  \\
  \leq & C\|v^\mtd_0\|_{L_2}^2 e^{C(1+\|v^\mtd_0\|_{L_2}^4)}
  + C \int_0^t \tau^4 \|\nabla_\mh^2 \partial_\tau v^\mtd\|_{L_2} \|\nabla_\mh^3 \partial_\tau v^\mtd\|_{L_2} \|\nabla_\mh v^\mtd\|_{L_2} \dd \tau  \\
  & + C \int_0^t \tau^4  \|\partial_\tau v^\mtd\|_{L_2}^{\frac{1}{2}} \|\nabla_\mh \partial_\tau v^\mtd\|_{L_2}^{\frac{1}{2}} \|\nabla_\mh^2 v^\mtd\|_{L_2}^{\frac{1}{2}}
  \|\nabla_\mh^3 v^\mtd\|_{L_2}^{\frac{1}{2}} \|\nabla_\mh^3 \partial_\tau v^\mtd\|_{L_2} \dd \tau \\
  & + C \int_0^t \tau^4  \|\nabla_\mh\partial_\tau v^\mtd\|_{L_2}^{\frac{1}{2}} \|\nabla_\mh^2 \partial_\tau v^\mtd\|_{L_2}^{\frac{1}{2}} \|\nabla_\mh v^\mtd\|_{L_2}^{\frac{1}{2}}
  \|\nabla_\mh^2 v^\mtd\|_{L_2}^{\frac{1}{2}} \|\nabla_\mh^3 \partial_\tau v^\mtd\|_{L_2} \dd \tau \\
  \leq & C\|v^\mtd_0\|_{L_2}^2 e^{C(1+\|v^\mtd_0\|_{L_2}^4)}
  + \frac{1}{2}\int_0^t \tau^4 \|\nabla_\mh^3 \partial_\tau v^\mtd\|_{L_2}^2 \dd \tau
  + C \int_0^t \tau^4 \|\nabla_\mh^2 \partial_\tau v^\mtd\|_{L_2}^2\|\nabla_\mh v^\mtd\|_{L_2}^2 \dd \tau  \\
  & + C \int_0^t \tau^4  \|\partial_\tau v^\mtd\|_{L_2} \|\nabla_\mh \partial_\tau v^\mtd\|_{L_2} \|\nabla_\mh^2 v^\mtd\|_{L_2} \|\nabla_\mh^3 v^\mtd\|_{L_2} \dd \tau \\
  & + C \int_0^t \tau^4  \|\nabla_\mh\partial_\tau v^\mtd\|_{L_2} \|\nabla_\mh^2 \partial_\tau v^\mtd\|_{L_2} \|\nabla_\mh v^\mtd\|_{L_2} \|\nabla_\mh^2 v^\mtd\|_{L_2} \dd \tau  .
\end{align*}
Noting that
\begin{align*}
   & \int_0^t \tau^4  \|\partial_\tau v^\mtd\|_{L_2} \|\nabla_\mh \partial_\tau v^\mtd\|_{L_2} \|\nabla_\mh^2 v^\mtd\|_{L_2} \|\nabla_\mh^3 v^\mtd\|_{L_2} \dd \tau \\
   \leq & \int_0^t \tau^6 \|\partial_\tau v^\mtd\|_{L_2}^2 \|\nabla_\mh^2 v^\mtd\|_{L_2}^2 \|\nabla_\mh^3 v^\mtd\|_{L_2}^2  \dd \tau
  + \int_0^t \tau^2\|\nabla_\mh \partial_\tau v^\mtd\|_{L_2}^2 \dd \tau \\
  \leq & \sup_{\tau \in[0,t]} \big( \tau^2\|\partial_\tau v^\mtd\|_{L_2}^2\big) \sup_{\tau \in[0,t]} \big( \tau^3\|\nabla_\mh^3 v^\mtd\|_{L_2}^2\big)
  \int_0^t \tau \|\nabla_\mh^2 v^\mtd\|_{L_2}^2\dd\tau + C  \|v^\mtd_0\|_{L_2}^2 e^{C(1+\|v^\mtd_0\|_{L_2}^4)} \\
  \leq & C \|v^\mtd_0\|_{L_2}^2 \big(\|v^\mtd_0\|_{L_2}^4 + 1\big) e^{C(1+\|v^\mtd_0\|_{L_2}^4)}
  \leq C \|v^\mtd_0\|_{L_2}^2 e^{C(1+\|v^\mtd_0\|_{L_2}^4)} ,
\end{align*}
and
\begin{align*}
   & \int_0^t \tau^4  \|\nabla_\mh\partial_\tau v^\mtd\|_{L_2} \|\nabla_\mh^2 \partial_\tau v^\mtd\|_{L_2} \|\nabla_\mh v^\mtd\|_{L_2} \|\nabla_\mh^2 v^\mtd\|_{L_2} \dd \tau \\
   \leq & \sup_{\tau \in[0,t]} \big( \tau^2\|\nabla_\mh^2 v^\mtd\|_{L_2}^2\big) \sup_{\tau \in[0,t]} \big( \tau^3\|\nabla_\mh \partial_\tau v^\mtd\|_{L_2}^2\big)
  \int_0^t \|\nabla_\mh v^\mtd\|_{L_2}^2\dd\tau
  + \int_0^t \tau^3\|\nabla_\mh^2 \partial_\tau v^\mtd\|_{L_2}^2 \dd \tau \\
  \leq & C \|v^\mtd_0\|_{L_2}^2 e^{C(1+\|v^\mtd_0\|_{L_2}^4)} ,
\end{align*}
Gr\"onwall's inequality directly yields
\begin{equation}
  t^4 \|\nabla_\mh^2 \partial_t v^\mtd\|_{L_2}^2 + \int_0^t \tau^4 \|\nabla_\mh^3 \partial_\tau v^\mtd\|_{L_2}^2 \dd \tau \leq C \|v^\mtd_0\|_{L_2}^2 e^{C(1+\|v^\mtd_0\|_{L_2}^4)},
\end{equation}
which combined with \eqref{L2es2d-4-2} concludes the desired estimate \eqref{L2es2d-4}.
\end{proof}

Next, we show the refined regularity estimate for the solution $v^\mtd$ other than \eqref{v2d-es1}.

\begin{lemma}\label{lem:2DNS-2.2}
  Let $v^\mtd_0\in L_2\cap \dot B^{4-2/p}_{p,p}(\R^2)$ with $p>3$, then the solution $v^\mtd =(v^\mtd_\mh, v^\mtd_3)$ of the 2D (HNS) system \eqref{HNS} satisfies
\begin{equation}\label{v2d-es2}
\begin{split}
  \sup_{t\geq 0}\|v^\mtd \|_{\dot B^{4-2/p}_{p,p}(\R^2)} + \|\partial_t v^\mtd, \nabla^2_\mh v^\mtd\|_{L_p(\R^+; W^2_p(\R^2))}
  \leq\,   C \Big( \|v^\mtd_0\|_{L_2\cap \dot B^{4-2/p}_{p,p}(\R^2)}^6 + 1\Big) e^{C(1+\|v^\mtd_0\|_{L_2}^4)}.
\end{split}
\end{equation}
\end{lemma}

\begin{proof}[Proof of Lemma \ref{lem:2DNS-2.2}]

From \eqref{HNS}, we see that $\nabla_\mh^2 v^\mtd$ satisfies
\begin{equation}
  \partial_t (\nabla^2_\mh v^\mtd) -\Delta_\mh (\nabla^2_\mh v^\mtd) + \nabla (\nabla_\mh^2 p^\mtd)
  = - (\nabla^2_\mh v^\mtd_\mh)\cdot\nabla_\mh v^\mtd - 2 \nabla_\mh v^\mtd_\mh \cdot \nabla_\mh^2 v^\mtd
  - v^\mtd_\mh\cdot \nabla_\mh(\nabla_\mh^2 v^\mtd),
\end{equation}
then we have
\begin{equation*}
\begin{split}
  & \sup_{t\geq 0} \|\nabla^2_\mh v^\mtd\|_{\dot B^{2-2/p}_{p,p}} + \|(\partial_t \nabla^2_\mh v^\mtd , \nabla^4_\mh v^\mtd )\|_{L_p(\R^2\times \R^+)} \\
  \leq & C \left(\| v^\mtd_\mh \cdot \nabla^3_\mh v^\mtd \|_{L_p(\R^2\times \R^+)} + \|\nabla_\mh v^\mtd_\mh \cdot \nabla_\mh^2 v^\mtd \|_{L_p(\R^2\times \R^+)}
  + \|\nabla^2_\mh v^\mtd_\mh \cdot \nabla_\mh v^\mtd\|_{L_p(\R^2\times \R^+)}+ \|v^\mtd_0 \|_{\dot B^{4-2/p}_{p,p}}\right).
\end{split}
\end{equation*}
By arguing as \eqref{fact5}, and using \eqref{bomgLpes2}, \eqref{nab-v3Lp}, \eqref{fact3}, we find that
\begin{align*}
  \left(\int_0^\infty \|v^\mtd_\mh\cdot\nabla_\mh^3 v^\mtd \|_{L_p(\R^2)}^p \dd t\right)^{1/p} & \leq  \left(\int_0^\infty \|v^\mtd_\mh\|_{L_\infty}^p
  \|\nabla^3_\mh v^\mtd\|_{L_p(\R^2)}^p \dd t\right)^{1/p} \\
  & \leq C\left(\int_0^\infty \|v^\mtd_\mh\|_{L_\infty}^p \|\nabla_\mh v^\mtd\|_{L_p}^{\frac{p}{3}}\|\nabla^4_\mh v^\mtd \|_{L_p(\R^2)}^{\frac{2p}{3}} \dd t\right)^{1/p} \\
  & \leq \frac{1}{4C} \|\nabla^4_\mh v^\mtd \|_{L_p(\R^2\times \R^+)} + C \|v^\mtd_\mh\|_{L_{3p}(\R^+; L_\infty)}^3 \|\nabla_\mh v^\mtd \|_{L_\infty(\R^+; L_p)} \\
  & \leq \frac{1}{4C} \|\nabla^4_\mh v^\mtd \|_{L_p(\R^2\times \R^+)} + C  \Big( \|v^\mtd_0\|_{L_2\cap \dot B^{3-2/p}_{p,p}}^5 +1\Big)e^{C(1+\|v^\mtd_0\|_{L_2}^4)} ,
\end{align*}
and
\begin{equation*}
\begin{split}
  & \left(\int_0^\infty \|\nabla_\mh v^\mtd_\mh \cdot \nabla^2_\mh v^\mtd \|_{L_p(\R^2)}^p \dd t\right)^{1/p} \leq
  \left(\int_0^\infty \|\nabla_\mh v^\mtd_\mh\|_{L_p}^p \|\nabla^2_\mh v^\mtd \|_{L_\infty}^p \dd t\right)^{1/p} \\
  & \leq C\left(\int_0^\infty \|\nabla_\mh v^\mtd_\mh\|_{L_p}^p
  \|\nabla_\mh v^\mtd \|_{L_p}^{\frac{2p-2}{3}} \|\nabla^4_\mh v^\mtd \|_{L_p}^{\frac{p+2}{3}} \dd t\right)^{1/p} \\
  & \leq  \frac{1}{4C} \|\nabla^4_\mh v^\mtd \|_{L_p(\R^2\times \R^+)} + C \|\nabla_\mh v^\mtd\|_{L_\infty(\R^+; L_p)}
  \bigg( \|\nabla_\mh v^\mtd_\mh\|_{L_\infty(0,1; L_p)}^{\frac{3p^2}{2p-2}}
  + \int_1^\infty \|\nabla_\mh v^\mtd_\mh\|_{L_p}^{\frac{3p^2}{2p-2}} \dd t\bigg)^{1/p} \\
  & \leq \frac{1}{4C} \|\nabla^4_\mh v^\mtd \|_{L_p(\R^2\times \R^+)} + C \Big( \|v^\mtd_0\|_{L_2\cap \dot B^{3-2/p}_{p,p}}^5 +1\Big)e^{C(1+\|v^\mtd_0\|_{L_2}^4)},
\end{split}
\end{equation*}
where in the last line we have used the following estimate
\begin{equation*}
  \Big(\int_1^\infty \|\nabla_\mh v^\mtd_\mh\|_{L_p}^{\frac{3p^2}{2p-2}} \dd t \Big)^{1/p}\leq C \|v^\mtd_0\|_{L_2}^{\frac{3p}{2p-2}} e^{C(1+\|v^\mtd_0\|_{L_2}^4)}
  \Big(\int_1^\infty t^{-\frac{3p}{2}} \dd t\Big)^{1/p} \leq  C \|v^\mtd_0\|_{L_2}^3 e^{C(1+\|v^\mtd_0\|_{L_2}^4)}.
\end{equation*}
From \eqref{estma0} and the continuous embedding $L_2\cap \dot B^{2-2/p}_{p,p}(\R^2)\hookrightarrow \dot W^1_p(\R^2)$, we also infer
\begin{equation*}
\begin{split}
  \|\nabla_\mh^2 v^\mtd_\mh \cdot \nabla_\mh v^\mtd \|_{L_p(\R^2\times \R^+)}  
  & \leq \|v^\mtd\|_{L_\infty(\R^+; L_2\cap \dot B^{2-2/p}_{p,p})}  \|\nabla_\mh^2 v^\mtd_\mh\|_{L_p(\R^2\times\R^+)} \\
  & \leq C  \Big( \|v^\mtd_0\|_{L_2\cap \dot B^{3-2/p}_{p,p}}^6  + 1\Big)e^{C(1+\|v^\mtd_0\|_{L_2}^4)} .
\end{split}
\end{equation*}
Hence, gathering the above estimates leads to
\begin{equation}\label{v2d3-Es3}
\begin{split}
  \sup_{t\geq 0} \|\nabla^2_\mh v^\mtd\|_{\dot B^{2-2/p}_{p,p}} + \|(\partial_t \nabla^2_\mh v^\mtd , \nabla^4_\mh v^\mtd )\|_{L_p(\R^2\times \R^+)}
  \leq  C \Big( \|v^\mtd_0\|_{L_2\cap \dot B^{4-2/p}_{p,p}}^6  + 1\Big)e^{C(1+\|v^\mtd_0\|_{L_2}^4)}.
\end{split}
\end{equation}

Therefore, \eqref{v2d3-Es3} and \eqref{v2d-es1} combined with Calder\'on-Zygmund's theorem yield \eqref{v2d-es2}, as desired.
\end{proof}

\section*{Acknowledgments}

The first author (P.B.M.) has been partly supported by National Science Centre grant 2014/14/M /ST1/00108 (Harmonia).
The second author (L.X.) has been partly supported by National Natural Science Foundation of China (grant Nos. 11401027, 11671039 and 11771043).
The third author (X.Z.) has been partly supported by National Natural Science Foundation of China (grant Nos. 11671047, 11771423 and 11871087).

\end{document}